\documentclass[leqno,english]{amsart}
\usepackage{amsmath,amssymb,amsthm,enumerate,esint,mathtools,xparse}
\usepackage[svgnames]{xcolor}
  \usepackage{hyperref}
 \usepackage{empheq}

  \usepackage{bbm}
  \DeclareMathOperator{\csch}{csch}

 \usepackage{mathrsfs} 
\usepackage{mathabx}

\numberwithin{equation}{section}
\newtheorem{definition}{Definition}[section]
\newtheorem{theorem}{Theorem}[section]
\newtheorem{corollary}[theorem]{Corollary}
\newtheorem{lemma}[theorem]{Lemma}
\newtheorem{proposition}[theorem]{Proposition}

\newtheorem{remark}[theorem]{Remark}

\newcommand{\bke}[1]{\left ( #1 \right )}
\newcommand{\bkt}[1]{\left [ #1 \right ]}
\newcommand{\bket}[1]{\left \{ #1 \right \}}
\newcommand{\norm}[1]{\left \| #1 \right \|}
\newcommand{\bka}[1]{{\langle #1 \rangle}}
\newcommand{\abs}[1]{\left | #1 \right |}

\newcommand\al{\alpha}
\newcommand\be{\beta}
\newcommand\ga{\gamma}
\newcommand\de{\delta}

\newcommand\ve{\varepsilon}

\newcommand\ze{\zeta}
\renewcommand\th{\theta}
\newcommand\ka{\kappa}
\newcommand\la{\lambda}
\newcommand\si{\sigma}

\newcommand\om{\omega}
\newcommand\Ga{\Gamma}
\newcommand\De{\Delta}
\newcommand\Th{\Theta}
\newcommand\La{\Lambda}

\newcommand\Om{\Omega}
\newcommand\Rg{\mathscr{R}_g}

\newcommand{\R}{\mathbb{R}}
\newcommand{\CC}{\mathbb{C}}

\renewcommand{\Re} {\mathop{\mathrm{Re}}\nolimits}

\renewcommand{\div}{\mathop{\rm div}\nolimits}

\newcommand{\dist} {\mathop{\mathrm{dist}}\nolimits}

\newcommand{\pd}{\partial}
\newcommand{\nb}{\nabla}

\newcommand{\wt}[1]{\widetilde {#1}}

\renewcommand{\bar}[1]{\overline{#1}}
\newcommand{\lec}{{\ \lesssim \ }}

\newcommand{\EQ}[1]{\begin{equation}\begin{split} #1 \end{split}\end{equation}}
\newcommand{\EQN}[1]{\begin{equation*}\begin{split} #1 \end{split}\end{equation*}}
\newcommand{\EN}[1]{\begin{enumerate} #1 \end{enumerate}}

\DeclarePairedDelimiter{\oldnormaux}{\bracevert}{\bracevert}

\NewDocumentCommand{\oldnorm}{som}{
  \IfBooleanTF{#1}
    {\oldnormaux*{#3}}
    {\IfNoValueTF{#2}
       {\oldnormaux*{\vphantom{dq}#3}}
       {\oldnormaux[#2]{#3}}
    }
}

\begin{document}
\title[Exponential stability of the spherical bubble]{Free boundary problem for a gas bubble in a liquid, and exponential stability of the manifold of spherically symmetric equilibria}

\author[C.-C. Lai]{Chen-Chih Lai}
\address{\noindent
Department of Mathematics,
Columbia University , New York, NY, 10027, USA}
\email{cl4205@columbia.edu}

\author[M. I. Weinstein]{Michael I. Weinstein}
\address{\noindent
Department of Applied Physics and Applied Mathematics and Department of Mathematics,
Columbia University , New York, NY, 10027, USA}
\email{miw2103@columbia.edu}

\begin{abstract}
We consider the dynamics of a gas bubble immersed in an incompressible fluid of fixed temperature, and focus on the relaxation of an expanding and contracting spherically symmetric bubble due to thermal effects. 
We study two models, both systems of PDEs
 with an evolving free boundary: the full mathematical model as well as an approximate model, arising for example  
 in the study of sonoluminescence.
For fixed physical parameters (surface tension of the gas--liquid interface, liquid viscosity, thermal conductivity of the gas, 
  etc.), both models share a family of spherically symmetric equilibria, smoothly parametrized by the mass of the gas bubble.
Our main result concerns the approximate model. We prove the nonlinear asymptotic stability of the manifold of equilibria with respect to small spherically symmetric perturbations. The rate of convergence is exponential in time. To prove this result we first prove a weak form of nonlinear asymptotic stability --with no explicit rate of time-decay-- using
the energy dissipation law, and then, via a center manifold analysis,  bootstrap the weak time-decay to exponential time-decay.

We also study the uniqueness of the family of spherically symmetric equilibria within each model.
The family of spherically symmetric equilibria captures all regular spherically symmetric equilibria of the approximate system. 
However within the full model, this family is embedded in a larger family of spherically symmetric solutions.
For the approximate system, we prove that all equilibrium bubbles are spherically symmetric, by an application of  
Alexandrov's theorem on closed surfaces of constant mean curvature.

\end{abstract}
\maketitle

\tableofcontents

\newpage

\section{Introduction}

This paper considers a free boundary problem  for the dynamics of a gas bubble immersed in a liquid. 
The bubble occupies a bounded and simply connected subset of $\mathbb{R}^3$, denoted $\Om(t)$.
The gas within the bubble is a compressible fluid characterized by its density, velocity, pressure and temperature, as well as constitutive relations relating these variables and the specific entropy. 
The surrounding liquid is assumed to be incompressible and is described by its velocity, pressure and temperature. 
The gas inside the bubble and liquid outside the bubble are coupled at the boundary by kinematic and stress-balance equations.  Section \ref{sec:formulation} contains the full mathematical formulation of the liquid / gas model.
We are interested in the long time evolution of the coupled bubble / liquid system for initial conditions which are near a spherically symmetric equilibrium. 

 Energy dissipation plays an important role in the  bubble / fluid dynamics.
Generally, there are three mechanisms for energy dissipation of bubbles \cite[p. 175]{Leighton-book2012}: 
\emph{radiation damping} (of sound waves toward infinity for the compressible fluid case), \emph{thermal damping}  (transfer of energy from the gas into the fluid via thermal conduction) and \emph{viscous damping}.
We consider an approximation to the full liquid / gas dynamics
 in which \underline{thermal damping} is the dominant dissipation mechanism; viscous damping is comparatively negligible, and there is no radiation damping due to sound wave emission because the liquid surrounding the gas bubble is assumed incompressible. 

The model we study is an asymptotic model introduced by Prosperetti \cite{Prosperetti-JFM1991}. 
In \cite{Prosperetti-JFM1991} the linearized problem was studied by means of Laplace transform, and under various simplifying approximations, linear  asymptotic stability of spherical equilibria is argued.
When the liquid is inviscid on the liquid--gas interface, the Prosperetti model coincides with the approximate model derived by Biro and Vel\'azquez in \cite{bv-SIMA2000} based on the parameter regimes of sonoluminescence experiments \cite{barber1991observation, barber1992light, barber1994sensitivity}.
We present the asymptotic model of \cite{Prosperetti-JFM1991,bv-SIMA2000} in Section \ref{sec-asymp-model}.
In this model, the gas pressure,  gas density and gas temperature all vary and are related via the ideal gas equation of state. 
Solutions which are spherically symmetric are determined by a reduced free boundary problem \eqref{eq-bv-3.10prime}--\eqref{eq-bv-3.16prime}: a quasilinear parabolic PDE (nonlinear diffusion) for the density $\rho_g(r,t)$ in the gas bubble region, $0\le r\le R(t)$,  coupled  to a second-order nonlinear ODE  for the  bubble radius, $R(t)$; see Section \ref{reduce}. Local-in-time well-posedness in H\"older spaces was proved for the initial value problem  in 
 \cite{bv-SIMA2000}.
 
\subsection{Main results}

\subsubsection{Exponential stability of the spherical equilibrium bubble} 
The system \eqref{eq-bv-3.10prime}--\eqref{eq-bv-3.16prime}, in which spherical symmetry is imposed, has an equilibrium solution for any prescribed gas bubble mass.  In \cite{bv-SIMA2000} these spherically symmetric equilibria were proved to be {\it Lyapunov stable}. That is, a small spherically symmetric perturbation of the spherical equilibrium bubble of the same mass will evolve, under the dynamics \eqref{eq-bv-3.10prime}--\eqref{eq-bv-3.16prime}, as a spherically symmetric solution which remains near the equilibrium bubble for all $t>0$.

The collection of all such spherical equilibrium bubbles forms a smooth manifold of spherically symmetric equilibria parameterized by the bubble mass; see Section \ref{sec-equilibria}.  Our main result, Theorem \ref{thm-nonlinear-exp-decay},  is the nonlinear asymptotic stability, with an exponential rate of time-decay, of the manifold of spherically symmetric equilibria with respect to small spherically symmetric perturbations: 

\begin{quote}
{\it
For sufficiently small spherically symmetric initial data perturbations of any given spherical equilibrium bubble, the evolving bubble shape and surrounding liquid relax, as time advances, toward a spherical equilibrium bubble exponentially fast. 
}
\end{quote}
\noindent
The equilibrium bubble evolving from the perturbed initial data are typically different from the given equilibrium bubble unless the perturbed initial data has the same mass as the given equilibrium bubble.
More precisely, the equilibrium bubble radius and uniform gas density, which emerge as $t\to\infty$, are determined by the bubble mass of the initial data.
These results are stated in detail in Theorem \ref{thm-nonlinear-exp-decay}.
The proof of Theorem \ref{thm-nonlinear-exp-decay} consists of the following two major steps.

\medskip
\noindent{\bf Step 1.}
Building on the well-posedness and Lyapunov stability work of   Biro--Vel\'azquez \cite{bv-SIMA2000}, we first prove a weak form of asymptotic stability --with no rate of time-decay-- using the energy dissipation law.
 Detailed statements are presented in Proposition \ref{thm-asystab-1} (asymptotic stability of a fixed equilibrium relative to small mass preserving perturbations) and Theorem \ref{thm-asystab} (asymptotic stability of the manifold of equilibria relative to arbitrary small perturbations). Theorem \ref{thm-asystab} is a consequence of Proposition \ref{thm-asystab-1} and the continuity of functionals.

\medskip
\noindent{\bf Step 2.} 
We bootstrap the weak asymptotic stability of Step 1 to obtain quantitative exponential asymptotic stability,
 Theorem \ref{thm-nonlinear-exp-decay}.
In particular, we show that the manifold of spherically symmetric equilibria is an attracting center manifold in Section \ref{sec-main-proof}.
This requires analysis of  spectrum of the linearized operator (Proposition \ref{prop-spectrum}), 
a proof of an exponential time-decay estimate on a codimension one subspace associated with the manifold of equilibria (Proposition \ref{prop-stab-center-mfd}), and estimates of nonlinear terms (Proposition \ref{prop-nonlinear-est}).
To implement this program we must extend the standard  center manifold analysis  to a class of fully nonlinear autonomous dynamical system equipped with {\it a priori} estimates coming from Step 1; see Appendix \ref{sec-center-manifold}.

\subsubsection{Symmetry of all equilibrium bubbles of the asymptotic model.}
We also study the general version of \eqref{eq-bv-3.10prime}--\eqref{eq-bv-3.16prime}, not constrained by the assumption of spherical symmetry, {\it i.e.} the asymptotic model \eqref{eq1.1simplified}--\eqref{eq1.3simplified}.
We show that all equilibrium bubbles of \eqref{eq1.1simplified}--\eqref{eq1.3simplified} must be spherically symmetric.
The detailed result is presented in Part (1) of Proposition \ref{prop-equilib}.

\subsection{General context of our work and  relation to other physical models}

The dynamics of gas bubbles immersed in a liquid play an important role in fundamental and applied physics and in engineering applications.
Examples  include 
underwater explosion bubbles \cite{KK-JAP1956}, 
bubble jetting \cite{KS-JAP1944, PC-JFM1971},
seismic wave-producing bubbles in magma \cite{RG-JGRSE1999}, 
bubbles at the ocean surface \cite{Longuet-JFM1989},
sonochemistry \cite{Suslick-Science1990}, 
sonoluminescence \cite{BLD-PRL1995, PR-PRL1998}.
Engineering and industrial examples include
microfluidics \cite{XA-PF2007}, 
ultrasonic cavitation cleaning \cite{OAIDVDL-BJ2006, SHLC-JAP2004},
and
applications of ultrasound cavitation bubbles such as
medical imaging \cite{FPB-ARBE2007}, 
shock wave lithotripsy (ESWL) \cite{KFTKSCSZ-JFM2007, CIS-JFM2008, IFS-JFM2008, LTBW-PRSL2013, LFCMRHDW-Ultra2008},
tissue ablation \cite{RHIWFC-TJU2006, CR-ARFM2008, CLWB-JFM2013},
oncology and cardiology \cite{LK-UQ2006}.
For a discussion of these and other applications of bubble dynamics, see the excellent review articles \cite{Leighton-IJMPB2004, Prosperetti-PF2004, Liu-thesis2018} and the book \cite{Leighton-book2012}, and references cited therein.

The study of bubble dynamics was initiated in 1917 by Lord Rayleigh \cite{Rayleigh-London1917} during his work with Royal Navy to investigate cavitation damage on ship propellers.
He derived an equation for the radial oscillations of a spherically symmetric gas bubble in an incompressible, inviscid liquid with surface tension and examined the pressure prediction during the collapse of a spherical bubble.
Over several decades his work was refined and developed by numerous researchers.
The Rayleigh-Plesset equation \cite{Plesset-JAM1949} is a second-order nonlinear ODE for the bubble radius. 
J.B. Keller and collaborators \cite{KK-JAP1956,KE-1972,KM-JASA1980}
  incorporated the effect of liquid compressibility on the bubble dynamics and incorporated sound radiation from the oscillating bubble.
These models have been  extensively used in modeling and studied by asymptotic analytical and numerical and methods; see, for example, \cite{PP-ANFM1977, TS-BJSME1977, Vokurka-AAUA1986, FL-ARFM1997, Prosperetti-PF2004, ZL-JH2012, VanGorder-JFM2016, FDFADL-EJMBF2018, SW-JFM2018, FOSY-JMAA2015, OS-MFDPF-2016} and references therein.
 These models all impose isothermal or adiabatic approximations in which the gas obeys polytropic equation of state
 (pressure $\times$ a power of the volume is equal to a constant).
Over the course of bubble oscillations, there are periods where the isothermal assumption and hence an adiabatic pressure volume law is valid (expansion), and periods over which the isothermal approximation is violated; strong compression, as in sonoluminescence experiments. 
Numerous works compare the two approximations and find a balance between them, {\it e.g.} \cite{Prosperetti-JFM1991,bv-SIMA2000,Goldsztein-SAM2004}.

The model we study \eqref{eq1.1simplified}--\eqref{eq1.3simplified}, or more specifically its spherically symmetric reduction, \eqref{eq-bv-3.10prime}--\eqref{eq-bv-3.16prime},  was introduced by Prosperetti \cite{Prosperetti-JFM1991}, as an asymptotic approximation of the full 
 liquid / gas-bubble system \eqref{eq-1.1}--\eqref{eq-1.3}, in which the gas pressure, density and temperature are related by an ideal gas law.
Neither an isothermal nor adiabatic assumption is made.
Over the years, researchers have extensively investigated the model using various approximation techniques. For instance, in \cite{ZP-JFM2020}, quadratic and biquadratic approximations were employed to transform the full PDE model into a simplified ODE model, leading to a significant reduction in computational costs during simulation. 
The model studied by Biro and Vel\'azquez in \cite{bv-SIMA2000}  reduces to that of Prosperetti \cite{Prosperetti-JFM1991} when the liquid viscosity, $\mu_l$, is assumed to be zero on the bubble interface.
The article \cite{bv-SIMA2000} studies, in the spherically symmetric setting:
 (i) local well-posedness in the space-time H\"older space, (ii) global well-posedness for initial data near a spherically symmetric equilibrium, and (iii) Lyapunov stability of the equilibrium relative to small mass-preserving perturbations.
At the heart of their stability result is  an energy dissipation identity and  a coercivity  estimate (lower bound) on the energy around the equilibrium, showing that spherically symmetric equilibria are constrained local minimizers.

Our work extends the result of \cite{bv-SIMA2000} in the following directions.

\EN{

\item We consider a more general model, {\it i.e.}, the Prosperetti model \cite{Prosperetti-JFM1991}, which incorporates liquid viscosity, $\mu_l\ge0$, on the liquid-bubble interface.

\item We construct a manifold of spherically symmetric equilibria parametrized by the mass of gas bubble (Proposition \ref{prop-equilib-original}).

\item We show that equilibrium bubbles of the approximate model are spherically symmetric provided $\mu_l\neq0$ (Part (1) of Proposition \ref{prop-equilib}).

\item We extend the conditional Lyapunov stability  result of \cite{bv-SIMA2000} to  Lyapunov stability relative to \underline{arbitrary} spherically symmetric perturbations which are small.

\item Most significantly, we prove asymptotic stability of the manifold of equilibria (Theorem \ref{thm-asystab}) with an exponential rate of convergence (Theorem \ref{thm-nonlinear-exp-decay}).
Our analysis demonstrates that the 
equilibrium gas-bubble, which emerges as $t\to\infty$, 
is determined by the initial data (and prescribed parameters of the model); it is the equilibrium bubble on the manifold of spherical states
having the same mass as the initial (perturbed) bubble data.

\item We also study the persistence (structural stability) of the asymptotic stability result, Theorem \ref{thm-asystab}, under a far-field time-dependent pressure, $p_\infty(t)$;
 see Corollary \ref{cor:rem-lyapunov-general-pinf} and Corollary \ref{cor:rem-asystab-general-pinf}.
 
}

Finally, we remark that 
there is an analogy of the present study with the 
asymptotic stability
of coherent structures that the equilibrium state is determined by initial data in other nonlinear diffusive dynamical systems,
{\it e.g.} 
volume-preserving geometric flows in the Mullins-Sekerka model \cite{ES-JDE1998},
smoothed out shock profiles in viscous perturbations of hyperbolic conservation laws \cite{Goodman:86}, 
traveling front solutions in nonlinear reaction diffusion dynamics \cite{McKean-CPAM1983, McKean-CPAM1984, Bramson-MAMS1983}, and 
spatial uniform equilibrium in two dimensional chemotaxis-fluid model \cite{Winkler-ARMA2014}.

 \subsection{Some future directions and open problems in the context of the current and closely related models}
{\ }\\

1. \emph{Time-periodically expanding and contracting bubble oscillations.}
The far-field liquid pressure $p_\infty(t)$, is an external forcing term in our free boundary problem. In this present work, it  is prescribed to be either the constant ($p_\infty(t)\equiv p_{\infty,*}$) or such that $p_\infty(t)-p_{\infty,*}$ is small and decaying to zero sufficiently rapidly as $t\to\infty$. 
It is also of interest to study the bubble dynamics when the far-field pressure $p_\infty(t)$ is time-periodic, for example of the form $p_\infty(t) = p_{\infty,*}+A\cos(\om t)$, corresponding to far-field periodic acoustic forcing as in physical experiments \cite{barber1991observation, barber1992light, barber1994sensitivity}.
In a forthcoming paper \cite{LW-vbas-linear}, we prove that for sufficiently small forcing amplitude, $A$, there exists a unique asymptotically stable $2\pi/\om$-time periodic spherically symmetric pulsating bubble solution.

2. \emph{Uniqueness of the spherically symmetric equilibria.} 
Spherically symmetric equilibria of the asymptotic model \eqref{eq1.1simplified}--\eqref{eq1.3simplified} are uniquely characterized in Part (2) of Proposition \ref{prop-equilib}.
In the context of the general evolution for the asymptotic model, any equilibrium bubble is necessarily spherical (Part (1) of Proposition \ref{prop-equilib}).
However, there exist non-trivial (rotational) equilibrium gas flows inside the equilibrium spherical bubble (Remark \ref{rmk-nonunique-gas}).
In other words, the spherically symmetric equilibria are not unique within the asymptotic model.
Under what circumstances is the family spherically symmetric equilibria are unique?
Certainly, the above rotational equilibrium gas flows are ruled out if only seek gas flows which are irrotational.
But are the spherically symmetric equilibria are unique within another closely related model?
We expect that adding a gas viscosity term in the stress balance equation \eqref{eq1.3simplified-b} can help us exclude the case of non-trivial equilibrium gas flow in an equilibrium spherical bubble.

3. \emph{Nonspherically symmetric dynamics.} 
 {\it Are spherically symmetric bubbles stable against small perturbations, unconstrained by symmetry?}
The main asymptotic stability result of the present paper requires spherically symmetric perturbations which are small.
It is then natural to ask: 
{\it Is the manifold of spherically symmetric equilibria 
asymptotically stable relative to small arbitrary (non-spherically symmetric) perturbations in the approximate system \eqref{eq1.1simplified}--\eqref{eq1.3simplified} (or in the full system \eqref{eq-1.1}--\eqref{eq-1.3})?}
We expect that surface tension plays an important role in rounding out bubbles during the evolution.

A related question was studied in a model of a spherical \emph{polytropic} gas bubble in a \emph{compressible} and \emph{non-viscous} liquid \cite{SW-SIMA2011, CTW-SIMA2013}. In this case, the damping mechanism  is acoustic radiation of waves to spatial infinity, rather than thermal diffusion. In \cite{SW-SIMA2011} it is proved that 
the spherically symmetric bubble is linearly asymptotically stable relative to general (not necessarily spherical) perturbations. In the weakly compressible regime, the sharp exponential decay rate of perturbations was determined in \cite{SW-SIMA2011,CTW-SIMA2013} and is proved to be governed by very high angular momentum {\it shape mode}
deformations of the bubble. The energy of these shape modes is transferred  very slowly to the surrounding compressible liquid and is radiated to infinity via acoustic waves.

\medskip

Finally, we list some related broader open questions:
\begin{enumerate}
\item \emph{Well-posedness} of the full liquid / gas model \eqref{eq-1.1}--\eqref{eq-1.3} for general data appears to be open.
\item
\emph{Nonuniqueness of spherical equilibria of the full liquid / gas model.}
Even the classification of equilibrium solutions of \eqref{eq-1.1}--\eqref{eq-1.3} appears to be non-trivial. It is not known, for example, whether
there are non-spherical equilibria. And, 
within the class of spherically symmetric equilibria of \eqref{eq-1.1}--\eqref{eq-1.3} there are solutions  with spatially \underline{non-uniform} temperature profiles (Remark \ref{rem-prop-equilib}). 
Some of these equilibria  have a singularity in the gas temperature at the origin. 
It would be an interesting and challenging mathematical problem  to investigate the dynamics for perturbations of such equilibria and to see whether they may participate in singularity formation for some classes of smooth initial conditions.

\item \emph{Radiation condition for the full liquid / gas model.}
The radiation condition \eqref{eq-radiation-condition} for the liquid temperature is required in proving the uniqueness of the spherical equilibrium solutions in \eqref{eq-equilibrium} of a specified mass.
\emph{Can one ensure that this radiation condition holds for the time-evolution  if it is imposed on the initial data?}
The proof would require an a priori regularity of solutions to a free boundary problem of parabolic equations with time-dependent boundary condition.

\item \emph{Global energy minimizer.} 
For the asymptotic model we have shown that the equilibrium $(\rho_*,R_*,\dot R_*=0)$ is a conditional local energy minimizer of the total energy $\mathcal{E}_{\rm total}$ (Definition \ref{def:total-en}), constrained to fixed mass.
Our analysis relies on the Taylor expansion of the energy near an equilibrium.
{\it Is the equilibrium bubble solution of mass $M$   a global minimizer of $\mathcal{E}_{\rm total}$
relative to arbitrarily large spherically symmetric deformations of mass $M$?}

\end{enumerate}

 \subsection{Notation and conventions}

\begin{enumerate}
\item $B_R=\{x\in\mathbb{R}^3: |x|<R\}$
\item For a function $f(r)$ defined for $0<r<R$, we set $\overline f(x) = f(|x|)$ for $x\in B_R$ and denote
\[
\norm{f}_{C^{2+2\al}_r} 
= \norm{\overline f}_{C^{2+2\al}_x} 
= \max_{|\be|\le 2} \sup_{x\in B_R} |D^\be \overline f(x)| + \sup_{\substack{x\neq y\\x,y\in B_R}} \frac{| D^2\overline f(x) - D^2\overline f(y) |}{|x-y|^{2\al}}.
\]
\item Denote $\nb_r f = \nb_x\overline f$ and $\De_r f = \De_x\overline f = \frac1{r^2}\pd_r(r^2\pd_rf)$ the radial part of the Laplace operator in $\mathbb R^3$.
\item For a state variable, such as the density $\rho$, if it corresponds to value of a constant equilibrium solution, then we denote it 
by $\rho_*$, and similarly for the values of other equilibrium state variables.
\item If ${\bf u}, {\bf v}$ and ${\bf w}$ are vector fields, then $ {\bf u}\cdot\nabla{\bf v}\cdot {\bf w} =
\left[ \left({\bf u}\cdot\nabla\right){\bf v}\right] \cdot {\bf w} $.
\item If $A=(A_{ij})$ and $B=(B_{ij})$, the $A:B = {\rm tr}(AB^\top) = \sum_{i,j} A_{ij}B_{ij}.$
\item $\nabla {\bf u} = (\pd_i u_j )$; $|\nabla{\bf u}|^2 \equiv \sum_{i,j} \pd_i u_j \pd_i u_j = {\rm tr}\left(\nabla {\bf u} (\nabla {\bf u})^\top\right)$.
\item To a function $f(t)$, defined for $t\ge0$, we introduce the Laplace transform and its inverse
\begin{equation}\label{inversion}
\tilde{f}(\tau) = \int_0^\infty e^{-t\tau}f(t)dt,\
\text{ where }\quad
 f(t) = \frac{1}{2\pi i}\int_{\Gamma_a} e^{\tau t} \tilde{f}(\tau) d\tau.
 \end{equation}
Here, $\Gamma_a = \{\tau:\Re\tau=a\}$ where $a$ is chosen such that $\tilde{f}(\tau)$ is analytic on the set
 $\{\tau: \Re\tau> a\}$. 

\item We introduce the normalized radial-Dirichlet eigenfunctions and eigenvalues in the unit ball $B_1$, {\it i.e.}
 \[ -\De_y\phi_j=\la_j\phi_j,\quad  \phi_j|_{y=1}=0,\quad 1=\int_{B_1} \phi_j^2(|x|)\, dx = 4\pi \int_0^1 \phi_j^2(y)y^2\, dy.\]
\EQ{\label{eq-eigen-formula}
\phi_j(y) = \frac{\sin(j\pi y)}{\sqrt{2\pi} y},\quad j=1,2,\ldots,\qquad
\la_j = (j\pi)^2,\quad j=1,2,\ldots.
}
\end{enumerate}

\bigskip\noindent{\bf Acknowledgements.}
The authors thank Juan J. L. Vel\'azquez for detailed discussions concerning the article \cite{bv-SIMA2000}, which motivated the present work. We also thank Panagiota Daskalopoulos, Qiang Du and Christophe Josserand for very stimulating discussions. CL acknowledges support by Simons Foundation as well as support from Department of Mathematics at Columbia University. MIW was supported in part by National Science Foundation Grant DMS-1908657 and Simons Foundation Math + X Investigator Award \#376319 (Michael I. Weinstein).

\section{Gas bubble in an incompressible liquid; the complete mathematical formulation}\label{sec:formulation}

In this section we first discuss the complete mathematical description of a 
 gas bubble immersed in an incompressible liquid 
 with constant surface tension. 
 We then present in Section \ref{sec-asymp-model} the asymptotic approximation studied in \cite{bv-SIMA2000}
  in which thermal diffusion is the key dissipation mechanism. 
  In Section \ref{sec-equilibria}, we derive an explicit family of our spherically symmetric equilibrium solutions of  both of the full and approximate systems.
  
\bigskip\noindent{\bf Equations for the liquid.}
Let ${\bf v}_l(x,t)$ denote the liquid velocity, $p_l(x,t)$ the liquid pressure and $T_l(x,t)$ is the liquid temperature.
We assume that the dynamics of the liquid outside the bubble is described by the incompressible (constant density) Navier--Stokes equations
\begin{subequations}\label{eq-1.1}
\begin{empheq}[right=\empheqrbrace\text{in $\R^3\setminus \Om(t)$, $t>0$,}]{align}
\pd_t (\rho_l {\bf v}_l) + \div(\rho_l {\bf v}_l \otimes {\bf v}_l) =&\,
\div \mathbb T_l, \label{eq-1.1-a}\\
\div {\bf v}_l =&\, 0, \label{eq-1.1-b}\\
\rho_l c_l (\pd_tT_l + {\bf v}_l\cdot\nb T_l) =&\, \div(\ka_l\nb T_l) + \mathbb{S}_l : \nb{\bf v}_l . \label{eq-1.1-c}
\end{empheq}
\end{subequations}
The stress tensor, $\mathbb T_l$ and viscous stress tensor,  $\mathbb{S}_l$, are given, respectively, by:
\[
\mathbb T_l= -p_l \mathbb I + \mathbb{S}_l({\bf v}_l)\quad{\rm and}\quad \mathbb{S}_l({\bf v}_l)=2\mu_l \mathbb{D}({\bf v}_l) ,
\quad 
\mathbb{D}({\bf u}) = \frac12 (\nb{\bf u} + \bke{\nb{\bf u}}^\top ).
\] 
Here,  $({\bf u}\otimes{\bf v})_{ij} = u_i v_j$ denotes the tensor product of vectors and  $\mathbb{A}:\mathbb{B} = \sum_{i,j=1}^3 A_{ij}B_{ij}={\rm tr}(AB^\top)$, where  $\mathbb{A} = (A)_{i,j=1}^3$, $\mathbb{B} = (B)_{i,j=1}^3$.
The equations depend on parameters: $\rho_l>0$, the density of the liquid, $\mu_l\ge0$, the dynamic viscosity of the liquid, $c_l$, the specific heat of the liquid, and $\ka_l$, the thermal conductivity of the liquid. 
Equations \eqref{eq-1.1-a} and \eqref{eq-1.1-b}  express, respectively, balance of momentum and conservation of mass.
These are coupled to equation \eqref{eq-1.1-c}, which governs the temperature field in the liquid.

\bigskip\noindent{\bf Equations for the gas.} The gas within the bubble is assumed to be a compressible fluid,
 characterized by its velocity ${\bf v}_g(x,t)$, pressure $p_g(x,t)$, density $\rho_g(x,t)$, temperature $T_g(x,t)$, and entropy per unit mass (specific entropy)  $s(x,t)$, with the assumption of the ideal gas law relating $p_g, T_g$ and $\rho_g$.
The governing equations are the viscous, compressible Navier--Stokes equations
 \begin{subequations}\label{eq-1.2}
\begin{empheq}[right=\empheqrbrace\text{in $\Om(t)$, $t>0$,}]{align}
\pd_t \rho_g + \div(\rho_g{\bf v}_g) =&\, 0, \label{eq-1.2-a}\\
\pd_t (\rho_g {\bf v}_g) + \div(\rho_g {\bf v}_g \otimes {\bf v}_g) =&\,
\div \mathbb T_g, \label{eq-1.2-b}\\
\rho_g T_g \bke{\pd_t s + {\bf v}_g\cdot\nb s } =&\, \div(\ka_g\nb T_g) + \mathbb S_g: \nb{\bf v}_g, \label{eq-1.2-c}\\
p_g =&\, \Rg T_g\rho_g , \label{eq-1.2-d}\\
s =&\, c_v \log\bke{\dfrac{p_g}{\rho_g^\ga} }, \label{eq-1.2-e}
\end{empheq}
\end{subequations}
where
\[
\mathbb T_g = -p_g\mathbb I + 2\mu_g\bke{\mathbb D({\bf v}_g) - \frac13\, (\div{\bf v}_g) \mathbb I} + \ze_g (\div{\bf v}_g) \mathbb I
\]
is the stress tensor of the gas in which
$\mu_g>0$ and $\ze_g$ are the dynamic viscosity and the bulk viscosity for the gas, respectively, and
\[\mathbb S_g = 2\mu_g\bke{\mathbb{D}({\bf v}_g) - \frac13 (\div{\bf v}_g)\mathbb{I}} + \ze_g(\div{\bf v}_g)\mathbb{I}\]
 is the viscous tensor of gas.
The constant $\ka_g$ is the thermal conductivity of the gas.
The constant $\Rg $ is the specific gas constant, the ratio of the ideal gas constant to the molar mass.
The  constant 
\EQ{\label{eq-def-gamma}
\ga \equiv 1 + \frac{\Rg}{c_v} = \frac{c_p}{c_v}>1
}
is called the adiabatic constant. 
Here, $c_p$ denotes the heat capacity at constant pressure and $c_v$ denotes the heat capacity at constant volume.
Equations \eqref{eq-1.2-a} and \eqref{eq-1.2-b} are the equations of motion and continuity, respectively, of a compressible fluid. Equation \eqref{eq-1.2-c} is the entropy equation. Equation \eqref{eq-1.2-d} is the equation of state (Boyle's law) for ideal gases. Equation \eqref{eq-1.2-e} is a consequence of the second law of thermodynamics, \eqref{eq-1.2-d}, and Joule's second law for ideal gases. 

\bigskip\noindent{\bf Boundary conditions at the liquid / gas interface.}
 Let the bubble surface, $\pd\Om(t)$, be given in spherical coordinates
\[
\boldsymbol{\om} = \boldsymbol{\om}(\th,\varphi,t)
\]
and let ${\bf n}$ denote the outward pointing unit outer normal on $\pd\Om(t)$.
The boundary conditions on $\pd\Om(t)$ are 
 \begin{subequations}\label{eq-1.3}
\begin{empheq}[right=\empheqrbrace\text{on $\pd\Om(t)$, $t>0$,}]{align}
{\bf v}_l(\boldsymbol{\om},t)\cdot\hat{\bf n} = {\bf v}_g(\boldsymbol{\om},t)\cdot\hat{\bf n} = \pd_t{\boldsymbol{\om}}\cdot\hat{\bf n}, \label{eq-1.3-a}\\
\hat {\bf n} \cdot \mathbb T_l - \hat {\bf n} \cdot \mathbb T_g = \si \hat {\bf n} (\nb_S\cdot \hat {\bf n}), \label{eq-1.3-b}\\
T_g = T_l, \label{eq-1.3-c}
\end{empheq}
\end{subequations}
where $\nb_S\,\cdot$ denotes the surface divergence, and $\si>0$ is the surface tension of the liquid - gas interface, here assumed to be a constant.
Equation \eqref{eq-1.3-a} is the kinematic boundary condition; the normal velocity of the material point on the bubble surface moves with the normal velocity of both the gas and the liquid. 
Equation \eqref{eq-1.3-b} is the stress balance equation.
Equation \eqref{eq-1.3-c} means the temperature is continuous across the interface. 
A detailed derivation of the fundamental equations of fluid dynamics is presented, for example, in \cite{Novotny-book2004,Feireisl-book2004}.

The system \eqref{eq-1.1}--\eqref{eq-1.3} depends on the
\begin{equation}
\textrm{ physical parameters: $\nu_l=\mu_l/\rho_l$, $\rho_l$, $c_l$, $\ka_l$, $\mu_g$, $\ka_g$, $\Rg $, $\ga$, $c_v$, $\ze_g$, $\si$,}\label{params}
\end{equation}
where $\nu_l$, the kinematic viscosity, is a nonnegative constant and all other parameters are all strictly positive constants.  We assume these parameters to be prescribed and fixed.  Furthermore, one prescribes:
\begin{align*}
&\textrm{the far-field liquid velocity ${\bf v}_\infty(t):=\lim_{|x|\to\infty}{\bf v}_l(x,t),$}\\
&\textrm{the far-field liquid pressure $p_\infty(t) := \lim_{|x|\to\infty}p_{l}(x,t)$, and}\\
&\textrm{the far-field liquid temperature $T_\infty(t):=\lim_{|x|\to\infty}T_l(x,t)$}.\end{align*}

Here, we assume that the far-field pressure in the liquid, $p_\infty(t)$,  is spatially uniform and is a small perturbation of a positive constant $p_{\infty,*}$: 
\[
p_\infty \in C^{1+\al}_t(\R_+),\quad
|p_\infty(t) - p_{\infty,*} | + \norm{\pd_tp_\infty}_{L^1_t(\R_+)}\le \eta_0,\quad
p_\infty(t)\to p_{\infty,*}\ \text{ as } t\to\infty,
\]
where $\eta_0>0$ is some small number to be chosen later.
We also assume that the far-field liquid temperature is a constant, $T_\infty$, and that the far-field liquid velocity vanishes:
\[
T_\infty(t) \equiv T_\infty,\qquad\qquad
{\bf v}_\infty(t) \equiv {\bf 0}.
\]

For fixed physical parameters \eqref{params}, $p_\infty(t)$, and $T_\infty$,
the system \eqref{eq-1.1}--\eqref{eq-1.3} governs the time-evolution of the 
\begin{align*}
 &\textrm{state variables in the liquid:}\quad  {\bf v}_l(x,t),\, p_l(x,t),\,T_l(x,t), \\
&\textrm{state variables in the gas:}\quad \rho_g(x,t) ,\, {\bf v}_g(x,t) ,\, p_g(x,t),\, T_g(x,t),\, s(x,t),\quad {\rm and}\\
& \textrm{the gas bubble  region,  $\Om(t)\subset \mathbb{R}^3$}.  \end{align*}
To this we add constitutive relations \eqref{eq-1.2-d}--\eqref{eq-1.2-e}, which enable us to express $T_g$ and $s$
in terms of $\rho_g, {\bf v}_g$ and $p_g$.
Moreover, since the liquid pressure $p_l$ solves the exterior Dirichlet boundary-value problem of Poisson equation 
\[
- \De p_l = \rho_l \nb{\bf v}_l : (\nb{\bf v}_l)^\top\ \text{ in }\R^3\setminus\Om(t),
\]
with the boundary conditions
\[
p_l = \hat{\bf n}\cdot\mathbb{S}_l\cdot\hat{\bf n} - \hat{\bf n}\cdot\mathbb{T}_g\cdot\hat{\bf n} - \si\nb_S\cdot \hat{\bf n}\ \text{ on }\pd\Om(t),\qquad
\lim_{|x|\to\infty} p_l(x,t) = p_\infty(t),
\]
$p_l - p_\infty(t)$ satisfies
\EQN{
- \De (p_l - p_\infty(t)) &= \rho_l \nb{\bf v}_l : (\nb{\bf v}_l)^\top\ \text{ in }\R^3\setminus\Om(t),\\
p_l - p_\infty(t) &= \hat{\bf n}\cdot\mathbb{S}_l\cdot\hat{\bf n} - \hat{\bf n}\cdot\mathbb{T}_g\cdot\hat{\bf n} - \si\nb_S\cdot \hat{\bf n} - p_\infty(t)\ \text{ on }\pd\Om(t),\qquad
\lim_{|x|\to\infty}p_l(x,t) - p_\infty(t) = 0.
}
For suitable $\nb {\bf v}_l$ with sufficient decay at spatial infinity, $p_l - p_\infty(t)$ can be expressed by means layer potentials as
\EQN{
p_l(x,t) - p_\infty(t) &= \int_{\R^3\setminus\Om(t)} G(x,y;t) \bkt{\rho_l \nb{\bf v}_l : (\nb{\bf v}_l)^\top}(y)\, dy\\
&\quad - \int_{\pd\Om(t)} \nb_y G(x,y;t)\cdot\hat{\bf n} \bkt{\hat{\bf n}\cdot\mathbb{S}_l\cdot\hat{\bf n} - \hat{\bf n}\cdot\mathbb{T}_g\cdot\hat{\bf n} - \si\nb_S\cdot \hat{\bf n} - p_\infty(t)}(y)\, dS_y,
}
where $G(x,y;t)$ is the Green's function for the exterior domain $\R^3\setminus\Om(t)$. 
Then
$p_l$, for $x\notin\Omega(t)$, can be expressed in terms ${\bf v}_l$ by
\EQN{
p_l(x,t) &= p_\infty(t) + \int_{\R^3\setminus\Om(t)} G(x,y;t) \bkt{\rho_l \nb{\bf v}_l : (\nb{\bf v}_l)^\top}(y)\, dy\\
&\quad - \int_{\pd\Om(t)} \nb_y G(x,y;t)\cdot\hat{\bf n} \bkt{\hat{\bf n}\cdot\mathbb{S}_l\cdot\hat{\bf n} - \hat{\bf n}\cdot\mathbb{T}_g\cdot\hat{\bf n} - \si\nb_S\cdot \hat{\bf n} - p_\infty(t)}(y)\, dS_y.
}
Therefore, \eqref{eq-1.1}--\eqref{eq-1.3} can be reduced to a problem for the unknown liquid and gas state variables,
 and the region filled with gas:
\[ {\bf v}_l(x,t),\,T_l(x,t),\, \rho_g(x,t) ,\, {\bf v}_g(x,t) ,\, p_g(x,t),\, \Om(t).\] 

\bigskip\noindent{\bf Initial data.} To solve for the evolution given by the full liquid / gas bubble system \eqref{eq-1.1}--\eqref{eq-1.3}, we must prescribe initial conditions
 for the state variables:
\EQ{\label{eq-initial}
{\bf v}_l(\cdot,0),\,
T_l(\cdot,0),\,
\rho_g(\cdot,0),\, 
{\bf v}_g(\cdot,0),\,
p_g(\cdot,0),
}
and for the bubble shape at time $t=0$:
\EQ{\label{eq-initial-Om} \Om(t)\Big|_{t=0} \ =\ \Om(0).
}
We assume the compatibility conditions for the initial data, {\it i.e.}, they satisfy \eqref{eq-1.1}--\eqref{eq-1.3} at $t=0$.
In particular, $\div {\bf v}_l(\cdot,0)=0$.

\section{An asymptotic approximation to \eqref{eq-1.1}--\eqref{eq-1.3}}\label{sec-asymp-model}

In this paper we work with the following approximation of Prosperetti \cite{Prosperetti-JFM1991} (see also Biro--Vel\'azquez \cite[Appendix A]{bv-SIMA2000} ) to the full liquid - bubble  system \eqref{eq-1.1}--\eqref{eq-1.3}: 
 \begin{subequations}\label{eq1.1simplified}
\begin{empheq}[right=\empheqrbrace\text{in $\R^3\setminus \Om(t)$, $t>0$,}]{align}
\pd_t {\bf v}_l =&\, \nu_l \De {\bf v}_l - {\bf v}_l\cdot\nb{\bf v}_l - \dfrac1{\rho_l}\, \nb p_l, \label{eq1.1simplified-a}\\
\div {\bf v}_l =&\, 0, \label{eq1.1simplified-b}\\
T_l(x,t) = &\, T_\infty,\quad \text{a prescribed constant},\label{eq1.1simplified-c}
\end{empheq}
\end{subequations}
where
$\nu_l = \frac{\mu_l}{\rho_l}\ge0$ is the kinematic viscosity of the liquid,
 \begin{subequations}\label{eq1.2simplified}
\begin{empheq}[right=\empheqrbrace\text{in $\Om(t)$, $t>0$,}]{align}
\pd_t \rho_g + \div(\rho_g{\bf v}_g) =&\, 0,\label{eq1.2simplified-a}\\
p_g =&\, p_g(t), \label{eq1.2simplified-b}\\
\rho_g T_g \bke{\pd_t s + {\bf v}_g\cdot\nb s } =&\, \div(\ka_g\nb T_g), \label{eq1.2simplified-c}\\
p_g =&\, \Rg T_g \rho_g , \label{eq1.2simplified-d}\\
s =&\, c_v \log\bke{\dfrac{p_g}{\rho_g^\ga} } \label{eq1.2simplified-e},
\end{empheq}
\end{subequations}
and
 \begin{subequations}\label{eq1.3simplified}
\begin{empheq}[right=\empheqrbrace\text{on $\pd\Om(t)$, $t>0$,}]{align}
{\bf v}_l(\boldsymbol{\om},t)\cdot\hat{\bf n} = {\bf v}_g(\boldsymbol{\om},t)\cdot\hat{\bf n} = \pd_t{\boldsymbol{\om}}\cdot\hat{\bf n}, \label{eq1.3simplified-a}\\
p_g \hat{\bf n} - p_l \hat{\bf n} + 2\mu_l \hat{\bf n} \cdot\mathbb{D}({\bf v}_l) = \si \hat{\bf n} (\nb_S\cdot \hat {\bf n}), \label{eq1.3simplified-b}\\
T_g = T_\infty, \label{eq1.3simplified-c}
\end{empheq}
\end{subequations}
with the far-field conditions 
\EQ{\label{eq-far-field-all}
\lim_{|x|\to\infty} {\bf v}_l(x,t) = {\bf 0},\qquad
\lim_{|x|\to\infty} p_l(x,t) = p_\infty(t),\qquad
\lim_{|x|\to\infty} T_l(x,t) = T_\infty.
}
This model reduces to that of \cite[Appendix A]{bv-SIMA2000} for the special case when $\mu_l=0$ in \eqref{eq1.3simplified-b}. 
Equation \eqref{eq1.3simplified-b} is the Young--Laplace boundary condition;
 the jump in pressure at the liquid--gas interface is equal to the surface tension, $\si$, times the mean curvature $H = \frac12 \nb_S\cdot \hat{\bf n}$. 
The approximate system \eqref{eq1.1simplified}--\eqref{eq1.3simplified} depends on the
\begin{equation}
\textrm{ physical parameters: $\nu_l$, $\rho_l$, $\ka_g$, $\Rg $, $\ga$, $c_v$, $\si$.}\label{params-approx}
\end{equation}
For fixed physical parameters \eqref{params-approx}, $p_\infty(t)$, and $T_\infty$,
the approximate system \eqref{eq1.1simplified}--\eqref{eq1.3simplified} governs the time-evolution of the state variables in the liquid:
\[ {\bf v}_l(x,t),\, p_l(x,t)\]
and in the gas
\[ \rho_g(x,t) ,\, {\bf v}_g(x,t) ,\, p_g(t),\, T_g(x,t),\, s(x,t)\]
and $\Om(t)$. 
We show below in Appendix \ref{sec-reduction} that the system \eqref{eq1.2simplified} can be reduced to a single equation \eqref{eq-bv-3.10-nonspherical} for $\rho_g$ depending only on $\Om$ and $p_g$.
Thus, \eqref{eq1.1simplified}--\eqref{eq1.3simplified} can be reduced to a problem with unknowns 
\[ {\bf v}_l(x,t),\, \rho_g(x,t) ,\, p_g(t),\, \Om(t).\]

As for the initial conditions to prescribe for the approximate system \eqref{eq1.1simplified}--\eqref{eq1.3simplified}, we need
\EQ{\label{eq-initial-approx}
{\bf v}_l(\cdot,0),\,
\rho_g(\cdot,0)
}
for state variables, and \eqref{eq-initial-Om} for the bubble shape at time $t=0$.
We do not prescribe initial data for $p_g$ since it can be derived from $p_l(\cdot,0)$ and $\Om(0)$ via \eqref{eq1.3simplified-b}.

In this article the approximate system \eqref{eq1.1simplified}--\eqref{eq1.3simplified} is considered under the assumption of spherical symmetry. 
More precisely, we assume for the system \eqref{eq1.1simplified}--\eqref{eq1.3simplified} that $\Om(t)$ is a sphere, ${\bf v}_l,{\bf v}_g$ are spherically symmetric, $p_l,\rho_g,T_g,s$ are radial.
Recall that a vector field ${\bf u}:\R^3\to\R^3$ is spherically symmetric if ${\bf u} = u(r)\hat{\bf r}$, $r=|x|$ and $\hat{\bf r}= x/|x|$, 
and a scalar function $f:\R^3\to\R$ is radial if $f=f(r)$. In this setting, \eqref{eq1.1simplified}--\eqref{eq1.3simplified}
reduces to the system \eqref{eq-bv-3.10prime}--\eqref{eq-bv-3.16prime} for $\rho(r,t), R(t)$.
Under the assumption of spherical symmetry, \eqref{eq-bv-3.10prime}--\eqref{eq-bv-3.16prime} is well-posed locally in time (Theorem \ref{thm-bv-3.1}), and well-posed globally in time for initial data which is close to the equilibrium (see \cite[Theorem 4.1]{bv-SIMA2000} and Theorem \ref{thm-bv-4.1-anyperturb}).

\section{Spherically symmetric equilibrium solutions}\label{sec-equilibria}

Both the full liquid / gas model \eqref{eq-1.1}--\eqref{eq-1.3} and the asymptotic model
 \eqref{eq1.1simplified}--\eqref{eq1.3simplified} share a family of spherically symmetric equilibrium (time-independent) solutions.  
Let $B_R$ denote the open ball  in $\mathbb{R}^3$  of radius $R$ which is centered at the origin.
Suppose a gas of density $\rho_g(x)$  occupies the region $B_R$. Then, the mass of the bubble is given by
\EQ{\label{eq-def-masss}
  \textrm{Mass}[\rho_g,R] \equiv \int_{B_R} \rho_g(x) dx.
}
Below we investigate spherical symmetric equilibrium solutions of both \eqref{eq-1.1}--\eqref{eq-1.3} and the approximation system \eqref{eq1.1simplified}--\eqref{eq1.3simplified}. 
We prove that the spherically symmetric equilibrium gas bubble of the approximate system \eqref{eq1.1simplified}--\eqref{eq1.3simplified}, and of the original system \eqref{eq-1.1}--\eqref{eq-1.3} with additional conditions, is, up to spatial translation of its center, uniquely determined by its total mass. 
Moreover, we prove that equilibrium bubbles of the approximate system \eqref{eq1.1simplified}--\eqref{eq1.3simplified} are spherical by applying the Alexandrov's theorem on closed constant-mean-curvature (CMC) surfaces.
In an equilibrium spherical bubble of the approximate system \eqref{eq1.1simplified}--\eqref{eq1.3simplified}, there exists a nontrivial, {\it e.g.} rotational, equilibrium gas flow (see Remark \ref{rmk-nonunique-gas}).

 \begin{proposition}[Spherically symmetric equilibria of the original system \eqref{eq-1.1}--\eqref{eq-1.3}]\label{prop-equilib-original}
Fix a constant $p_{\infty,*}>0$.
Assume the radiation condition for liquid temperature:
 \EQ{\label{eq-radiation-condition}
 T_l(|x|) = T_\infty + o(|x|^{-1}),\quad |x|\to\infty.
 }
 Then, there is a smooth map from values of the bubble mass to equilibrium radii, $R_*$:
   \[ M\in(0,\infty)\mapsto R_*[M],\]
   such that any regular (non-singular) spherical equilibrium solution
of \eqref{eq-1.1}--\eqref{eq-1.3} (for fixed parameters \eqref{params})  of bubble mass $M$ is expressible as:
\begin{subequations}
\label{eq-equilibrium}
\begin{align}
{\bf v}_{l,*} &= {\bf 0},\qquad\qquad 
p_{l,*} = p_{\infty,*} ,\qquad\qquad
\Om_* = B_{_{R_*[M]}}, \label{eq-equilibrium-a}\\
\rho_{g,*}[M] &= \frac1{\Rg T_\infty} \bke{ p_{\infty,*} +\frac{2\si}{R_*[M]} },\quad
{\bf v}_{g,*} = {\bf 0},\quad 
p_{g,*}[M] = p_{\infty,*} + \frac{2\si}{R_*[M]},\label{eq-equilibrium-b}\\
T_{g,*} &= T_{l,*} = T_\infty,\quad
s_* = c_v\log\bke{(\Rg T_\infty)^\ga \bke{ p_{\infty,*} +\frac{2\si}{R_*[M]}}^{1-\ga}}.\label{eq-equilibrium-c}
\end{align}
\end{subequations}

\end{proposition}

The proof of Proposition \ref{prop-equilib-original} is given in Appendix \ref{sec-steadystate}.

\begin{remark}\label{rem-prop-equilib}
The radiation condition \eqref{eq-radiation-condition} and the regularity assumption in part (1) are necessary for the uniqueness of the spherical equilibrium solutions $T_{l,*}$ and $T_{g,*}$ in \eqref{eq-equilibrium}. 
In fact, without such hypotheses there exists a two parameter family of  spherical equilibrium solutions $T_{l,*}$ and $T_{g,*}$ of \eqref{eq-1.1}--\eqref{eq-1.3} given by:
\begin{align*}
\textrm{ $T_{l,*}(r) = T_\infty - a_1/r$, $r\in[R_*,\infty)$, and $T_{g,*}(r) = T_\infty - a_1/R_* + a_2 (1/R_* - 1/r)$, $r\in[0,R*)$,}
\end{align*}
where $a_1, a_2\in\R$ are arbitrary.
\end{remark}

 \begin{proposition}[Equilibria of the approximate system \eqref{eq1.1simplified}--\eqref{eq1.3simplified}]\label{prop-equilib}
Fix a constant $p_{\infty,*}>0$.
    \begin{enumerate}
    
\item 
\underline{Equilibrium bubbles of \eqref{eq1.1simplified}--\eqref{eq1.3simplified} are spherical}:
Let $({\bf v}_{l,*},p_{l,*},\pd\Om_*,p_{g,*}, {\bf v}_{g,*}, \rho_{g,*},\cdots)$ be a $C^2$ steady-state solution of \eqref{eq1.1simplified}--\eqref{eq1.3simplified} with $\lim_{|x|\to\infty} p_{l,*}(x) = p_{\infty,*}$.
Assume $\mu_l\neq0$ and $\si\neq0$ in \eqref{eq1.3simplified-b}. 
Suppose that $\lim_{|x|\to\infty}{\bf v}_{l,*}(x) = O(|x|^{-2})$ and $\lim_{|x|\to\infty}\nb{\bf v}_{l,*}(x) = \mathbb O$.
Then ${\bf v}_{l,*}={\bf 0}$, $p_{l,*}=p_{\infty,*}$ and $\Om_*$ is a sphere.
Moreover, $\rho_{g,*}$ is constant and $\div{\bf v}_{g,*} = 0$.

\item 
\underline{Spherically symmetric equilibria of \eqref{eq1.1simplified}--\eqref{eq1.3simplified}}: The reduced / asymptotic model \eqref{eq1.1simplified}--\eqref{eq1.3simplified} shares the family of spherically symmetric equilibria displayed in \eqref{eq-equilibrium}. 
Furthermore, 
any regular spherical equilibrium solution of \eqref{eq1.1simplified}--\eqref{eq1.3simplified} is uniquely determined by its total mass as \eqref{eq-equilibrium}. No radiation condition \eqref{eq-radiation-condition} is required.

\item The mappings $M\in(0,\infty)\mapsto R_*[M]$ and $\rho_*[M]$, where $\rho_*:= \rho_{g,*}$, arising in Proposition \ref{prop-equilib-original} are continuous (even smooth).

\end{enumerate}

\end{proposition}

\begin{remark}\label{rmk-nonunique-gas}
Part (1) of Proposition \ref{prop-equilib} is the uniqueness of \eqref{eq-equilibrium-a} for the equilibrium liquid flow and bubble shape. 
It does not imply the uniqueness of \eqref{eq-equilibrium-b}--\eqref{eq-equilibrium-c} for the gas phase.
Indeed, replacing ${\bf v}_{g,*}={\bf 0}$ in \eqref{eq-equilibrium-b} with any non-trivial solenoidal vector field in $B_{R_*}$ yields another steady state solution to the approximate system \eqref{eq1.1simplified}--\eqref{eq1.3simplified}. 
Recall that a vector field ${\bf u}$ in $\Om$ is solenoidal if $\div{\bf u}=0$ and ${\bf u}\cdot\hat{\bf n}|_{\pd\Om}=0$.
One can choose, for example, ${\bf v}_{g,*}(x_1,x_2,x_3) = (-x_2,x_1,0)$.
The example is ruled out by the spherically symmetric assumption in Part (1) of Proposition \ref{prop-equilib}.
Other possible way to show the uniqueness of \eqref{eq-equilibrium-b}--\eqref{eq-equilibrium-c} is to impose irrotational assumption.
The nonuniqueness is due to the absence of viscosity for the gas in the approximated stress balance equation \eqref{eq1.3simplified-b}.
\end{remark}

\begin{remark}
When the liquid is inviscid, {\it i.e.}, $\mu_l=\nu_l=0$, in Part (1) of Proposition \ref{prop-equilib}, we still expect ${\bf v}_{l,*}\equiv{\bf 0}$ because of the far-field condition $\lim_{|x|\to\infty} {\bf v}_{l,*}(x) = {\bf 0}$.
In this case, a Liouville-type theorem for the stationary Euler equations in a three-dimensional exterior domain with slip boundary condition is needed.
However, such Liouville-type result is unavailable to our best knowledge. 
For survey on related problems, see, for example, \cite[I.2.1]{Galdi-book1994} and \cite{Bang-thesis2021}. 
\end{remark}

\begin{proof}[Proof of Proposition \ref{prop-equilib}]

We first prove Part (1) concerning the uniqueness of the equilibrium \eqref{eq-equilibrium-a} of the approximate system \eqref{eq1.1simplified}--\eqref{eq1.3simplified}. 
Note that steady-state solutions of \eqref{eq1.1simplified}--\eqref{eq1.3simplified} solve
 \begin{subequations}\label{eq1.1simplified-equilib}
\begin{empheq}[right=\empheqrbrace\text{in $\R^3\setminus \Om_*$,}]{align}
{\bf 0} =&\, \nu_l \De {\bf v}_{l,*} - {\bf v}_{l,*}\cdot\nb{\bf v}_{l,*} - \dfrac1{\rho_l}\, \nb p_{l,*}, \label{eq1.1simplified-a-equilib}\\
\div {\bf v}_{l,*} =&\, 0, \label{eq1.1simplified-b-equilib}\\
T_{l,*}(x) = &\, T_\infty,\quad \text{a prescribed constant},\label{eq1.1simplified-c-equilib}
\end{empheq}
\end{subequations}
 \begin{subequations}\label{eq1.2simplified-equilib}
\begin{empheq}[right=\empheqrbrace\text{in $\Om_*$,}]{align}
\div(\rho_{g,*}{\bf v}_{g,*}) =&\, 0,\label{eq1.2simplified-a-equilib}\\
p_{g,*}(x) =&\, p_{g,*},\quad \text{a constant}, \label{eq1.2simplified-b-equilib}\\
\rho_{g,*} T_{g,*} \bke{ {\bf v}_{g,*}\cdot\nb s_* } =&\, \div(\ka_g\nb T_{g,*}), \label{eq1.2simplified-c-equilib}\\
p_{g,*} =&\, \Rg T_{g,*} \rho_{g,*} , \label{eq1.2simplified-d-equilib}\\
s_* =&\, c_v \log\bke{\dfrac{p_{g,*}}{\rho_{g,*}^\ga} } \label{eq1.2simplified-e-equilib},
\end{empheq}
\end{subequations}
 \begin{subequations}\label{eq1.3simplified-equilib}
\begin{empheq}[right=\empheqrbrace\text{on $\pd\Om_*$,}]{align}
{\bf v}_{l,*}\cdot\hat{\bf n} = {\bf v}_{g,*}\cdot\hat{\bf n} = 0, \label{eq1.3simplified-a-equilib}\\
p_{g,*} \hat{\bf n} - p_{l,*} \hat{\bf n} + 2\mu_l \hat{\bf n} \cdot\mathbb{D}({\bf v}_{l,*}) = \si \hat{\bf n} (\nb_S\cdot \hat {\bf n}), \label{eq1.3simplified-b-equilib}\\
T_{g,*} = T_\infty, \label{eq1.3simplified-c-equilib}
\end{empheq}
\end{subequations}
and the far-field velocity and pressure are
\EQ{\label{eq-far-field-pressure-equilib}
\lim_{|x|\to\infty} {\bf v}_{l,*}(x) = {\bf 0},\qquad\qquad
\lim_{|x|\to\infty} p_{l,*}(x) = p_{\infty,*}.
}

For $r>0$ sufficiently large, multiplying the equation \eqref{eq1.1simplified-a-equilib} by ${\bf v}_{l,*}$, integrating over $B_r\setminus\Om_*$, using integration by parts formula and $\div {\bf v}_{l,*}=0$, we obtain
\EQ{\label{eq-v*-energy-alt}
0 &= \nu_l \int_{B_r\setminus\Om_*} \De{\bf v}_{l,*}\cdot{\bf v}_{l,*}\, dx - \int_{B_r\setminus\Om_*} {\bf v}_{l,*}\cdot\nb{\bf v}_{l,*}\cdot{\bf v}_{l,*}\, dx - \frac1{\rho_l} \int_{B_r\setminus\Om_*} \nb p_{l,*}\cdot{\bf v}_{l,*}\\
&= - \nu_l \int_{B_r\setminus\Om_*} |\nb{\bf v}_{l,*}|^2\, dx + \nu_l \int_{\pd\Om_*} (-\hat{\bf n})\cdot\nb{\bf v}_{l,*}\cdot{\bf v}_{l,*}\, dS - \int_{\pd\Om_*} \frac{|{\bf v}_{l,*}|^2}2\, {\bf v}_{l,*}\cdot(-\hat{\bf n})\, dS\\
&\quad - \frac1{\rho_l} \int_{\pd\Om_*} p_{l,*} {\bf v}_{l,*}\cdot(-\hat{\bf n})\, dS 
+ \nu_l\int_{\pd B_r} \hat{\bf n}_{_{\pd B_r}}\cdot\nabla {\bf v}_{l,*}\cdot  {\bf v}_{l,*}\, dS
 - \int_{\pd B_r} \frac{|{\bf v}_{l,*}|^2}2\, {\bf v}_{l,*}\cdot \hat{\bf n}_{_{\pd B_r}}\, dS \\
&\quad - \frac1{\rho_l} \int_{\pd B_r} p_{l,*} {\bf v}_{l,*}\cdot\hat{\bf n}_{_{\pd B_r}}\, dS.
}
The third and fourth terms of \eqref{eq-v*-energy-alt} on the right hand side vanish since ${\bf v}_{l,*}\cdot\hat{\bf n}|_{\pd\Om_*}=0$.
Consider now the last three terms in \eqref{eq-v*-energy-alt}. The first two tend to zero as $r\to\infty$
using the hypotheses that ${\bf v}_{l,*}(x) = O(|x|^{-2})$ and $\lim_{|x|\to\infty} \nabla{\bf v}_{l,*}(x) = \mathbb O$.
Finally, 
the last term in \eqref{eq-v*-energy-alt} also tends to zero as $r\to\infty$. 
Indeed, since $\lim_{|x|\to\infty}{\bf v}_{l,*}(x) = O(|x|^{-2})$,
\EQN{
\frac1{\rho_l} \lim_{r\to\infty} \int_{\pd B_r} p_{l,*} {\bf v}_{l,*}\cdot\hat{\bf n}\, dS 
&= \frac{p_{\infty,*}}{\rho_l} \lim_{r\to\infty} \int_{\pd B_r} {\bf v}_{l,*}\cdot\hat{\bf n}_{_{\pd B_r}}\, dS + \frac1{\rho_l}  \lim_{r\to\infty} \int_{\pd B_r} (p_{l,*} - p_{\infty,*}){\bf v}_{l,*}\cdot\hat{\bf n}_{_{\pd B_r}}\, dS\\
&= \bke{ \frac{p_{\infty,*}}{\rho_l}  \int_{\R^3\setminus\Om_*} \div{\bf v}_{l,*}\, dx - \frac1{\rho_l}\, p_{\infty,*} \int_{\pd \Om_*} {\bf v}_{l,*}\cdot(-\hat{\bf n})\, dS} + 0 = 0
}
since ${\bf v}_{l,*}$ is solenoidal: $\div {\bf v}_{l,*} = 0$ and ${\bf v}_{l,*}\cdot\hat{\bf n}|_{\pd\Om_*}=0$.
Thus, by taking $r\to\infty$, \eqref{eq-v*-energy-alt} becomes 
\EQ{\label{eq-v*-energy-1}
0 = - \nu_l \int_{\R^3\setminus\Om_*} |\nb{\bf v}_{l,*}|^2\, dx - \nu_l\int_{\pd\Om_*} \hat{\bf n}\cdot\nb{\bf v}_{l,*}\cdot{\bf v}_{l,*}\, dS.
}
Multiplying the stress balance equation \eqref{eq1.3simplified-b-equilib} by ${\bf v}_{l,*}$ and using ${\bf v}_{l,*}\cdot\hat{\bf n}= 0$ yield $\hat{\bf n}\cdot\mathbb{D}({\bf v}_{l,*})\cdot{\bf v}_{l,*}|_{\pd\Om_*} = 0$ since $\mu_l\neq0$.
Using the expression of the deformation tensor $\mathbb{D}({\bf v}_{l,*}) = ( \nb{\bf v}_{l,*} + (\nb{\bf v}_{l,*})^\top )/2$,
\EQ{\label{eq-sbe-integral}
0 = 2 \hat{\bf n}\cdot\mathbb{D}({\bf v}_{l,*})\cdot{\bf v}_{l,*} 
&= \hat{\bf n}\cdot\bke{\nb{\bf v}_{l,*} + (\nb{\bf v}_{l,*})^\top}\cdot{\bf v}_{l,*}\\
&= \hat{\bf n}\cdot\nb{\bf v}_{l,*}\cdot{\bf v}_{l,*} + {\bf v}_{l,*}\cdot\nb{\bf v}_{l,*}\cdot\hat{\bf n}.
}
Using \eqref{eq-sbe-integral}, \eqref{eq-v*-energy-1} can be written as
\EQ{\label{eq-v*-energy-2}
0 = - \nu_l \int_{\R^3\setminus\Om_*} |\nb{\bf v}_{l,*}|^2\, dx + \nu_l\int_{\pd\Om_*}{\bf v}_{l,*}\cdot\nb{\bf v}_{l,*} \cdot \hat{\bf n}\, dS.
}
Since $-\hat{\bf n}$ is the outward normal of $\R^3\setminus\Om_*$ on $\pd\Om_*$, by \eqref{eq-v*-energy-2} and the divergence theorem,
\EQN{
0 &= - \nu_l \int_{\R^3\setminus\Om_*} |\nb{\bf v}_{l,*}|^2\, dx - \nu_l \int_{\pd\Om_*} {\bf v}_{l,*}\cdot\nb{\bf v}_{l,*}\cdot(-\hat{\bf n})\, dS\\
&= - \nu_l \int_{\R^3\setminus\Om_*} |\nb{\bf v}_{l,*}|^2\, dx - \nu_l \int_{\R^3\setminus\Om_*} \div({\bf v}_{l,*}\cdot\nb{\bf v}_{l,*})\, dx\\
&= - \nu_l \int_{\R^3\setminus\Om_*} |\nb{\bf v}_{l,*}|^2\, dx - \nu_l \int_{\R^3\setminus\Om_*} \nb{\bf v}_{l,*} : (\nb{\bf v}_{l,*})^\top\, dx,\quad \text{where } A:B := \sum_{i,j}A_{ij}B_{ij} = {\rm tr}(AB^\top),\\
&= - \nu_l \int_{\R^3\setminus\Om_*} \nb{\bf v}_{l,*} : \bke{\nb{\bf v}_{l,*} + (\nb{\bf v}_{l,*})^\top} dx 
 = - \nu_l \int_{\R^3\setminus\Om_*} {\rm tr}\bkt{ \nb{\bf v}_{l,*}\bke{\nb{\bf v}_{l,*} + (\nb{\bf v}_{l,*})^\top} } dx.
}
For any square matrix $A$, decomposing into symmetric and anti-symmetric parts we have:\\
  $A(A+A^\top) =  \frac{1}{2}(A+A^\top)^2 + \frac{1}{2}(A-A^\top)(A+A^\top)$.
Linearity of the trace and the identities ${\rm tr}(AC) = {\rm tr}(CA)$ and ${\rm tr}(A) = {\rm tr}(A^\top)$, then imply 
${\rm tr}\left[A(A+A^\top)\right] = \frac{1}{2}{\rm tr}\left[(A+A^\top)^2\right] = \frac{1}{2} |A+A^\top|^2$ since $A+A^\top$ is symmetric. Hence,
\EQN{ \frac12 \int_{\R^3\setminus\Om_*} | \nb{\bf v}_{l,*} + (\nb{\bf v}_{l,*})^\top |^2\, dx = 0.}
Therefore, $\nb{\bf v}_{l,*} + (\nb{\bf v}_{l,*})^\top\equiv0$ in $\R^3\setminus\Om_*$.
Integrating directly the equation $\pd_i({\bf v}_{l,*})_j + \pd_j({\bf v}_{l,*})_i = 0$, $i,j=1,2,3$, we obtain that for some vector ${\bf v}_0, \boldsymbol{\om} \in\R^3$ and some point $x_0\in\R^3$ 
\[
{\bf v}_{l,*}(x) = {\bf v}_0 + \boldsymbol{\om}\times(x-x_0).
\]
This implies ${\bf v}_{l,*}\equiv{\bf 0}$ since ${\bf v}_{l,*}(x)\to{\bf 0}$ as $|x|\to0$.
Since ${\bf v}_{l,*}\equiv{\bf 0}$, \eqref{eq1.1simplified-a-equilib} implies that $\nb p_{l,*} = {\bf 0}$, and so $p_{l,*} \equiv p_{\infty}$ is a constant.
Since $\mathbb D({\bf v}_{l,*}) = (\nb{\bf v}_{l,*} + (\nb{\bf v}_{l,*})^\top)/2 = 0$, the stress balance equation \eqref{eq1.3simplified-b-equilib} becomes
\[
p_{g,*} - p_{l,*} = \si \nb_S\cdot \hat {\bf n}\ \text{ on }\pd\Om_*.
\] 
Since both $p_{g,*}$ and $p_{l,*}$ are constant, $\pd\Om_*$ is a closed constant-mean-curvature (CMC) surface.
By Alexandrov's Theorem \cite{Alexandrov-AMPA1962}, $\Om_*$ must be a sphere.

We now deal with the system \eqref{eq1.2simplified-equilib} for the gas. 
Plugging \eqref{eq1.2simplified-e-equilib} into \eqref{eq1.2simplified-c-equilib} and using \eqref{eq1.2simplified-b-equilib} and \eqref{eq1.2simplified-d-equilib}, we have
\EQ{\label{eq1.2simplified-c+e-equilib}
\frac{\ka_g}{\ga c_v} \De\bke{\frac1{\rho_{g,*}}} = -{\bf v}_{g,*}\cdot\nb \log\rho_{g,*}.
}
Integrating the above equation over $\Om_*$, applying integration by parts formula, and using the boundary condition \eqref{eq1.3simplified-a-equilib}, we derive
\[
\int_{\Om_*} \De\bke{\frac1{\rho_{g,*}}} dx = 0.
\]
In view of \eqref{eq1.2simplified-d-equilib} and \eqref{eq1.3simplified-c-equilib}, $\rho_{g,*}|_{\pd\Om_*}$ is a constant.
So
\[
0 = \int_{\Om_*} \De\bke{\frac1{\rho_{g,*}}} dx = \int_{\pd\Om_*} \nb\bke{\frac1{\rho_{g,*}}}\cdot\hat{\bf n}\, dS = -\frac1{(\rho_{g,*}|_{\pd\Om_*})^2} \int_{\pd\Om_*} \nb\rho_{g,*}\cdot\hat{\bf n}\, dS,
\]
implying $\int_{\pd\Om_*} \nb\rho_{g,*}\cdot\hat{\bf n}\, dS = 0$.
Hence, for any $f,g\in C^\infty$,
\EQ{\label{eq-fg-nabla-rho}
\int_{\Om_*} g(\rho_{g,*})& \De\bke{f(\rho_{g,*})} dx
= -\int_{\Om_*} \nb\bke{g(\rho_{g,*})}\cdot \nb\bke{f(\rho_{g,*})} dx + \int_{\pd\Om_*} g(\rho_{g,*})\nb\bke{f(\rho_{g,*})}\cdot\hat{\bf n}\, dS\\
&= - \int_{\Om_*} g'(\rho_{g,*}) f'(\rho_{g,*}) |\nb \rho_{g,*}|^2\, dx + g(\rho_{g,*}|_{\pd\Om_*}) f'(\rho_{g,*}|_{\pd\Om_*}) \int_{\pd\Om_*} \nb\rho_{g,*}\cdot\hat{\bf n}\, dS\\
&= - \int_{\Om_*} g'(\rho_{g,*}) f'(\rho_{g,*}) |\nb \rho_{g,*}|^2\, dx.
}
Moreover, using the steady-state continuity equation \eqref{eq1.2simplified-a-equilib} and the boundary condition \eqref{eq1.3simplified-a-equilib}, we have for any $h\in C^\infty$ that
\EQN{
\int_{\Om_*} h(\rho_{g,*}) \div{\bf v}_{g,*}\, dx 
&= - \int_{\Om_*} \nb\bke{h(\rho_{g,*})}\cdot{\bf v}_{g,*}\, dx + \int_{\pd\Om_*} h(\rho_{g,*}){\bf v}_{g,*}\cdot\hat{\bf n}\, dS\\
&= - \int_{\Om_*} h'(\rho_{g,*})\nb\rho_{g,*}\cdot{\bf v}_{g,*}\, dx
= \int_{\Om_*} h'(\rho_{g,*})\rho_{g,*}\div{\bf v}_{g,*}\, dx,
}
or equivalently,
\[
0 = \int_{\Om_*} \bke{h(\rho_{g,*}) - h'(\rho_{g,*})\rho_{g,*}}\div{\bf v}_{g,*}\, dx.
\]
Using \eqref{eq1.2simplified-a-equilib}, \eqref{eq1.2simplified-c+e-equilib}, and \eqref{eq-fg-nabla-rho} with $g(\rho) = h'(\rho) - h''(\rho)\rho$, $f(\rho) = 1/\rho$, we obtain
\EQN{
0&= - \int_{\Om_*} \bke{h(\rho_{g,*}) - h'(\rho_{g,*})\rho_{g,*}} \frac{\nb\rho_{g,*}}{\rho_{g,*}} \cdot{\bf v}_{g,*}\, dx
= - \int_{\Om_*} \bke{h(\rho_{g,*}) - h'(\rho_{g,*})\rho_{g,*}} \nb\log\rho_{g,*}\cdot{\bf v}_{g,*}\, dx\\
&= \frac{\ka_g}{\ga c_v} \int_{\Om_*} \bke{h(\rho_{g,*}) - h'(\rho_{g,*})\rho_{g,*}} \De\bke{\frac1{\rho_{g,*}}} dx
= - \frac{\ka_g}{\ga c_v} \int_{\Om_*} \frac{h''(\rho_{g,*})}{\rho_{g,*}} |\nb\rho_{g,*}|^2\, dx.
}
Simply choose $h(\rho) = \rho^3$ to derive $\int_{\Om_*} |\nb\rho_{g,*}|^2\, dx = 0$.
This implies $\rho_{g,*}$ is constant and thus $\div{\bf v}_{g,*} = {\bf 0}$ by \eqref{eq1.2simplified-a-equilib}. Part (1) of Proposition \ref{prop-equilib} is asserted.

To derive the spherically symmetric equilibria of the approximate system \eqref{eq1.1simplified}--\eqref{eq1.3simplified}, we make use of Proposition \ref{prop:reduction}  (below), which 
presents a reduction of  \eqref{eq1.1simplified}--\eqref{eq1.3simplified}, in the spherically symmetric case, to an equivalent system for  $\rho$ and $R$, where $\rho = \rho_g$; see \eqref{eq-bv-3.10prime}--\eqref{eq-bv-3.16prime} below. All other state variables 
may be derived from these; see Remark \ref{red-ext}. It therefore suffices to seek time-independent solutions of  \eqref{eq-bv-3.10prime}--\eqref{eq-bv-3.16prime}. Setting $\partial_t\rho=\partial_tR=0$ 
we obtain from \eqref{eq-bv-3.10prime} and \eqref{eq-bv-3.15prime} that $R(t)\equiv R_*$ (constant equilibrium radius)  and 
\[
\De \log\rho = 0\ \text{ in }B_{R_*},\ \text{ and }\ 
\pd_r\rho(R_*) = 0.
\]
Therefore, $\rho(r)\equiv \rho_*$ for $0\le r\le R_*$ (constant equilibrium density). 
Evaluating \eqref{eq-bv-3.16prime} at $r=R_*$ and using that $\rho(R_*)=\rho_*$ we conclude
\[
\rho_* = \frac1{\Rg T_\infty}\bke{p_{\infty,*}+\frac{2\si}{R_*}}.
\]
The mass of the equilibrium gas bubble of density $\rho_*$ and radius $R_*$ is given by
\[
M = \int_{B_{R_*}} \rho_*\, dx 
= \frac{4\pi}3 \rho_* R_*^3.
\]
Therefore, for fixed mass $M$, the steady state $(\rho_*,R_*)$ is determined by the simultaneous
algebraic equations 
:
  \begin{subequations}\label{eq-steadystate}
\begin{align}
\dfrac{4\pi}3 \rho_* R_*^3 &= M, \label{eq-steadystate-a}\\
\Rg T_\infty  \rho_* &= 
 p_{\infty,*} + \dfrac{2\si}{R_*}. 
\label{eq-steadystate-b}
\end{align}
\end{subequations}
Therefore, the equilibrium radius $R_*$ is given by a solution to the cubic equation
\EQ{\label{eq-cubic}
p_{\infty,*} R_*^3 + 2\si R_*^2 - \frac{3\Rg T_\infty M}{4\pi} = 0.
}
It is readily seen that for each fixed $M>0$, the cubic \eqref{eq-cubic} has a unique positive root $R_*$. 
This choice of $R_*$ determines the equilibria gas density  and, via the relation $p_g=\Rg T_g \rho_g$, the gas pressure:
\[ \rho_* = \frac1{\Rg T_\infty}\bke{p_{\infty,*}+\frac{2\si}{R_*}},\quad p_* = \Rg T_\infty \rho_* = 
p_{\infty,*} + \dfrac{2\si}{R_*}.
\]

Once we obtain the equilibrium $(\rho_*,R_*,p_*)$, we can recover, using the formulas in \eqref{eq-reconstruct}, the corresponding steady state solution to the system \eqref{eq1.1simplified}--\eqref{eq1.3simplified} for the gas velocity ${\bf v}_g$, the gas temperature $T_g$, the specific entropy $s$ of the gas, the liquid velocity ${\bf v}_l$, and the liquid pressure $p_l$:
\[
{\bf v}_{g,*} = {\bf v}_{l,*} = {\bf 0},\quad
T_{g,*} = T_\infty,\quad
s_* = c_v \log\bke{\frac{p_*}{\rho_*^\ga}},\quad
p_{l,*} = p_{\infty,*}.
\]
Summarizing, we have derived the spherically symmetric equilibrium stated in \eqref{eq-equilibrium}. This proves Part (2). Part (3) of Proposition \ref{prop-equilib} follows from the smooth dependence of the simple roots of a given polynomial on its coefficients. 
This completes the proof of Proposition \ref{prop-equilib}.

\end{proof}

\begin{remark}
The equilibrium radius $R_*$ can be expressed explicitly in terms of $p_{\infty,*}, \si, \Rg , T_\infty, M$ by using the solution formula for the cubic equation \eqref{eq-cubic}.
\end{remark}

Below, in Proposition \ref{prop:reduction}, we shall reduce the study of spherically symmetric solutions to a closed
 system of equations for the gas density, $\rho(r,t)$ and the bubble radius $R(t)$, together with a condition on $\rho(R(t),t)$, the gas density at the free boundary.  Our proof of asymptotic stability is carried out in this setting.
We shall use the following continuity result:

\begin{proposition}\label{prop:contin} 
Fix a constant $p_{\infty,*}>0$.
Fix a density-radius pair $(\rho_0(x),R_0)\in L^\infty\times\R_+$ and let $M_0={\rm Mass}[\rho_0,R_0]$ denote the mass of the corresponding bubble. Let $(\rho_*[M_0], R_*[M_0])$ denote the equilibrium radius and density, given by Proposition \ref{prop-equilib},  for which ${\rm Mass}[\rho_*[M_0], R_*[M_0]] = M_0 = {\rm Mass}[\rho_0,R_0]$. 
Then, any equilibrium $(\rho_*[M_*], R_*[M_*])$, $M_*>0$, close to $(\rho_0(x),R_0)$ in $L^\infty\times\R$ is also close to 
  $(\rho_*[M_0], R_*[M_0])$. 
Even more strongly, there is a constant $C=C(R_*,\rho_0, R_0)>0$ such that
 \begin{equation}
  | R_*[M_0] -R_*[M_*]|  + | \rho_*[M_0] -\rho_*[M_*]|  \le C \left(|R_0-R_*[M_*]|\ +\ \|\rho_0-\rho_*[M_*]\|_{L^\infty(B_{R_0})}\right).\label{Rrho-cont}
  \end{equation}
 \end{proposition}
 
\begin{proof} Let $(\rho_0(x),R_0)$, $M_0$, and $(\rho_*[M_*], R_*[M_*])$  be as hypothesized. 
We first bound the difference $|R_*[M_0]-R_*[M_*]|$. 
The equilibrium radii $R_*[M_*]$ and $R_*[M_0]$ satisfy cubic equations with mass parameters $M_*$ and $M_0$, respectively:
\begin{align*}
 p_{\infty,*}\ R_*[M_*]^3 + 2\si  R_*[M_*]^2 - \frac{3\Rg M_*T_\infty}{4\pi} &= 0\ \ \textrm{and}\ \ p_{\infty,*}\ R_*[M_0]^3 + 2\si R_*[M_0]^2 - \frac{3\Rg M_0T_\infty}{4\pi} = 0.
  \end{align*}
Taking the difference of these two equations gives:
\EQN{
p_{\infty,*} (R_*[M_*] - R_*[M_0])& (R_*[M_*]^2 + R_*[M_*] R_*[M_0] + R_*[M_0]^2)\\ 
& +2\si(R_*[M_*] - R_*[M_0])(R_* + R_*[M_0]) - \frac{3\Rg T_\infty}{4\pi} (M_* - M_0) = 0,
}
and therefore
\EQ{\label{eq-diff-tdRs}
\abs{R_*[M_0] - R_*[M_*]} = \frac{3\Rg T_\infty  |M_0 - M_*|}{4\pi\bkt{ p_{\infty,*}(R_*[M_*]^2 + R_*[M_*] R_*[M_0] + R_*[M_0]^2) + 2\si(R_*[M_*] + R_*[M_0])}}.
}
Bounding  $\abs{R_*[M_0] - R_*[M_*]}$ therefore reduces to bounding $ |M_0 - M_*|$.
Expanding  $M_0$ about $ M_* $ we have:
\EQN{
M_0 &= \int_{B_{R_0}} \rho_0 = \int_{B_{R_0}} \rho_*[M_*] + \int_{B_{R_0}} (\rho_0-\rho_*[M_*])
 = \frac{4\pi}3 R_0^3\rho_*[M_*] + \int_{B_{R_0}} (\rho_0 - \rho_*[M_*])\\
 &= M_* + \frac{4\pi}3 (R_0^3 - R_*[M_*]^3)\rho_*[M_*] + \int_{B_{R_0}} (\rho_0 - \rho_*[M_*]).
 }
 Therefore,
 \EQ{\label{eq-diff-tdM0}
\abs{M_0 - M_* }\le \frac{4\pi}3 \rho_*[M_*](R_0^2 + R_0R_*[M_*] + R_*[M_*]^2)|R_0 - R_*[M_*]| + \frac{4\pi}3 R_0^3\norm{\rho_0 - \rho_*[M_*]}_{L^\infty(B_{R_0})} .
}
The bounds \eqref{eq-diff-tdM0} and \eqref{eq-diff-tdRs} imply that $\abs{R_*[M_0] - R_*[M_*]}$ satisfies the bound \eqref{Rrho-cont}. \\
${\quad}$ Finally, we bound the difference $\abs{\rho_*[M_0] - \rho_*[M_*]}$. Taking the difference of the relations 
$  \dfrac{4\pi}3 \rho_*[M_0] R_*[M_0]^3 = M_0$ and $ \dfrac{4\pi}3 \rho_*[M_*] R_*[M_*]^3 = M_*$,
we have 
\[
\frac{4\pi}3 R_*[M_0]^3 \rho_*[M_0] - \frac{4\pi}3 R_*[M_*]^3 \rho_*[M_*] = M_0 - M_*.
\]
Therefore,
\EQ{\label{eq-diff-rhos}
\frac{4\pi}{3} R_*[M_0]^3 &( \rho_*[M_0] - \rho_*[M_*])\\
& =  (M_0 - M_*)  + \frac{4\pi}{3} \left[ R_*[M_0]^2+ R_*[M_0] R_*[M_*] +R_*[M_*]^2 \right]\left( R_*[M_0] - R_*[M_*]\right).
}
The bound on $\abs{ \rho_*[M_0] - \rho_*[M_*]}$ now follows by estimating \eqref{eq-diff-rhos} using the bounds \eqref{eq-diff-tdM0} and \eqref{eq-diff-tdRs}.
 \end{proof}

\section{Reduction of the asymptotic model to a system for $\rho(r,t)$ and $R(t)$}\label{reduce}

The main purpose of this article is to study the stability of the spherically symmetric equilibrium \eqref{eq-equilibrium} of the approximation system \eqref{eq1.1simplified}--\eqref{eq1.3simplified}. The perturbations we consider are spherically symmetric and hence we work 
with the following reduction of the initial value problem:

\begin{proposition}\label{prop:reduction}
Any sufficiently regular spherically symmetric solution of \eqref{eq1.1simplified}--\eqref{eq1.3simplified}
 can be constructed from a solution of the following reduced system of equations for the gas density $\rho_g(r,t)\equiv\rho(r,t) $, for $0\le r\le R(t)$, and the bubble radius  $R(t)$, together with 
 a boundary condition on $\rho(r,t)$ at the free boundary $r=R(t)$:
 
\begin{subequations}
\label{red-eqns}
\begin{align}
\pd_t\rho(r,t) &= 
\frac{\ka}{\ga c_v} \De_r\log\rho(r,t) + \frac{1}{\ga} \frac{\pd_t p(t)}{ p(t)}\Big(  \frac13 r \pd_r\rho(r,t) 
+ \rho(r,t) \Big) ,\quad 0\le r\le R(t),\ t>0, \label{eq-bv-3.10prime}\\
\dot R(t) &= -\frac\ka{\ga c_v} \frac{\pd_r\rho(R(t),t)}{\bke{\rho(R(t),t)}^2} - \frac{R(t)}{3\ga} \frac{\pd_t p(t)}{p(t)},\quad t>0,\label{eq-bv-3.15prime}\\
\rho(R(t),t) &= \frac1{\Rg T_\infty} \bkt{p_\infty(t) + \frac{2\si}{R(t)} 
+ 4\mu_l \frac{\dot R}R
+ \rho_l\bke{R(t)\ddot R(t) + \frac32 (\dot R(t))^2} },\quad t>0,
 \label{eq-bv-3.16prime}
\end{align}
\end{subequations}
with initial data $\rho(\cdot,0)$, $R(0)$, $\dot R(0)$.
Here, $p=p(t)=p_g(t)$ (gas pressure) and $\rho(R(t),t)$ are related through the constitutive relation
\begin{align}
p(t) &= \Rg T_\infty \rho(R(t),t),\quad t>0,
\label{eq-bv-3.14prime}
\end{align}
\end{proposition}

Note that \eqref{eq-bv-3.16prime} and \eqref{eq-bv-3.14prime} imply
\begin{equation} p(t)-p_\infty(t) - \frac{2\si}{R(t)}
- 4\mu_l \frac{\dot R}R
 = \rho_l\bke{R(t)\ddot R(t) + \frac32 (\dot R(t))^2} .
\label{pjump}\end{equation}
The system \eqref{eq-bv-3.10prime}--\eqref{eq-bv-3.16prime} depends on the
\begin{equation}
\textrm{ physical parameters: $\ka=\ka_g$, $\Rg$, $\ga$, $c_v$, $\si$.}\label{params-approx-spherical}
\end{equation}
The proof of Proposition \ref{prop:reduction} is given in Appendix \ref{sec-reduction}. The calculations also yields the following expressions for all state variables:

\begin{proposition}\label{red-ext}
Denote the radial components
 of the gas and liquid velocities by $v_g(r,t)$ and $v_l(r,t)$, respectively.
Given a solution $(\rho(r,t),R(t))$ to the system \eqref{eq-bv-3.10prime}--\eqref{eq-bv-3.16prime}, we can reconstruct a spherically symmetric solution $(v_{l},p_l,\rho_g,v_{g}, p_g, T_g, s)$ to the system 
\eqref{eq1.1simplified}--\eqref{eq1.3simplified} in terms of $\rho$ and $R$ by 
\EQ{\label{eq-reconstruct}
\arraycolsep=1.4pt\def\arraystretch{2.2}
\begin{array}{rll}
\Om(t) &= B_{R(t)},&\quad t>0,\\
\rho_g(r,t) &= \rho(r,t),&\quad
 0\le r\le R(t),\ t>0,\\
p_g(t) &= \Rg T_\infty \rho(R(t),t) ,&\quad t>0,\\
v_{g}(r,t) &= \dfrac{\ka}{\ga c_v} \pd_r\bke{\dfrac1{\rho(r,t)}} - \dfrac{\pd_tp_g(t)}{p_g(t)}\dfrac{r}{3\ga},&\quad 0\le r\le R(t),\ t>0,\\
T_g(r,t) &= \dfrac{p_g(t)}{\Rg  \rho(r,t)},&\quad 0\le r\le R(t),\ t>0,\\
s(r,t) &= c_v \log\bke{\dfrac{p_g(t)}{(\rho(r,t))^\ga}},&\quad 0\le r\le R(t),\ t>0,\\
v_{l}(r,t) &= \dfrac{(R(t))^2\dot R(t)}{r^2},&\quad r\ge R(t),\ t>0,\\
p_l(r,t) &= p_\infty(t) + \rho_l \bke{ \dfrac{2R(t) (\dot R(t))^2 + (R(t))^2\ddot R(t)}r - \dfrac{(R(t))^4(\dot R(t))^2}{2r^4} },&\quad r\ge R(t),\ t>0.
\end{array}
}
\end{proposition}  

When $\mu_l=0$ in \eqref{eq-bv-3.16prime}, to study well-posedness  \cite{bv-SIMA2000}, Biro and Vel\'azquez mapped, by a change of variables,  the free boundary problem on $B_{R(t)}$:  \eqref{eq-bv-3.10prime}--\eqref{eq-bv-3.16prime} to a problem on the fixed domain $B_1$ as
\begin{subequations}
\label{red-eqns-fix-domain}
\begin{align}
\pd_t\overline\rho(y,t) &= \frac{\ka}{\ga c_v} \frac1{R^2} \De_y\log\overline\rho(y,t) + \frac{1}{\ga} \frac{\pd_t p(t)}{ p(t)}\Big(  \frac13 y \pd_y\overline\rho(y,t) 
+ \overline\rho(y,t) \Big),\quad 0\le y\le 1,\ t>0, \label{eq-bv-3.10prime-fix-domain}\\
\dot R(t) &= -\frac{\ka}{\ga c_v} \frac1R \frac{\pd_y\overline\rho(1,t)}{(\overline\rho(1,t))^2} - \frac{R(t)}{3\ga}\frac{\pd_tp(t)}{p(t)},\quad t>0,\label{eq-bv-3.15prime-fix-domain}\\
\overline\rho(1,t) &= \frac1{\Rg T_\infty} \bkt{p_\infty(t) + \frac{2\si}{R(t)} 
+ 4\mu_l \frac{\dot R}R
+ \rho_l\bke{R(t)\ddot R(t) + \frac32 (\dot R(t))^2} },\quad t>0. \label{eq-bv-3.16prime-fix-domain}
\end{align}
\end{subequations}
In this setting they proved the local well-posedness for the free boundary problem \eqref{eq-bv-3.10prime}--\eqref{eq-bv-3.16prime}.  The proof is based on the  derivation of a priori Schauder estimates, application of a  Leray--Schauder fixed point argument and the classical regularity theory for quasilinear parabolic equations; see, for example,  \cite[Chapter V. Theorem 6.1]{LSU-book1967}. We extend their result to the general case involving liquid viscosity on the free boundary, which implies local well-posedness of the system \eqref{eq1.1simplified}--\eqref{eq1.3simplified} in the spherically symmetric case.

\begin{theorem}[Local in time well-posedness]\label{thm-bv-3.1}
Consider the initial value problem for  \eqref{eq-bv-3.10prime}--\eqref{eq-bv-3.16prime}  with initial radius $R(0)=R_0>0$. and initial density $\rho_0\in C^{2+2\al}([0,R(0)])$, $0<\al<\frac12$. Suppose also that for some $\eta>0$, $\rho_0(r)\ge\eta$ for  $0\le r\le R_0$. Then, there exists $\de=\de(\norm{\rho_0}_{C^{2+2\al}})$ such that the free boundary problem \eqref{eq-bv-3.10prime}--\eqref{eq-bv-3.16prime} has a unique solution satisfying 
\EQN{
\rho&\in C^{1+\al}_t([0,\de]; C^{2+2\al}_r([0,R(t)))),\\
R&\in C^{3+\al}[0,\de].
}
\end{theorem}
\begin{proof}
The proof is essentially the same as the proof for the case when $\mu_l=0$ in \cite[Theorem 3.1]{bv-SIMA2000}.
The only difference is that, for the case when $\mu>0$, an extra viscous term $4\mu_l\dot R/R$ needs to be added to the right hand side of \cite[(3.18)]{bv-SIMA2000}.
Since $4\mu_l\dot R/R$ is analytic in $R$ and $\dot R$ for $R\neq0$, one can follow the same procedure in the proof of \cite[Theorem 3.1]{bv-SIMA2000}--deriving an a priori Schauder estimates and applying Leray--Schauder fixed point theorem along with the regularity theory for quasilinear parabolic equations--to conclude the same local well-posedness result.
We omit the proof and refer the reader to \cite{bv-SIMA2000}.
\end{proof}

\section{Dynamic stability of spherical bubble}\label{sec:stab}

In Section \ref{sec:LS} we recall the results in    \cite{bv-SIMA2000}  on conditional Lyapunov stability of spherical symmetric equilibria, that is Lyapunov stability relative to small perturbations  of the same bubble mass.
We then, in Theorem \ref{thm-bv-4.1-anyperturb},  extend this result to Lyapunov stability relative to \underline{arbitrary} small perturbations. Then, in Section \ref{sec:main-theo} we state Theorem \ref{thm-asystab}, the result of asymptotic stability.  The proof is presented in subsequent sections.

\subsection{Lyapunov stability}\label{sec:LS}
In \cite[Theorem 4.1]{bv-SIMA2000}, Biro and Vel\'azquez established the global well-posedness of the free boundary problem \eqref{eq-bv-3.10prime}--\eqref{eq-bv-3.16prime}, $\mu_l=0$ in \eqref{eq-bv-3.16prime},
when the initial data is sufficiently close to a given spherically symmetric equilibrium and has the same mass as the mass of the equilibrium solution.

When $\mu_l>0$ in \eqref{eq-bv-3.16prime}
the extra viscous term on the boundary leads to the extra term: $- 16\pi\mu_l R(t) (\dot R(t))^2$ on the right hand side of the energy dissipation law \eqref{eq-bv-4.16}. 
Hence, the key bound \cite[(4.41)]{bv-SIMA2000} still holds, and thus, their proof also applies.
In other words, we have that the spherical equilibrium are Lyapunov stable relative to mass preserving perturbations.  Introduce the norm 
\begin{align}
&\oldnorm{\left(\rho_1(\cdot,t)-\rho_2(\cdot,t),R_1(t)-R_2(t),\dot R_1(t)-\dot R_2(t)\right)} \notag\\
&\quad  \equiv  \| \overline\rho_1(\cdot,t)- \overline\rho_2(\cdot,t) \|_{C^{2+2\al}_y(B_1)} + |R_1(t)-R_2(t)| + \abs{\dot R_1(t)-\dot R_2(t)},\quad
\overline \rho_i(y,t) = \rho_i(R_i(t)y,t),\ i=1,2.
\label{onorm-def}\end{align}
In \cite{bv-SIMA2000} it is shown that given $\varepsilon_0>0$, there exist $\eta_0=\eta_0(\varepsilon_0)>0$ such that
\EQ{\label{eq-bv-4.28}
\oldnorm{\left(\rho_0-\rho_*[M_0],R_0-R_*[M_0],\dot R_0\right)}\le\eta_0,}
where $M_0 = \textrm{Mass}[\rho_0,R_0]$,  
then for all $t>0$
\EQ{\label{eq-bv-4.30}
\oldnorm{\left(\rho(\cdot,t)-\rho_*[M_0],R(t)-R_*[M_0],\dot R(t)\right)}\le\ve_0.
}
\begin{remark}
We note that the smallness of initial radial velocity, $\dot R(0)$, is not explicitly assumed in \cite[Theorem 4.1]{bv-SIMA2000}. 
The smallness is needed to control the kinetic energy, $KE_l$, and higher derivatives of $R$.
\end{remark}

\begin{remark}
The proof of \cite[Theorem 4.1]{bv-SIMA2000} gives a better regularity and control than $|R(t) - R_*|\le \ve_0$ and $|\dot R(t)|\le\ve_0$ which was stated in \cite[Theorem 4.1]{bv-SIMA2000}, . 
In fact, we obtain \EQ{\label{eq-higher-regularity-R}
\norm{R - R_*}_{C^{3+\al}_t(\R_+)}\le\ve_0.
}
Note further from \eqref{eq-bv-3.14prime}, that 
\EQ{\label{eq-higher-regularity-p}
\norm{p_g}_{C^{1+\al}_t(\R_+)}\le\ve_0.
}
\end{remark}

Using the continuity of functionals, 
we now extend the conditional Lyapunov stability result \cite[Theorem 4.1]{bv-SIMA2000}  to Lyapunov stability relative to arbitrary small perturbations. 
Specifically, we prove the Lyapunov stability of the manifold of equilibria to the system \eqref{eq-bv-3.10prime}--\eqref{eq-bv-3.16prime}
\EQ{\label{eq-equilib-manifold}
\mathcal{M}_*= \bket{(\rho_*[M],R_*[M], \dot R_*=0): 0<M<\infty },
}
where $\rho_*[M]$, $R_*[M]$ are given in Proposition \ref{prop-equilib}.

Introduce the distance of the state defined by $(\rho(\cdot,t),R(t),\dot R(t))$  to the manifold of equilibria:
\EQN{
{\rm dist}( (\rho(\cdot,t),R(t),\dot R(t)), \mathcal{M}_*)
&\equiv \inf \bket{ \oldnorm{\left(\rho(\cdot,t)-\rho_*,R(t)-R_*,\dot R(t)-\dot R_*\right)}: (\rho_*,R_*,\dot R_*)\in \mathcal{M}_*}\\
&=\inf_{0<M<\infty}
\oldnorm{ \left(\rho(\cdot,t)-\rho_*[M],R(t)-R_*[M],\dot R(t)\right)}.
}

\begin{theorem}[Lyapunov stability]\label{thm-bv-4.1-anyperturb}
Consider the time evolution equation 
\eqref{eq-bv-3.10prime}--\eqref{eq-bv-3.16prime} with $p_\infty(t) \equiv p_{\infty,*}$.
Let $\varepsilon_0>0$ be arbitrary. There exists $\eta_0>0$ such if the
initial data $\rho_0(r), R_0$, and $\dot R_0$ satisfies 
\[
{\rm dist}((\rho_0,R_0,\dot R_0), \mathcal{M}_*)\le\eta_0,
\]
then $(\rho(r,t), R(t))$,  the global in time solution of the initial value problem \eqref{eq-bv-3.10prime}--\eqref{eq-bv-3.16prime},  satisfies 
\[
{\rm dist}((\rho(\cdot,t),R(t),\dot R(t)), \mathcal{M}_*)\le\ve_0,\quad\textrm{for all}\quad  t>0.
\]
\end{theorem}
\begin{proof}
The proof is a consequence of \cite[Theorem 4.1]{bv-SIMA2000} and  Proposition \ref{prop:contin}. 
Namely, assuming $(\rho_*[M_*], R_*[M_*], \dot{R}_* = 0)\in \mathcal M_*$ is close to $(\rho_0,R_0,\dot{R}_0)$,
Proposition \ref{prop:contin}
 implies that there is a unique $( \rho_*[M_0], R_*[M_0])$ such that $\textrm{Mass}[ \rho_*[M_0], R_*[M_0]] = M_0 = \textrm{Mass}[\rho_0, R_0]$,
  and 
\EQ{
  | R_*[M_0]-R_*[M_*]|  + | \rho_*[M_0] - \rho_*[M_*] |  
  &\le C \left(|R_0-R_*[M_*]|\ +\ \|\rho_0-\rho_*[M_*]\|_{L^\infty(B_{R_0})}\right)\\
  &\le C \left(|R_0-R_*[M_*]|\ +\ \|\rho_0-\rho_*[M_*]\|_{C_r^{2+2\al}(B_{R_0})}\right).
  }
Hence, 
\begin{align*}
\|\rho_0 - \rho_*[M_0]&\|_{C_r^{2+2\alpha}} + |R_0- R_*[M_0]|\\
&= \|\rho_0-\rho_*[M_*]+\rho_*[M_*]- \rho_*[M_0]\|_{C_r^{2+2\alpha}} + |R_0-R_*[M_*]+R_*[M_*]- R_*[M_0]|\\
&\le  \|\rho_0-\rho_*[M_*]\|_{C_r^{2+2\alpha}} + |\rho_*[M_*]- \rho_*[M_0]| + |R_0-R_*[M_*]| + |R_*[M_*]- R_*[M_0]|\\
&\le C^\prime\left( \|\rho_0-\rho_*[M_*]\|_{C_r^{2+2\alpha}} + |R_0-R_*[M_*]|  \right).
\end{align*}
Therefore, choosing $\|\rho_0-\rho_*[M_*]\|_{C_r^{2+2\alpha}} + |R_0-R_*[M_*]| $ and $\dot{R}_0$ sufficiently small 
we conclude from 
\cite[Theorem 4.1]{bv-SIMA2000} (\eqref{eq-bv-4.28} implies \eqref{eq-bv-4.30})
that 
\[
{\rm dist}( (\rho(\cdot, t),R(t),\dot R(t)), \mathcal{M}_*)
\le \norm{\rho(\cdot,t)- \rho_*[M_0]}_{C^{2+2\al}_r} + |R(t) - R_*[M_0]| + |\dot R(t)| \le \ve_0.
\]
This completes the proof.
\end{proof}

The proof of the Lyapunov stability in \cite[Theorem 4.1]{bv-SIMA2000} relies on a coercive energy estimate (\cite[Lemma 4.2]{bv-SIMA2000}), for the case of constant external far-field pressure.
In Appendix \ref{sec-bv-revisit} we prove an extension of this energy estimate, Theorem \ref{thm-XYZ}, which enables us to generalize Theorem \ref{thm-bv-4.1-anyperturb}:
\begin{corollary}\label{cor:rem-lyapunov-general-pinf} 
 The conclusions of Theorem \ref{thm-bv-4.1-anyperturb} hold provided we choose $\eta_0$ sufficiently small and so that the following additional conditions are satisfied:
  \[
|p_\infty(t) - p_{\infty,*} | \le \eta_0,\qquad  
\norm{\pd_t p_\infty}_{L^1_t(\R_+)} \le\eta_0.\quad 
\]
\end{corollary}

\subsection{Nonlinear asymptotic stability with no decay rate}\label{sec:main-theo}

  The first goal of this article is to study the asymptotic stability of  the family of spherically symmetric equilibria of the approximate system \eqref{eq1.1simplified}--\eqref{eq1.3simplified} against small spherically symmetric perturbations. 
This is a consequence of the following result on asymptotic stability for the reduced system \eqref{eq-bv-3.10prime}--\eqref{eq-bv-3.16prime}.

\begin{theorem}[Asymptotic stability of the manifold of spherically symmetric equilibria]\label{thm-asystab}
Fix parameters \eqref{params-approx-spherical} and set $p_\infty(t)=p_{\infty,*}$ in the system \eqref{eq-bv-3.10prime}--\eqref{eq-bv-3.16prime}.
\begin{enumerate}
\item There exist a constant $\eta>0$ such that if ${\rm dist}((\rho_0,R_0,\dot R_0), \mathcal{M}_*)\le\eta$, then 
\[
{\rm dist}((\rho(\cdot,t),R(t),\dot R(t)), \mathcal{M}_*)\to0\ \text{ as }t\to+\infty.
\]
\item More precisely,  let $M_*>0$ and  $(\rho_*[M_*],R_*[M_*])$ as Proposition \ref{thm-asystab-1}, there exist constants $\eta_1>0$ and $C_1>0$ such that the following holds  for all $0<\eta\le\eta_1$: 
Consider initial data $\rho_0(r)$, $R_0$, $\dot R_0$, an \underline{arbitrary}
small perturbation of $(\rho_*[M_*],R_*[M_*],\dot R_* = 0)$:
\begin{equation} 
\oldnorm{\left(\rho_0-\rho_*[M_*],R_0-R_*[M_*],\dot R_0\right)}\le\eta.\label{data-small}
 \end{equation}
 Let $M_0=\int_{B_{R_0}}\rho_0$ denote the initial bubble mass. In general, $M_0\ne M_*$, however by Proposition \ref{prop:contin} 
  the corresponding points on the manifold of equilibria are close:
\[ |\rho_*[M_0] - \rho_*[M_*]|\le C_1\eta\quad  {\rm and}\quad |R_*[M_0] - R_*[M_*]|\le C_1\eta.\]
Let $(\rho(r,t),R(t))$ denote the global in time solution of the free boundary problem \eqref{eq-bv-3.10prime}--\eqref{eq-bv-3.16prime} with  initial data satisfying \eqref{data-small}. Then, as $t\to+\infty$
\begin{equation}
   {\rm dist}((\rho(\cdot,t),R(t),\dot R(t)), (\rho_*[M_0], R_*[M_0],\dot R_*=0))\to0,
  \label{a-stab2}\end{equation}
  and $|\ddot R(t)| + |\dddot R(t)|\to0$ as $t\to+\infty$. 
  \item The convergence of $(\rho,R)$ in \eqref{a-stab2} is sufficient to imply the convergence of the quantities in Proposition \ref{red-ext} to their equilibrium values. Therefore, the spherical equilibrium \eqref{eq-equilibrium} of the system \eqref{eq1.1simplified}--\eqref{eq1.3simplified} is asymptotically stable. 
\end{enumerate}
\end{theorem}

It is simple to generalize Theorem \ref{thm-asystab}, the asymptotic stability for the model of constant external far-field pressure $p_\infty$, to the following result for the case that $p_\infty(t)$ is a small perturbation of a constant.

\begin{corollary}\label{cor:rem-asystab-general-pinf}
The conclusions of Theorem \ref{thm-asystab} hold provided we choose
the constant $\eta_0>0$ sufficiently small and such the following conditions on the asymptotically constant far-field pressure hold:
\EQ{\label{eq-pinf-condition}
p_\infty \in C^{1+\al}_t(\R_+),\quad
|p_\infty(t) - p_{\infty,*} | + \norm{\pd_tp_\infty}_{L^1_t(\R_+)}\le \eta_0,\quad
p_\infty(t) \to p_{\infty,*} \text{ as } t\to\infty.
}
\end{corollary}

\medskip\noindent{\bf Strategy of the proof of Theorem \ref{thm-asystab}.} By a continuity argument, the  proof of asymptotic stability relative to arbitrary small perturbations, can be reduced to Proposition \ref{thm-asystab-1} on asymptotic stability relative to perturbations of a spherical equilibrium which have the same bubble mass.  At the heart of the proof of Proposition \ref{thm-asystab-1}  is
\begin{enumerate}
\item  the time-integrability over $[0,\infty)$ of the energy dissipation rate:
 $\int_{B_{R(t)}} |T_g(\cdot,t)|^{-2} |\nabla T_g(\cdot,t)|^2$ and 
 \item the coercive
energy estimate, Theorem \ref{thm-XYZ}, which expresses that the spherical equilibrium of an arbitrary specified mass is a local minimizer of the total energy relative to spherically symmetric perturbations of the same mass; see \cite[Lemma 4.2]{bv-SIMA2000}.
\end{enumerate} By 
the equation of state, \eqref{eq1.2simplified-d}, $\int_{B_{R(t)}} |\rho_g(\cdot,t)|^{-2} |\nabla \rho_g(\cdot,t)|^2$ is time-integrable over $[0,\infty)$. This implies
convergence of $\rho(\cdot,t)$,  to an equilbrium density. With control of the density, $\rho(r,t)$,
 and in particular $\rho(R(t),t)$, we obtain that $\dot R(t)\to0$ as $t\to\infty$, from \eqref{eq-bv-3.15prime}, the equation 
 for the motion of the boundary. The  limiting constant values of $\rho_g$ and $R$ satisfy the system 
 \eqref{eq-steadystate} and it follows that these correspond to the unique spherically symmetric equilibrium of the given initial mass. 
We present the detailed proofs in Section \ref{sec-asystab}.

\subsection{Main result: exponential rate of convergence toward the manifold of equilibria}

The ultimate goal of this article is to show that, within the approximate system \eqref{eq1.1simplified}--\eqref{eq1.3simplified}, the manifold of the spherically symmetric equilibria, given in \eqref{eq-equilibrium} and parametrized by the bubble mass $M$ is nonlinearly and exponentially asymptotically stable with respect to small spherically symmetric perturbations.

\begin{theorem}\label{thm-nonlinear-exp-decay}
Assume constant pressure at infinity,  $p_\infty(t)\equiv p_{\infty,*}$. Then, 
the manifold of the equilibria $\mathcal M_*$ (defined in \eqref{eq-equilib-manifold}) of the free boundary problem \eqref{red-eqns} is (locally) nonlinearly exponentially stable.
Specifically,  there exist a constant $\eta>0$ such that if $\dist\bke{(\rho_0(\cdot),R(t),\dot R(t)), \mathcal M_*} \le\eta$, then for some $\bar\be>0$
\[
\dist\bke{(\rho(\cdot,t),R(t),\dot R(t)), \mathcal M_*} = O \bke{ e^{-\bar\be t}}\ \text{ as }t\to+\infty.
\]
\end{theorem}

\medskip\noindent{\bf Strategy of the proof of Theorem \ref{thm-nonlinear-exp-decay}.}
The detailed proof of Theorem \ref{thm-nonlinear-exp-decay} is presented in Section \ref{sec-nonlinear} and makes use of ideas from  center manifold theory; see, {\it e.g.} \cite{Carr-book1981}. First, we rewrite the quasi-linear parabolic partial differential equation in the free boundary problem \eqref{red-eqns} as an infinite-dimensional dynamical system by means of Dirichlet eigenfunction decomposition in Proposition \ref{prop-equivalent-systems}.
 The linearized operator has a neutral direction (zero eigenvalue)
 associated with the manifold of equilibrium and all its other spectrum is contained in the open left half plane, 
 and bounded away from the imaginary axis; Proposition \ref{prop-spectrum}, Appendix \ref{sec-negative-upper-bound}.  Subject to a codimension one constraint, the linearized flow satisfies an exponential time-decay bound; Proposition \ref{prop-invariance-decay}. 
Next, we develop a center manifold analysis to obtain an exponential rate of convergence of the solution toward the center manifold
 of spherically symmetric equilibria.  

We note that an obstacle to applying the standard center manifold framework to our stability problem  is that our infinite dimensional dynamical system \eqref{eq-carr-6.3.1-simple-form} is quasilinear and autonomous:
 $\dot{\bf w} = \mathcal L {\bf w} + \mathcal N^1({\bf w})\dot{\bf w} + \mathcal N^0({\bf w})$ (or \eqref{red-eqns} equivalently). 
 The system  is nonlinear, and  linear in $\dot{\bf w}$. We are unaware of a center manifold formulation admitting direct application to our 
 free boundary problem. Hence in Appendix \ref{sec-center-manifold}, we develop an applicable general approach for a class of fully nonlinear autonomous systems equipped with ``weak'' {\it a priori} bounds; see Proposition \ref{prop-stab-center-mfd} .
The {\it a priori} information on the regularity and weak time-decay of $\dot{\bf w}$ 
(Theorem \ref{thm-asystab}) enables us, via Proposition \ref{prop-stab-center-mfd}, to  prove exponential convergence to a center manifold.

\section{Conservation of mass and energy dissipation}\label{sec:conservation}

Solutions of the free boundary problem \eqref{eq1.1simplified}--\eqref{eq1.3simplified} satisfy conservation of mass
 and an energy dissipation law. These play a central role in the Lyapunov stability theory of \cite{bv-SIMA2000}
 and in our asymptotic stability theory. To derive these statements one makes use of the following technical results.

Given a smooth velocity field, ${\bf v}:(x,t)\mapsto {\bf v}(x,t)\in\mathbb{R}^3$, define $X^t(\alpha)$ to be the particle trajectory map given by the solution of the initial value problem: $\dot{X}^t(\alpha) = {\bf v}\left(X^t(\alpha),t\right)$, $X^0(\alpha)=\alpha$. The mapping $\alpha\mapsto X^t(\alpha)$ is smooth and invertible for all $t$ sufficiently small. For an open subset 
$\Omega\subset\mathbb{R}^2$, let $X^t(\Omega)=\{X^t(\alpha): \alpha\in\Omega\}$.
We first recall the transport formula that gives the rate of change of a function in a domain transported with the fluid.

\begin{proposition}[$\text{\cite[Proposition 1.3]{MB-book2002}}$]\label{prop-MB-1.3}
Let $\Om$ denote  an open, bounded domain with a smooth boundary.  Then for any smooth function $f(x,t)$,
\[
\frac{d}{dt} \int_{X^t(\Om)} f\, dx = \int_{X^t(\Om)} \bkt{\pd_t f + \div(f{\bf v})} dx.
\]
\end{proposition}

With the aid of this transport formula, we have the following lemma which is used to compute the time-evolution of the mass and energies on a time-varying spatial domain, $\Omega(t)$.

\begin{lemma}\label{lem-bv-append-B-nonspherical}
If $\pd_t\rho + \div(\rho{\bf v})=0$ in $\Om(t)$ and ${\bf v}(\boldsymbol{\om}(t),t) = \dot{\boldsymbol{\om}}(t)$ for $\boldsymbol{\om}(t)\in\pd\Om(t)$, then
\[
\frac{d}{dt} \int_{\Om(t)} \rho\phi\, dx
 = \int_{\Om(t)} \rho \frac{D\phi}{Dt}\, dx
\]
for any smooth function $\phi$, where $\frac{D}{Dt}$ is the material derivative of $f$ given by
\[
\frac{Df}{Dt} = \pd_tf + {\bf v}\cdot\nb f.
\]
\end{lemma}

\begin{proof}[Proof of Lemma \ref{lem-bv-append-B-nonspherical}]
Since the boundary $\pd\Om(t)$ moves along the particle-trajectory mapping $X$ of the velocity field ${\bf v}$, $\Om(t)=X(\Om(0),t)$.
By the transport formula Proposition \ref{prop-MB-1.3},
\[
\frac{d}{dt} \int_{\Om(t)} \rho\phi\, dx = \int_{\Om(t)} \bkt{\pd_t(\rho\phi) + \div_x(\rho\phi{\bf v})}\, dx.
\]
Using the continuity equation, the right hand side of above equation becomes
\[
\int_{\Om(t)} \bkt{\rho \pd_t\phi + \rho{\bf v}\cdot\nb\phi} dx.
\]
This proves the lemma.
\end{proof}

\subsection{Conservation of mass}

By taking $\phi \equiv 1$ in Lemma \ref{lem-bv-append-B-nonspherical} we have
\begin{proposition}[Bubble mass conservation]\label{mass-cons}
Let $X^t$ be the particle-trajectory mapping associated with ${\bf v}_g$.
Denote by $\Omega_0\subset\mathbb{R}^3$ the bubble region at time $t=0$, assumed to have a smooth boundary, and $\Omega(t)=X^t(\Omega_0)$.
Let $\rho=\rho_g$ denote a $C^{1+\al}_t([0,\infty); C^{2+2\al}_x(\Om(t)))$, $0<\al<\frac12$ solution  of \eqref{eq-1.2-a} (or equivalently \eqref{eq1.2simplified-a}) with ${\bf v}_g$ satisfying \eqref{eq-1.3-a} (or equivalently \eqref{eq1.3simplified-a}) 
and initial data $\rho_0\in C^{2+2\al}(\Om_0)$.
 Then, the mass of the bubble is constant in time:
 \EQ{\label{eq-mass-preserve}
\int_{\Om(t)} \rho(x,t)\, dx = \int_{\Om_0} \rho_0(x)\, dx \quad t>0.
}
\end{proposition}

\subsection{Energy dissipation law}

\begin{definition}[The total energy]\label{def:total-en}
Consider the case of spherically symmetric solutions of \eqref{eq1.1simplified}--\eqref{eq1.3simplified}. 
The total energy of the system  is given by
\begin{equation}
\label{total-en}
\mathcal{E}_{\rm total}(t) = FE(t) + KE_l(t) + U_{g-l}(t) + PV_{p_\infty}(t),\\
\end{equation}
where the  total energy is made up of the following components:
\begin{enumerate}
\item $FE(t)$, the Helmholtz free energy:
\EQ{\label{eq-helmholtz}
FE(t) 
&= c_v\int_{B_R} \rho_g T_g\, dx - T_\infty \int_{B_R} \rho_g s\,dx\\
&= \frac{4\pi c_v}{3\Rg } p_gR^3 - T_\infty \int_{B_R} \rho_g s\,dx\\
&= \frac{4\pi c_v}{3\Rg } p_gR^3 - c_v T_\infty M_0 \log p_g + c_v\ga T_\infty \int_{B_R} \rho_g\log\rho_g\, dx,
}
where $M_0 = \textrm{Mass}[\rho_g,R]$, and the second and the last equalities hold by \eqref{eq1.2simplified-d} and \eqref{eq1.2simplified-e}, respectively.
\item $KE_l(t) $, the kinetic energy of the liquid:
\[
KE_l(t) = \frac12 \int_{\R^3\setminus B_{R(t)}} \rho_l |{\bf v}_l|^2\, dx
= 2\pi \rho_l [R(t)]^3\ [\dot R(t)]^2,\quad 
\]
\item $U_{g-l}(t)$, the surface energy of the liquid--gas interface
\[
U_{g-l}(t) = \si \int_{\pd B_{R(t)}}\, dS
= 4\pi \si\ [R(t)]^2,
\] 
\item $PV_{p_\infty}(t)$, the energy contributed by the work done by the external sound field. 
\[
PV_{p_\infty}(t) = |B_{R(t)}|\, p_\infty(t) = \frac{4\pi}3 [R(t)]^3 p_\infty(t).
\]
\end{enumerate}
\end{definition}

The energy functional \eqref{total-en} is at the heart of the stability analysis. Its importance is clear from the following
result on energy dissipation, proved  in \cite{bv-SIMA2000} (using Lemma \ref{lem-bv-append-B-nonspherical}) for the system \eqref{red-eqns} (equivalently \eqref{eq1.1simplified}--\eqref{eq1.3simplified} under the assumption of spherical symmetry) and $p_\infty=1$. We state and prove  a mild generalization to the case of a time-dependent  pressure at $p_\infty(t)$.
We shall use the abbreviated notation: $\rho=\rho_g, p=p_g, T=T_g$ and $\ka=\ka_g$.

\begin{proposition}[Energy dissipation law]\label{en-diss}
Assume that $\left(\rho(r,t), R(t), p(t)=\Rg T(r,t)\rho(r,t)\right)$ is a solution of \eqref{red-eqns}, or equivalently \eqref{eq1.1simplified}--\eqref{eq1.3simplified} under the assumption of spherical symmetry.  

Then, 
\EQ{\label{eq-bv-4.16}
\frac{d}{dt} \mathcal{E}_{\rm total}(t) =  - \ka T_\infty \int_{B_{R(t)}} \frac{|\nb_r T(|x|,t)|^2}{T^2(|x|,t)}\, dx 
- 16\pi\mu_l R(t) (\dot R(t))^2
+ \frac{4\pi}3 R^3(t) \pd_tp_\infty(t).
}
\end{proposition}

\begin{proof}
By Lemma \ref{lem-bv-append-B-nonspherical}, differentiating $\mathcal{E}_{\rm total}$ with respect to $t$ using the third line of \eqref{eq-helmholtz} yields
\EQ{\label{eq-en-diss-pf1}
\frac{d}{dt} \mathcal{E}_{\rm total} 
&= \frac{4\pi c_v}{3\Rg}\bke{\pd_tp_gR^3 + 3p_gR^2\dot R} - T_\infty \int_{B_R} \rho_g \frac{Ds}{Dt} \, dx\\
&\quad +4\pi\rho_lR^2\dot R\left(R\ddot R + \frac32\dot R^2\right) +  8\pi\si R\dot R + 4\pi R^2\dot R\ p_\infty(t) + \frac{4\pi}3 R^3 \pd_tp_\infty.
}
Consider the second term on the right hand side of \eqref{eq-en-diss-pf1}.
 Using \eqref{eq1.2simplified-c}, integrating by parts, and  \eqref{eq1.3simplified-c}, we obtain
\begin{align}
\int_{B_R} \rho_g \frac{Ds}{Dt} \, dx 
= \ka_g \int_{B_R} \frac{\De_r T_g}{T_g} \,dx
&=  \ka_g \int_{B_R} \frac{|\nb_r T_g|^2}{T_g^2}\, dx + \frac{4\pi  \ka_g R^2}{T_\infty} \pd_rT_g(R(t),t)\nonumber\\
&= \ka_g \int_{B_R} \frac{|\nb_r T_g|^2}{T_g^2}\, dx - \frac{4\pi  \ka_g R^2 p_g(t)}{\Rg T_\infty} \frac{\partial_r\rho_g(R(t),t)}{\rho^2_g(R(t),t)},
\label{eq-en-diss-pf2}\end{align}
where the last equality follows from the constitutive relation $T_g= p_g (\Rg \rho_g)^{-1}$.

For the third term on the right hand side of \eqref{eq-en-diss-pf1} we use \eqref{pjump}
\begin{equation}
\rho_l\bke{R(t)\ddot R(t) + \frac32 (\dot R(t))^2} = p_g(t)-p_\infty(t) - \frac{2\si}{R(t)}
-4\mu_l \frac{\dot R}R.  
\label{pjump1}\end{equation}
Substituting \eqref{eq-en-diss-pf2} and \eqref{pjump1} into \eqref{eq-en-diss-pf1} we obtain
\EQ{\label{eq-en-diss-pf1A}
\frac{d}{dt} \mathcal{E}_{\rm total} 
&= \frac{4\pi c_v}{3\Rg}\ \pd_tp_gR^3 + 4\pi\left(\frac{c_v}{\Rg}+1\right)p_gR^2\dot R - T_\infty \ka_g \int_{B_R} \frac{|\nb_r T_g|^2}{T_g^2}\, dx \\
&+ 4\pi  \ka_g R^2 p_g(t)\ \frac{1}{\Rg} \frac{\partial_r\rho_g(R(t),t)}{\rho^2_g(R(t),t)} 
- 16\pi\mu_l R(t) (\dot R(t))^2
+   \frac{4\pi}3 R^3 \pd_tp_\infty.
}
or
\EQ{\label{eq-en-diss-pf1B}
\frac{d}{dt} \mathcal{E}_{\rm total} 
&= - T_\infty \ka_g \int_{B_R} \frac{|\nb_r T_g|^2}{T_g^2}\, dx 
- 16\pi\mu_l R(t) (\dot R(t))^2
+   \frac{4\pi}3 R^3 \pd_tp_\infty\\
&+ 4\pi R^2 p_g\Big[ \frac{c_v}{\Rg} \frac13 \frac{\pd_t p_g}{p_g} R + \left(1+\frac{c_v}{\Rg}\right)\dot R + \frac{\ka_g}{\Rg} \frac{\pd_r\rho_g}{\rho_g^2}\Big]  .
}
Finally we claim that  the expression in the square brackets in \eqref{eq-en-diss-pf1B} vanishes:
\begin{equation}  \mathcal{I}(r,t)\equiv \frac{c_v}{\Rg} \frac13 \frac{\pd_t p_g}{p_g} R + \left(1+\frac{c_v}{\Rg}\right)\dot R + \frac{\ka_g}{\Rg} \frac{\pd_r\rho_g}{\rho_g^2}=0, \label{claim}
\end{equation}
from which Proposition \ref{en-diss} follows. 
To prove \eqref{claim} note that the relation  $\gamma= 1 + \frac{\Rg}{c_v}$ (see \eqref{eq-def-gamma})  implies 
\begin{align}
\mathcal{I}(r,t) &=\frac{1}{\gamma-1} \frac13 \frac{\pd_t p_g}{p_g} R + \frac{\gamma}{\gamma-1}\dot R +
 \frac{\ka_g}{\Rg} \frac{\pd_r\rho_g}{\rho_g} \nonumber  \\
&= \frac{\gamma}{\gamma-1}\left(  \frac{1}{3\gamma} \frac{\pd_t p_g}{p_g} R +\dot R +
 \frac{\ka_g}{\Rg} \frac{\gamma-1}{\gamma} \frac{\pd_r\rho_g}{\rho_g^2} \right) \label{eq-en-diss-pf1C}
\end{align}
Next, we use \eqref{eq-bv-3.15prime} to simplify \eqref{eq-en-diss-pf1C}. This yields
\begin{align}
\mathcal{I}(r,t) & =
\frac{\ka_g}{\gamma-1} \frac{1}{\Rg}\left( -\frac{\Rg}{ c_v}    +
 \gamma-1  \right)\frac{\pd_r\rho_g}{\rho_g^2}=0 \label{eq-en-diss-pf1C}
 \end{align}
 The proof of Proposition \ref{en-diss} is now complete.
\end{proof}

\subsection{Coercivity energy estimate}
To prove the global existence of solutions and Lyapunov stability, the authors in \cite{bv-SIMA2000} considered the energy $\mathcal{E}_{\rm total}$ defined in \eqref{total-en} for $p_\infty(t)\equiv 1$, and used the energy dissipation formula \eqref{eq-bv-4.16}. 
By expanding the energy $\mathcal{E}_{\rm total}$ at the steady state energy up to quadratic terms, 
they derived the coercivity estimate of the perturbed energy from the steady state energy in \cite[Lemma 4.2]{bv-SIMA2000}.
We generalize \cite[Lemma 4.2]{bv-SIMA2000} to the following result for the case of general (nonstationary) external far-field pressure $p_\infty(t)$ whose proof is similar to that of \cite[Lemma 4.2]{bv-SIMA2000} and is in Appendix \ref{sec-bv-revisit} for the reader's convenience.

\begin{theorem}\label{thm-XYZ}
Given positive constants $\rho_*, R_*, p_{\infty,*}$.
Assume that there exists a constant $\nu>1$ such that 
\EQ{\label{eq-bv-4.32}
\nu^{-1} \le \rho(r) \le \nu,
}
\EQ{\label{eq-bv-4.33}
\nu^{-1} \le R \le \nu,
}
\EQ{\label{eq-bv-4.34}
\textrm{Mass}[\rho,R] = \textrm{Mass}[\rho_*,R_*],
}
where $\textrm{Mass}[\rho,R]$ is given in \eqref{eq-def-masss}.
Let 
\[ \textrm{$\varrho(y)=\rho(Ry)-\rho_*$, $\mathcal R = R - R_*$, $\dot{\mathcal R} = \dot R$, and $\mathcal P_\infty =  p_\infty - p_{\infty,*}$.} \]
Then, 
\begin{enumerate}
\item 
 \begin{align}
\mathcal{E}_{\rm total} - \mathcal{E}_*  &\ge
 \frac{c_v T_\infty R_*^3 |B_1|}{2\rho_*} \left( \varrho(1)- \frac{1}{|B_1|}\int_{B_1}\varrho \right)^2
   + 2\pi \rho_l R_*^3 \dot{\mathcal{R}}^2 +\frac{R_*^3}{\rho_*^2} \left( \frac{p_{\infty,*}}2 + \frac{2\sigma}{3R_*}\right)\int_{B_1} \varrho^2\label{eq-coercive-explicit}\\
  &\quad  - 4\pi R_*^2 |\mathcal P_\infty| |\mathcal R |
- \frac{R_*^3}{4\pi \rho_*^2} |\mathcal P_\infty| |B_1| \int_{B_1} \varrho^2 
 + O\left( |\mathcal{R}|^3 +  | \varrho(1)|^3 +\left(\int_{B_1} | \varrho|\right)^3 \right),
 \nonumber\end{align}
where $\mathcal{E}_*$ is $\mathcal{E}_{\rm total}$ evaluated at $(\rho_*,R_*, \dot R_*=0)$. Furthermore, 
\item there exist constants
$\Th>0$, $\de_0\in(0,1]$ depending only on $\textrm{Mass}[\rho_*,R_*]$, $T_\infty$, and $\nu$ such that 
if $\norm{\rho - \rho_*}_{L^\infty(B_R)} + |p_\infty - p_{\infty,*}|\le \de_0$, then 
\EQ{\label{eq-bv-4.35}
\mathcal{E}_{\rm total} - \mathcal{E}_* \ge \Th\bke{\int_{B_R} (\rho(|x|) - \rho_*)^2\, dx}.
}
\end{enumerate}
\end{theorem}

\emph{Comments on  Theorem \ref{thm-XYZ} }
:
\EN{
\item 
Fix $p_\infty=p_{\infty,*}$. Then, the coercive energy estimate of Theorem \ref{thm-XYZ} implies that, relative to perturbations of the same bubble mass, the total energy $\mathcal{E}_{\rm total}$ is locally convex around the equilibrium $(\rho_*,R_*,\dot R_*=0)$ and that the equilibrium $(\rho_*,R_*,\dot R_*=0)$ is the unique  local minimizer of the total energy $\mathcal{E}_{\rm total}$ .
\item The estimate \eqref{eq-bv-4.35} is a lower bound for the  functional $(\rho,R,\dot R)\mapsto 
\mathcal{E}_{\rm total}[\rho,R,\dot R]$. It does not depend on $(\rho,R,\dot R)$ being a solution of the evolution equations.
\item Theorem \ref{thm-XYZ}
 applies to all positive  solutions $(\rho_*,R_*)$ of the algebraic system \eqref{eq-steadystate}.
}

\begin{remark}[Surface tension versus thermal diffusion]
In the spherically symmetric approximate model we study, the surface tension $\si$ does not play a role in the relaxation of the bubble to equilibrium. 
One expects it to play a role in the rounding out of non-spherical bubble deformations, which are not under consideration here.
Although typically surface tension $\si$ is positive for liquid / gas interface, our analysis for asymptotic stability applies to $\si=0$ or even some negative range of $\si$. 
In fact, the cubic equation \eqref{eq-cubic} admits a unique positive solution for all $\si\in\R$. 
Besides, 
the equilibrium energy $\mathcal E_*$ remains the conditional minimizer as long as the coefficient of the third term on the right of \eqref{eq-coercive-explicit} is positive by the coercive energy estimate \eqref{eq-bv-4.35}. 
It is equivalent to $\si> - 3R_*p_{\infty,*}/4$.

On the other hand, the thermal conductivity of gas, $\ka_g$, and far-field liquid temperature, $T_\infty$, both play a role in energy dissipation \eqref{eq-bv-4.16}.
Our analysis fails when either of these two parameters vanishes.
However, the case when $T_\infty=0$ is excluded since it would lead to a solution that is singular everywhere:
$p_g(t) = \Rg T_\infty \rho_g(R(t),t)=0$, 
and $s=c_v\log(p_g/(\rho_g)^\ga) = -\infty$
. 
Therefore, the only physical parameter that plays a role in the damping mechanism is the thermal conductivity of gas, $\ka_g$. Indeed, $\ka_g/(\ga c_v)$ is the diffusion coefficient of the parabolic PDE \eqref{eq-bv-3.10prime} that forces the gas density to distribute uniformly inside the bubble and causes the energy dissipation. 
This then leads to the thermal damping mechanism of the bubble radius by the ODE of bubble radius.
\end{remark}

\section{Nonlinear asymptotic stability: Proof of Theorem \ref{thm-asystab}}\label{sec-asystab}

In this section, we prove Theorem \ref{thm-asystab}. 
We show that family (manifold) of sphericlly symmetric equilibrium states is asymptotically stable. Our first step is to prove

\subsection{Asymptotic stability of equilibria with respect to small mass-preserving perturbations}

We begin with proving the asymptotic stability of a fixed equilibrium relative to mass preserving perturbations.
To this end, we make use of Theorem \ref{thm-XYZ} and the temporal integrability of the right hand side of \eqref{eq-bv-4.16}.  

\begin{proposition}[Asymptotic stability of a fixed equilibrium relative to mass preserving perturbations]\label{thm-asystab-1}
~Fix parameters \eqref{params-approx-spherical} and set $p_\infty=p_{\infty,*}$ in the system \eqref{eq-bv-3.10prime}--\eqref{eq-bv-3.16prime}.
 For arbitrary fixed $M_0>0$, let $(\rho_*[M_0],R_*[M_0])$ denote the unique spherically symmetric equilibrium with bubble mass $M_0$ given in Proposition \ref{prop-equilib}, {\it i.e.} $M_0 = \frac43 \pi \rho_*[M_0] \left(R_*[M_0]\right)^3$. 
 There exists $\eta_0>0$, such that the following holds for all $0<\eta\le\eta_0$:\\
  Let $(\rho(r,t),R(t))$ be the global in time solution of the free boundary problem \eqref{eq-bv-3.10prime}--\eqref{eq-bv-3.16prime}  with initial data $\rho(r,0)=\rho_0(r)$, $R(0)=R_0$, $\dot R(0)=\dot R_0$ which is 
   a \underline{mass preserving} and small  perturbation of $(\rho_*[M_0],R_*[M_0])$, {\it i.e.} 
   \[
M_0 = \int_{R_0} \rho_0 \quad \textrm{(mass preserving perturbation)}
\]
and 

\begin{equation}
\oldnorm{\bke{ \rho_0 - \rho_*[M_0],R_0 - R_*[M_0],\dot R_0} } 
\le \eta.
\label{sm-data}
\end{equation}
 Then, as $t\to+\infty$
 \begin{equation}
  \oldnorm{\bke{\rho(\cdot,t) - \rho_*[M_0],R(t) - R_*[M_0],\dot R(t))}}
  \to0,
  \label{a-stab1}\end{equation}
where the norm $\oldnorm{\cdot}$ is defined in \eqref{onorm-def}.
Moreover, $|\ddot R(t)| + |\dddot R(t)|\to0$ as $t\to+\infty$.
\end{proposition}

\begin{proof}
Consider a fixed equilibrium $(\rho_*,R_*,\dot R_*=0)$ and a nearby (non-constant) initial condition $(\rho_0,R_0,\dot{R}_0)$ (see the hypothesis \eqref{sm-data}) and such that 
\[ \int_{B_{R_0}} \rho_0 = \int_{B_{R_*}} \rho_* = \frac{4\pi}{3} R_*^3 \rho_*=M_0.\]
Hence, 
\begin{equation}
 \int_{B_{R(t)}} \rho(\cdot,t) = \frac{4\pi}{3} R_*^3 \rho_*,\quad \textrm{for all $t\ge0$}.
 \label{c_of_m}
 \end{equation}

First recall that $p_g=\Rg T_g\rho_g $ (see \eqref{eq-1.2-d}) and hence
\[
\frac{\nb_r T}{T} = \nb_r \log T = \nb_r \bke{\log p - \log(\Rg \rho)} = - \frac{\nb_r\rho}\rho,\qquad T=T_g,\ p=p_g,\ \rho=\rho_g.
\]
Together with the energy dissipation relation \eqref{eq-bv-4.16} we have
\EQ{\label{eq-bv-4.40}
\frac{d}{dt} \mathcal{E}_{\rm total}(t) 
= -\ka T_\infty \int_{B_{R(t)}} \frac{|\nb_r T|^2}{T^2}
-16\pi\mu_l R\dot R^2
= -\ka T_\infty \int_{B_{R(t)}} \frac{|\nb_r \rho|^2}{\rho^2}
-16\pi\mu_l R\dot R^2,
}
where $\mathcal{E}_{\rm total}(t)$ is given in Definition \ref{def:total-en}:
\begin{align*} \mathcal{E}_{\rm total}(t) =
FE(t) + 4\pi \si (R(t))^2 + \frac{4\pi}3 R^3(t) + 2\pi\rho_l \dot{R}^2(t) R^3(t).
\end{align*}
Integrating \eqref{eq-bv-4.40} with respect  time we obtain
\begin{equation}
\ka T_\infty \int_0^t\int_{B_{R(\tau)}} \frac{|\nb_r \rho(|x|,\tau)|^2}{\rho(|x|,\tau)^2}\, dxd\tau  
+ 16\pi\mu_l \int_0^t R(\tau)(\dot R(\tau))^2\, d\tau
= \mathcal{E}_{\rm total}(0) -\mathcal{E}_{\rm total}(t).\label{en-dis-int}
\end{equation}
Applying now  the key coercive lower bound on $\mathcal{E}_{\rm total}$  (Theorem \ref{thm-XYZ}, and in particular \eqref{eq-bv-4.35}) gives
\begin{align*}
\ka T_\infty \int_0^t\int_{B_{R(\tau)}} \frac{|\nb_r \rho(|x|,\tau)|^2}{\rho(|x|,\tau)^2}\, dxd\tau  &
+ 16\pi\mu_l \int_0^t R(\tau)(\dot R(\tau))^2\, d\tau\\
&\le \mathcal{E}_{\rm total}(0) - \mathcal{E}_* - \bke{\mathcal{E}_{\rm total}(t) - \mathcal{E}_*}  \\
&\le \mathcal{E}_{\rm total}(0) - \mathcal{E}_* - \Th \bke{\int_{B_R} (\rho - \rho_*)^2\, dx}
\le \mathcal{E}_{\rm total}(0) - \mathcal{E}_*.
\end{align*}

By the regularity of $\rho(x,t)$ and $R(t)$, we have that $\int_{B_{R(t)}} \frac{|\nb_r \rho(|x|,\tau)|^2}{\rho(|x|,\tau)^2}\, dx$ and $R(t)(\dot R(t))^2$ are uniformly continuous. 
Recall the following alternative form of Barbalat's lemma:
\EQ{\label{eq-barbalat}
\text{
\emph{Suppose $\int_0^t f(\tau)\, d\tau$ has a finite limit as $t\to\infty$.}}\\
\text{\emph{If $f(t)$ is uniformly continuous function, then $\lim_{t\to\infty} f(t)=0$.}}
}
By the above Barbalat's lemma, we conclude that
\begin{equation} \int_{B_{R(t)}} \frac{|\nb_r \rho(|x|,t)|^2}{\rho(|x|,t)^2}\, dx\in L^1_t((0,\infty))\quad \textrm{and hence}\quad
\lim_{t\to\infty} \int_{B_R} \frac{|\nb_r \rho(|x|,t)|^2}{\rho(|x|,t)^2}\, dx = 0,
\label{L1-dt}
\end{equation}
and, if $\mu_l>0$,
\[
R(t)(\dot R(t))^2 \in L^1_t((0,\infty))\quad \textrm{and hence}\quad
\lim_{t\to\infty} R(t)(\dot R(t))^2  = 0.
\]

We next change variables, via $|x|=r=Ry$, to transform integrals over $B_{R(t)}$ into integrals over $B_1$.
Noting that $\overline\rho(y,t) = \rho(Ry,t)$, we get
\[
\int_{B_{R(t)}} \frac{|\nb_r \rho(|x|,t)|^2}{\rho(|x|,t)^2}\, dx 
= R(t) \int_{B_1} \frac{|\nb_y \overline\rho(y,t)|^2}{\overline\rho(y,t)^2}\, dy.
\]
Since that $R(t)$ is bounded away from zero and that $\overline\rho(y,t)$ is bounded from above (see \eqref{eq-bv-4.30}, \eqref{eq-bv-4.32})  we have 
\begin{equation}
\lim_{t\to\infty} \int_{B_1} |\nb_y \overline\rho(y,t)|^2\, dy = 0.
\label{H1to0}
\end{equation} 
Using the interpolation Lemma \ref{lem-interpolate} with $k=0$, $n=3$, $p=2$, $m=1$, $\ga=2\al$, and $s=2$,
\[
\norm{\nb_y \overline\rho}_{L^\infty(B_1)} \le C_1\norm{\nb_y \overline\rho}_{L^2(B_1)}^{\frac25} \norm{\nb_y \overline\rho}_{C_y^{1+2\al}(B_1)}^{\frac35} + C_2 \norm{\nb_y \overline\rho}_{L^2(B_1)}
\]
where $\norm{\nb_y \overline\rho}_{C_y^{1+2\al}(B_1)} \sim \norm{\rho}_{C_r^{2+2\al}(B_R)}$ is uniformly bounded
 by \eqref{eq-bv-4.30}. Therefore, 
 \begin{equation*}
 \textrm{$\nb_y \overline\rho(\cdot,t)\to0$ uniformly in $B_1$ as $t\to\infty$}.
 \end{equation*}
 and by \eqref{eq-bv-4.30}
 \begin{equation}
 \textrm{$\nb_r\rho(\cdot,t)\to0$ uniformly in $B_{R(t)}$ as $t\to\infty$}.
 \end{equation}

Furthermore, by the Poincar\'e inequality and \eqref{H1to0}
\begin{equation}
\int_{B_1} \bke{ \overline\rho(|y|,t) - \frac1{|B_1|} \int_{B_1} \overline\rho(|z|,t)\, dz}^2 dy 
\lec \int_{B_1} |\nb_y \overline\rho|^2\, dy \to0\ \text{ as }t\to\infty.
\label{pineq-ap}
\end{equation}
Moreover, by \eqref{c_of_m} (conservation of mass) 
\begin{equation} \frac1{|B_1|} \int_{B_1} \overline\rho(|z|,t)\, dz = \left(\frac{R_*}{R(t)}\right)^3\rho_*.\label{mass-s}\end{equation}
and hence, 

\begin{equation}
\int_{B_1} \left( \overline\rho(\cdot,t) - \left(\frac{R_*}{R(t)}\right)^3\rho_*\right)^{2} 
 \to0\ \text{ as }t\to\infty.
\label{pineq-ap1}
\end{equation}
In fact we claim that 
\begin{equation}
\textrm{$\overline\rho(\cdot,t) - \left(\frac{R_*}{R(t)}\right)^3\rho_*\to 0$ uniformly on $B_1$.}
\label{tdr-unif}
\end{equation}
 Indeed, applying  interpolation Lemma \ref{lem-interpolate}, with $k=0$, 
 $n=3$, $p=2$, $m=2$, $\ga=2\al$, and $s=2$, we have
\EQN{
\norm{\overline\rho(\cdot,t) - \bke{\frac{R_*}{R(t)}}^3\rho_*}_{L^\infty(B_1)}  
&\le C_1 \norm{\overline\rho(\cdot,t) - \bke{\frac{R_*}{R(t)}}^3\rho_*}_{L^2(B_1)}^{\frac47}  \norm{\overline\rho(\cdot,t) - \bke{\frac{R_*}{R(t)}}^3\rho_*}_{C_y^{2+2\al}(B_1)}^{\frac37}\\
&\quad + C_2 \norm{\overline\rho(\cdot,t) - \bke{\frac{R_*}{R(t)}}^3\rho_*}_{L^2(B_1)}.
}
Since by \eqref{eq-bv-4.30}, we have  that $\norm{\overline\rho(\cdot,t) - \bke{\frac{R_*}{R(t)}}^3\rho_*}_{C_y^{2+2\al}(B_1)} \sim \norm{\rho(\cdot,t) - \bke{\frac{R_*}{R(t)}}^3\rho_*}_{C_r^{2+2\al}(B_R)}$ is uniformly bounded,
 we have \eqref{tdr-unif} and hence  
\begin{equation}
\overline\rho(\cdot,t)(R(t))^3\to\rho_* R_*^3\quad \textrm{ uniformly in $B_1$ as $t\to\infty$}.
\label{trR-uni}
\end{equation}

 We next show that as $t\to\infty$, we have $\dot R$, $\ddot R$, and $\dddot R \to0$.
Since $p(t) = \Rg T_\infty \overline\rho(y=1,t)$ (using \eqref{eq1.2simplified} and $T(r=1,t)=T_\infty$), we have $p(t)(R(t))^3\to \Rg T_\infty \rho_*R_*^3$ as $t\to\infty$. By the regularity of $p(t)$ and $R(t)$, \eqref{eq-higher-regularity-R} and \eqref{eq-higher-regularity-p}
, we have that $\frac{d}{dt} (p(t)(R(t))^3)$ is uniformly continuous. 
Recall  Barbalat's lemma:
\EQN{
\text{
\emph{Suppose $f(t)\in C^1(a,\infty)$  and $\lim_{t\to\infty} f(t)=\alpha$, with $|\alpha|<\infty$.}}\\
\text{\emph{ If $f^\prime(t)$ is uniformly continuous, then $\lim_{t\to\infty} f^\prime(t)=0$.}}
}
By the Barbalat's lemma, $\frac{d}{dt} (p(t)(R(t))^3)\to0$ as $t\to\infty$. That is, $
\pd_tp(t) (R(t))^3 + 3 p(t) (R(t))^2\dot R(t) \to0\ \text{ as } t\to\infty,
$
and therefore since $\rho(\cdot,t)$ and $R(t)$ are uniformly bounded away from zero,
\[
\frac{\pd_tp(t)}{3p(t)} R(t) + \dot R \to0\ \text{ as } t\to\infty,
\]
Sending $t\to\infty$ in \eqref{eq-bv-3.15prime} yields
\begin{align*}
\lim_{t\to\infty} \dot R(t) &= \lim_{t\to\infty}\bkt{ -\frac{\ka}{\ga c_v} \frac{\pd_r\rho(R(t),t)}{(\rho(R(t),t))^2} - \frac{R}{3\ga} \frac{\pd_tp}p }\\
&=\lim_{t\to\infty}\bkt{ -\frac{\ka}{\ga c_v} \frac{\pd_r\rho(R(t),t)}{(\rho(R(t),t))^2} - 
\frac{1}{\ga}\left(\frac{R}{3} \frac{\pd_tp}p +\dot{R}(t)\right) +  \frac{1}{\ga}\dot{R}(t)}
= \ga^{-1} \lim_{t\to\infty} \dot R(t).
\end{align*}
Hence,  $(1-\ga^{-1}) \lim_{t\to\infty}\dot R(t)=0$, and  since $\ga\neq1$ we have 
$\lim_{t\to\infty}\dot R(t)=0$.
Further application of Barbalat's lemma yields $\ddot R(t), \dddot R(t)\to0$ as $t\to\infty$.

We next prove $R(t)\to R_*$ along a sequence of time $t_k\to\infty$.
Since $R(t)$ is bounded in $t$, there is a sequence $\{t_k\}$ along which $R(t_k)\to R_{**}$ as $k\to\infty$ for some $R_{**}>0$. 
Furthermore, by \eqref{tdr-unif} $ \overline\rho(y,t_k)$ converges uniformly on $B_1$ to a limit $\rho_{**}$ as $k\to\infty$. 
By \eqref{eq-bv-3.16prime}, since $\dot{R}(t)$ and $\ddot{R}(t)$ tend to zero as $t\to\infty$,  we obtain
\[
\rho_{**} = \frac1{\Rg T_\infty} \bke{ p_{\infty,*} + \frac{2\si}{R_{**}}}.
\]
Further, passing to the limit in \eqref{c_of_m}, we obtain 
\[\rho_{**}R_{**}^3 = \rho_*R_*^3.\]
Hence, $(\rho_{**},R_{**})$  satisfies the algebraic system \eqref{eq-steadystate} characterizing the unique spherically symmetric equilibrium  of with bubble mass $(4\pi/3) R_*^3 \rho_*$ .
We conclude that  $(\rho_{**},R_{**}) = (\rho_*,R_*)$. 
 Furthermore, 
$p(t_k) \to p_*$ as $k\to\infty$,  where $p_* = \Rg T_\infty\rho_* = p_{\infty,*} + \frac{2\si}{R_*}$. 

The above discussion establishes $(\overline\rho(y,t), R(t))$ converge to $(\rho_*,R_*)$ along the sequence $\{t_k\}$.
To prove this limit holds as $t\to\infty$ we return to the energy dissipation law. 
Since $\int_{B_{R(\tau)}} \frac{|\nb \rho(|x|,\tau)|^2}{\rho(|x|,\tau)^2}\, dx$
and $R(\tau)(\dot R(\tau))^2$ are integrable on $[0,\infty)$ (see \eqref{L1-dt}) we have, from the identity 
\begin{align}
\mathcal{E}_{\rm total}(t) &= -\ka T_\infty \int_0^t\int_{B_{R(\tau)}} \frac{|\nb_r \rho(|x|,\tau)|^2}{\rho(|x|,\tau)^2}\, dxd\tau 
-16\pi\mu_l \int_0^t R(\tau)(\dot R(\tau))^2\, d\tau
+ \mathcal{E}_{\rm total}(0) ,
\end{align}
which follows from rearranging terms in \eqref{en-dis-int}, that $\lim_{t\to\infty} \mathcal{E}_{\rm total}(t) $ exists.
Since $\mathcal{E}_{\rm total}(t_k)\to \mathcal{E}_*$ as $k\to\infty$, we have $\mathcal{E}_{\rm total}(t)\to\mathcal{E}_*$ as $t\to\infty$.
By Theorem \ref{thm-XYZ},
\[
\lim_{t\to\infty} \int_{B_{R(t)}} (\rho(\cdot,t) - \rho_*)^2 \le \Th^{-1} \lim_{t\to\infty} \left( \mathcal{E}_{\rm total}(t) - \mathcal{E}_* \right)
= 0.
\]

Using the interpolation Lemma \ref{lem-interpolate} with $k=0$, $n=3$, $p=2$, $m=2$, and $\ga=2\al$, we obtain 
\begin{align}
\norm{\rho(\cdot,t) - \rho_*}_{L^\infty(B_{R(t)})} 
&\le C_1 \norm{\rho(\cdot,t) - \rho_*}_{L^2(B_{R(t)})}^{\frac47} \norm{\rho(\cdot,t) - \rho_*}_{C_r^{2+2\al}(B_{R(t)})}^{\frac37}\nonumber\\
&\qquad\qquad + C_2 \norm{\rho(\cdot,t) - \rho_*}_{L^2(B_{R(t)})} 
\to0\ \text{ as }t\to\infty.\label{rho2rhos}
\end{align}
Thus, $\rho(x,t)\to\rho_*$ uniformly on $B_{R(t)}$.

From the uniform convergence of $\rho(x,t)$ to $\rho_*$, we will finally conclude that $R(t)\to R_*$.
  By conservation of mass, 
\EQ{\label{eq-bv-4.37}
\int_{B_{R(t)}} (\rho(\cdot,t)-\rho_*) = \int_{B_{R_0}}\rho_0 - \frac{4\pi\rho_*}3 R^3(t) = \frac{4\pi\rho_*}3 (R_*^3 - R^3(t)),
}
which implies
\begin{align}
\label{eq-bv-4.38}
|R(t)- R_*| &\le \frac3{4\pi\rho_*(R^2(t)+R(t)R_*+R_*^2)} |B_{R(t)}|^{\frac12} \bke{\int_{B_{R(t)}} |\rho(\cdot,t)-\rho_*|^2 }^{\frac12} \\
&\le C \bke{\int_{B_{R(t)}} |\rho(\cdot,t)-\rho_*|^2 }^{\frac12}
\nonumber\end{align}
It follows from \eqref{rho2rhos} and \eqref{eq-bv-4.38} that $R(t)\to R_*$ as $t\to\infty$. 

Finally, since we have $\norm{\rho(\cdot,t) - \rho_*}_{L^\infty_r} + \norm{\nb_r\rho(\cdot,t)}_{L^\infty_r} + |R(t) - R_*| + |\dot R(t)| + |\ddot R(t)| + |\dddot R(t)|\to0$ as $t\to+\infty$, the same bootstrap argument in the end of the proof of \cite[Theorem 4.1]{bv-SIMA2000} yields
\[
\norm{\rho(\cdot,t) - \rho_*}_{C^{2+2\al}_r} \to0\ \text{ as }t\to\infty,
\]
completing the proof of Proposition \ref{thm-asystab-1}.
\end{proof}

\subsection{Proof of Theorem \ref{thm-asystab}; asymptotic stability of the family of spherically symmetric equilibria relative to arbitrary small  perturbations}

We are in the position to prove Theorem \ref{thm-asystab}.
The proof is similar to that of Theorem \ref{thm-bv-4.1-anyperturb}.
We now assume  that $(\rho_0, R_0, \dot R_0)$ is an \underline{arbitrary} sufficiently small perturbation 
of  some fixed equilibrium $(\rho_*[M_*],R_*[M_*],\dot R_*=0)$, in the sense of \eqref{data-small}. By Proposition \ref{prop:contin}
 there is a unique $(\rho_*[M_0], R_*[M_0])$ such that $\textrm{Mass}[ \rho_*[M_0], R_*[M_0]]=M_0=\textrm{Mass}[\rho_0, R_0]$,
  and 
\EQ{
  | R_*[M_0] - R_*[M_*]|  + | \rho_*[M_0]-\rho_*[M_*]|  
  &\le C \left(|R_0 - R_*[M_*]|\ +\ \|\rho_0 - \rho_*[M_*]\|_{L^\infty(B_{R_0})}\right)\\
  &\le C \left(|R_0 - R_*[M_*]|\ +\ \|\rho_0 - \rho_*[M_*]\|_{C_r^{2+2\al}(B_{R_0})}\right).
  }
Hence, 
\begin{align*}
\|\rho_0 - \rho_*[M_0]&\|_{C_r^{2+2\alpha}} + |R_0- R_*[M_0]|\\
&= \|\rho_0-\rho_*[M_*]+\rho_*[M_*]- \rho_*[M_0]\|_{C_r^{2+2\alpha}} + |R_0 - R_*[M_*]+R_*[M_*] - R_*[M_0]|\\
&\le  \|\rho_0-\rho_*[M_*]\|_{C_r^{2+2\alpha}} + |\rho_*[M_*] - \rho_*[M_0]| + |R_0 - R_*[M_*]| + |R_*[M_*] - R_*[M_0]|\\
&\le C^\prime\left( \|\rho_0 - \rho_*[M_*]\|_{C_r^{2+2\alpha}} + |R_0 - R_*[M_*]|  \right).
\end{align*}
Therefore, choosing $\|\rho_0 - \rho_*[M_*]\|_{C_r^{2+2\alpha}} + |R_0 - R_*[M_*]| $ and $\dot{R}_0$ sufficiently small 
we conclude from Proposition \ref{thm-asystab-1} that $(\rho(\cdot,t),R,\dot R)$ converges to $(\rho_*[M_0], R_*[M_0], \dot{R}_*=0)$ in the norm $\oldnorm{\cdot}$ as $t\to\infty$.
This completes the proof of Theorem \ref{thm-asystab}.

\section{Exponential decay in nonlinear bubble oscillations; $\mathcal{M}_*$, as an attracting center manifold: Proof of the main result Theorem \ref{thm-nonlinear-exp-decay}}\label{sec-nonlinear}

In this section, we use a weak form of the asymptotic stability result in Theorem \ref{thm-asystab} to promote the result of exponential rate of convergence toward the manifold of equilibria, Theorem \ref{thm-nonlinear-exp-decay}.
We carry out the analysis in line with center manifold theorem for a dynamical system in the phase space $\ell^2$.
To this end, we first convert the free boundary problem \eqref{red-eqns}--\eqref{eq-bv-3.14prime} into an infinite-dimensional dynamical system, \eqref{eq-carr-6.3.1-simple-form}, in Proposition \ref{prop-equivalent-systems}.
Next, we interpret the decay result from Theorem \ref{thm-asystab} as an \emph{a priori} estimate in Proposition \ref{prop-a-priori}.
We then discuss the spectrum of the linear operator in the dynamical system \eqref{eq-carr-6.3.1-simple-form} and find that there is a neutral mode and decaying modes (Proposition \ref{prop-spectrum}).
In light of this, we decompose the solution of the system a slow decaying part \eqref{eq-carr-6.3.4-x} and a fast decaying part \eqref{eq-carr-6.3.4-x}.
Moreover, we show that the manifold of equilibria is a center manifold of the system \eqref{eq-carr-6.3.1-simple-form} in Lemma \ref{lem-mfld-equil-ceneter-mfld}.
We derive an estimate of the nonlinearity in Proposition \ref{prop-nonlinear-est}.
With the {\it a priori} estimate and the nonlinear estimate in hand, we apply Proposition \ref{prop-stab-center-mfd} to conclude Theorem \ref{thm-nonlinear-exp-decay}.

\subsection{A dynamical system formulation of the free boundary problem \eqref{red-eqns}--\eqref{eq-bv-3.14prime}}

To investigate the solution near a given equilibrium solution, we derive two equivalent systems of \eqref{red-eqns}--\eqref{eq-bv-3.14prime} below. 

\begin{proposition}\label{prop-equivalent-systems}
Given $(\rho_*, R_*, \dot R_*=0)\in \mathcal M_*$. 
Let $(\rho,R)$ denote a solution of the free boundary problem \eqref{red-eqns}--\eqref{eq-bv-3.14prime} with $p_\infty(t) \equiv p_{\infty,*}$.
Decompose 
\EQ{\label{eq-perturbation}
\rho(R(t)y,t) = \rho_* + u(y,t) + z(t),\qquad 
z(t) = \rho(R(t),t) - \rho_*,\qquad
R(t) = R_* + \mathcal R(t).
}
Then,
\begin{enumerate}
\item Equations  \eqref{red-eqns}--\eqref{eq-bv-3.14prime} are equivalent to the following system for $(u, z, \mathcal R)$ with zero-Dirichlet boundary condition 
\begin{subequations}
\label{red-eqns-linear-new-1}
\begin{align}
\pd_tu &= \bar\ka \De_yu - \bke{ 1 - \frac1{\ga} } \dot z + F,\quad 0\le y\le1,\qquad u(1,t)=0,\quad \ t>0,\label{eq-u-pde-1}\\
\dot{\mathcal R} &= -\frac{R_*\bar\ka}{\rho_*}  \pd_yu(1,t) - \frac{R_*}{3\ga\rho_*}\, \dot z + G,\quad t>0,\label{eq-dotR-u-1}\\
z(t) &= \frac1{\Rg T_\infty} \bke{ -\frac{2\si}{R_*^2}\mathcal R + \frac{4\mu_l}{R_*}\, \dot{\mathcal R} + \rho_lR_*\ddot{\mathcal R} } + H,\quad t>0,\label{eq-z-1}
\end{align}
\end{subequations}
where
\EQ{\label{eq-bar-ka-def}
\bar\ka = \frac{\ka}{\ga c_vR_*^2\rho_*},
}
\begin{subequations}
\label{red-eqns-linear-nonlinear-term}
\begin{align}
F &= \frac{\ka}{\ga c_v} \bkt{\frac1{(R_*+\mathcal R)^2(\rho_*+u+z)} - \frac1{R_*^2\rho_*}} \De_yu \label{red-eqns-linear-nonlinear-term-f}\\
&\quad  - \frac{\ka}{\ga c_v}\, \frac{|\nb_yu|^2}{(R_*+\mathcal R)^2(\rho_*+u+z)^2} + \frac1\ga\, \frac{\dot z}{\rho_* + z} \bke{\frac13 y\pd_yu + u},\notag
\\
G &= -\frac{\ka}{\ga c_v} \bkt{\frac1{(R_*+\mathcal R)(\rho_*+z)^2} - \frac1{R_*\rho_*^2}} \pd_yu(1,t) - \frac{\mathcal R \dot z}{3\ga(\rho_*+z)} + \frac{R_*}{3\ga}\,\frac{z\dot z}{\rho_*(\rho_*+z)},\label{red-eqns-linear-nonlinear-term-g}\\
H &= \frac1{\Rg T_\infty} \bkt{ -\frac{\mathcal R}{R_*(R_*+\mathcal R )}\bke{ -\frac{2\si}{R_*}\, \mathcal R + 4\mu_l \dot{\mathcal R} }  + \rho_l \bke{\mathcal R \ddot{\mathcal R} + \frac32 \dot{\mathcal R}^2} }.
\end{align}
\end{subequations}
\item Let $c_j = c_j(t)$ denote the $j$-th coefficients in the radial-Dirichlet-eigenfunction decomposition of $u$ as in \eqref{eq-u-expansion}.

Then, \eqref{red-eqns-linear-new-1} is further equivalent to the following infinite-dimensional dynamical system for ${\bf w}= (z,\mathcal R,\dot{\mathcal R}, c_1,c_2,\cdots)^\top$
\EQ{\label{eq-carr-6.3.1-simple-form} 
\dot{\bf w} = \mathcal L {\bf w} + \mathcal N({\bf w},\dot{\bf w}) =\mathcal L {\bf w} + \mathcal N^1({\bf w})\dot{\bf w} + \mathcal N^0({\bf w}),
}
where $\mathcal L$ is a linear operator given in \eqref{eq-matrixC},
$\mathcal N({\bf w},\dot{\bf w})$ is defined in \eqref{eq-def-N}, and $\mathcal N^1({\bf w})$ and $\mathcal N^0({\bf w})$ are given in \eqref{def-N1-N0-simple-form}.
\end{enumerate}
\end{proposition}

Below, we shall study the operator $\mathcal L$ acting in $\ell^2$ with dense domain consisting of
vectors
${\bf w}= (z,\mathcal R,\dot{\mathcal R}, c_1,c_2,\cdots)^\top\in \ell^2$, such that $\sum_{j\ge1} j^4 |c_j|^2<\infty$.
 
\begin{proof}[Proof of Proposition \ref{prop-equivalent-systems}]
To begin with, plugging \eqref{eq-perturbation} into the equation \eqref{red-eqns}--\eqref{eq-bv-3.14prime} and grouping linear and nonlinear terms, one directly obtain \eqref{red-eqns-linear-new-1}.

Expand $u$ in terms of the radial-Dirichlet eigenfunctions as
\EQ{\label{eq-u-expansion}
u(y,t) = \sum_{j=1}^\infty c_j(t) \phi_j(y),
}
where $\phi_j$, $j=1,2,\ldots$, are defined in \eqref{eq-eigen-formula}.
Plugging \eqref{eq-u-expansion} into \eqref{eq-u-pde-1},
\EQ{\label{eq-u-fourier}
\sum_{j=1}^\infty \dot c_j(t) \phi_j(y) = -\bar\ka \sum_{j=1}^\infty \la_jc_j(t)\phi_j(y) - \frac{\ga-1}\ga \dot z(t) + F.
}

Taking inner product of \eqref{eq-u-fourier} in $L^2(B_1)$ with $\phi_k(y)$, $k=1,2,\ldots$, one has
\EQ{\label{eq-dotc-1}
\dot c_k(t) = -\bar\ka\la_kc_k(t) - \Ga_k\dot z(t) + F_k,
}
where
\EQ{\label{eq-Ga-formula}
\Ga_k = \frac{\ga-1}\ga \int_{B_1} \phi_k(|x|)\, dx
= \frac{2\sqrt2(\ga-1)}{\sqrt\pi \ga}\, \frac{(-1)^{k-1}}k,
}
and $F_k = \int_{B_1} F\phi_k\, dx$.

Since $\phi_j$ is explicitly given in \eqref{eq-eigen-formula}, it is readily to compute $
\pd_y\phi_j(1) = \sqrt{\frac\pi2}\, (-1)^j j$.
Then \eqref{eq-dotR-u-1} becomes
\EQ{\label{eq-dotR-u-expand-1}
\dot{\mathcal R} = -\frac{R_*\bar\ka}{\rho_*} \sum_{j=1}^\infty c_j(t) \pd_y\phi_j(1) - \frac{R_*}{3\ga\rho_*}\, \dot z + G
= \sum_{j=1}^\infty c_j(t) \La_j - \frac{R_*}{3\ga\rho_*}\, \dot z + G,
}
where
\EQ{\label{eq-def-La}
\La_j = -\frac{R_*\bar\ka}{\rho_*} \sqrt{\frac{\pi}2}\, (-1)^j  j.
}
Thus, \eqref{eq-dotR-u-expand-1}, \eqref{eq-z-1}, and \eqref{eq-dotc-1} form the infinite-dimensional dynamical system
\[
\begin{pmatrix}
\frac{R_*}{3\ga\rho_*}&1&0&0&0&\cdots\\
0&1&0&0&0&\cdots\\
0&0&1&0&0&\cdots\\
\Ga_1&0&0&1&0&\cdots\\
\Ga_2&0&0&0&1&\cdots\\
\vdots&\vdots&\vdots&\vdots&\vdots&\ddots
\end{pmatrix}
\begin{bmatrix}
z\\
\mathcal R\\
\dot{\mathcal R}\\
c_1\\
c_2\\
\vdots
\end{bmatrix}^\prime
=
\begin{pmatrix}
0&0&0&\La_1&\La_2&\cdots\\
0&0&1&0&0&\cdots\\
\frac{\Rg T_\infty}{\rho_lR_*}& \frac{2\si}{\rho_lR_*^3}& -\frac{4\mu_l}{\rho_lR_*^2}&0&0&\cdots\\
0&0&0&-\bar\ka\la_1&0&\cdots\\
0&0&0&0&-\bar\ka\la_2&\cdots\\
\vdots&\vdots&\vdots&\vdots&\vdots&\ddots
\end{pmatrix}
\begin{bmatrix}
z\\
\mathcal R\\
\dot{\mathcal R}\\
c_1\\
c_2\\
\vdots
\end{bmatrix}
+ 
\begin{bmatrix}
G\\
0\\
-\frac{\Rg T_\infty}{\rho_lR_*}\, H\\
F_1\\
F_2\\
\vdots
\end{bmatrix},
\]
where $\Ga_k$ and $\La_j$ are given in \eqref{eq-Ga-formula} and \eqref{eq-def-La}, respectively, $F_k = \int_{B_1} F\phi_k\, dx$, and $F = F({\bf w},\dot{\bf w})$, $G = G({\bf w},\dot{\bf w})$, $H = H({\bf w},\dot{\bf w})$ are defined in \eqref{red-eqns-linear-nonlinear-term-f}--\eqref{red-eqns-linear-nonlinear-term-f} with $u = \sum_{k=1}^\infty c_k\phi_k$.
The inverse of the matrix on the left hand side above is 
\[
\begin{pmatrix}
\frac{R_*}{3\ga\rho_*}&1&0&0&0&\cdots\\
0&1&0&0&0&\cdots\\
0&0&1&0&0&\cdots\\
\Ga_1&0&0&1&0&\cdots\\
\Ga_2&0&0&0&1&\cdots\\
\vdots&\vdots&\vdots&\vdots&\vdots&\ddots
\end{pmatrix}^{-1}
=
\begin{pmatrix}
\frac{3\ga\rho_*}{R_*}& -\frac{3\ga\rho_*}{R_*}&0&0&0&\cdots\\
0&1&0&0&0&\cdots\\
0&0&1&0&0&\cdots\\
-\Ga_1 \frac{3\ga\rho_*}{R_*}& \Ga_1 \frac{3\ga\rho_*}{R_*}&0&1&0&\cdots\\
-\Ga_2 \frac{3\ga\rho_*}{R_*}& \Ga_2 \frac{3\ga\rho_*}{R_*}&0&0&1&\cdots\\
\vdots&\vdots&\vdots&\vdots&\vdots&\ddots
\end{pmatrix}.
\]
Left-multiplying the inverse on both sides, we obtain
\[
\begin{bmatrix}
z\\
\mathcal R\\
\dot{\mathcal R}\\
c_1\\
c_2\\
\vdots
\end{bmatrix}^\prime
=
\begin{pmatrix}
0&0&-\frac{3\ga\rho_*}{R_*}&\La_1\frac{3\ga\rho_*}{R_*}&\La_2\frac{3\ga\rho_*}{R_*}&\cdots\\
0&0&1&0&0&\cdots\\
\frac{\Rg T_\infty}{\rho_lR_*}&\frac{2\si}{\rho_lR_*^3}&-\frac{4\mu_l}{\rho_lR_*^2}&0&0&\cdots\\
0&0&\Ga_1\frac{3\ga\rho_*}{R_*}&-\Ga_1\La_1\frac{3\ga\rho_*}{R_*} - \bar\ka\la_1&-\Ga_1\La_2\frac{3\ga\rho_*}{R_*}&\cdots\\
0&0&\Ga_2\frac{3\ga\rho_*}{R_*}&-\Ga_2\La_1\frac{3\ga\rho_*}{R_*}& -\Ga_2\La_2\frac{3\ga\rho_*}{R_*} - \bar\ka\la_2&\cdots\\
\vdots&\vdots&\vdots&\vdots&\vdots&\ddots
\end{pmatrix}
\begin{bmatrix}
z\\
\mathcal R\\
\dot{\mathcal R}\\
c_1\\
c_2\\
\vdots
\end{bmatrix}
+
\begin{bmatrix}
\frac{3\ga\rho_*}{R_*} G\\
0\\
-\frac{\Rg T_\infty}{\rho_lR_*}\, H\\
-\Ga_1\,\frac{3\ga\rho_*}{R_*}\, G + F_1\\
-\Ga_2\,\frac{3\ga\rho_*}{R_*}\, G + F_2\\
\vdots
\end{bmatrix}.
\]
It can be written in the form
\EQ{\label{eq-carr-6.3.1}
\dot{\bf w} = \mathcal L {\bf w} + \mathcal N({\bf w}, \dot{\bf w}),
}
where ${\bf w} = (z,\mathcal R,\dot{\mathcal R}, c_1, c_2,\ldots)^\top$, 
\EQ{\label{eq-matrixC}
\mathcal L = \begin{pmatrix}
0&0&-\frac{3\ga\rho_*}{R_*}&\La_1\frac{3\ga\rho_*}{R_*}&\La_2\frac{3\ga\rho_*}{R_*}&\cdots\\
0&0&1&0&0&\cdots\\
\frac{\Rg T_\infty}{\rho_lR_*}&\frac{2\si}{\rho_lR_*^3}&-\frac{4\mu_l}{\rho_lR_*^2}&0&0&\cdots\\
0&0&\Ga_1\frac{3\ga\rho_*}{R_*}&-\Ga_1\La_1\frac{3\ga\rho_*}{R_*} - \bar\ka\la_1&-\Ga_1\La_2\frac{3\ga\rho_*}{R_*}&\cdots\\
0&0&\Ga_2\frac{3\ga\rho_*}{R_*}&-\Ga_2\La_1\frac{3\ga\rho_*}{R_*}& -\Ga_2\La_2\frac{3\ga\rho_*}{R_*} - \bar\ka\la_2&\cdots\\
\vdots&\vdots&\vdots&\vdots&\vdots&\ddots
\end{pmatrix},
}
\EQ{\label{eq-def-N}
\mathcal N({\bf w}, \dot{\bf w}) = 
\begin{bmatrix}
\frac{3\ga\rho_*}{R_*} G\\
0\\
-\frac{\Rg T_\infty}{\rho_lR_*}\, H\\
-\Ga_1\,\frac{3\ga\rho_*}{R_*}\, G + F_1\\
-\Ga_2\,\frac{3\ga\rho_*}{R_*}\, G + F_2\\
\vdots
\end{bmatrix}.
}

Write $\mathcal N({\bf w},\dot{\bf w}) = \mathcal N({\bf w}, {\bf p})$, in which
\[
{\bf p} = \dot{\bf w} = (\dot z, \dot{\mathcal R}, \ddot{\mathcal R}, \dot{c_1}, \dot{c_2},\ldots)^\top\\
=:(a,\mathcal S, \mathcal U, d_1, d_2,\ldots)^\top.
\]
Then
\begin{subequations}
\begin{align}
F({\bf w}, {\bf p}) &= \frac{\ka}{\ga c_v} \bkt{\frac1{(R_*+\mathcal R)^2 \bke{\rho_*+ u +z} } - \frac1{R_*^2\rho_*}} \sum_{j=1}^\infty (-\la_j)c_j\phi_j\notag\\
&\quad - \frac{\ka}{\ga c_v}\, \frac{ \abs{\nb_y u}^2}{(R_*+\mathcal R)^2 \bke{\rho_*+ u +z}^2} + \frac1\ga\, \frac{a}{\rho_* + z} \bke{\frac13 y \sum_{j=1}^\infty c_j\pd_y\phi_j + \sum_{j=1}^\infty c_j\phi_j },\label{def-F}\\
\text{ where }\ &u= \sum_{j=1}^\infty c_j\phi_j, \ \text{ and } \notag\\
G({\bf w}, {\bf p}) &= -\frac{\ka}{\ga c_v} \bkt{\frac1{(R_*+\mathcal R)(\rho_*+z)^2} - \frac1{R_*\rho_*^2}} \sum_{j=1}^\infty \sqrt{\frac\pi2} (-1)^j j c_j - \frac{\mathcal R a}{3\ga(\rho_*+z)} + \frac{R_*}{3\ga}\,\frac{za}{\rho_*(\rho_*+z)},\label{def-G}\\
H({\bf w}, {\bf p}) &= \frac1{\Rg T_\infty} \bkt{ -\frac{\mathcal R}{R_*(R_*+\mathcal R )}\bke{ -\frac{2\si}{R_*}\, \mathcal R + 4\mu_l \dot{\mathcal R} }  + \rho_l \bke{\mathcal R \mathcal U + \frac32 \dot{\mathcal R}^2} }.\label{def-H}
\end{align}
\end{subequations}
It is easily see from above that
\[
F({\bf w}, {\bf p}) = \bka{ {\bf F}^1({\bf w}), {\bf p} } + F^0({\bf w}),\quad
G({\bf w}, {\bf p}) = \bka{ {\bf G}^1({\bf w}), {\bf p} } + G^0({\bf w}),\quad
H({\bf w}, {\bf p}) = \bka{ {\bf H}^1({\bf w}), {\bf p} } + H^0({\bf w}),
\]
where $\bka{\cdot,\cdot}$ denotes inner product in the Hilbert space $\ell^2$ and
\EQ{\label{def-FGH-simple-form}
{\bf F}^1({\bf w}) &= ( \frac1\ga \frac1{\rho_*+z} \bke{\frac13 y \sum_{j=1}^\infty c_j\pd_y\phi_j + \sum_{j=1}^\infty c_j\phi_j },0,0,0,0,\cdots)^\top,\\
F^0({\bf w}) &= \frac{\ka}{\ga c_v} \bkt{\frac1{(R_*+\mathcal R)^2 \bke{\rho_*+ u +z} } - \frac1{R_*^2\rho_*}} \sum_{j=1}^\infty (-\la_j)c_j\phi_j - \frac{\ka}{\ga c_v}\, \frac{ \abs{\nb_y u}^2}{(R_*+\mathcal R)^2 \bke{\rho_*+ u +z}^2},\\
{\bf G}^1({\bf w}) &= ( -\frac{\mathcal R}{3\ga(\rho_*+z)} + \frac{R_*z}{3\ga\rho_*(\rho_*+z)},0,0,0,0,\cdots)^\top,\\
G^0({\bf w}) &= -\frac{\ka}{\ga c_v} \bkt{\frac1{(R_*+\mathcal R)(\rho_*+z)^2} - \frac1{R_*\rho_*^2}} \sum_{j=1}^\infty \sqrt{\frac\pi2} (-1)^j j c_j,\\
{\bf H}^1({\bf w}) &= (0,0,\frac{\rho_l\mathcal R}{\Rg T_\infty},0,0,\cdots)^\top,\quad 
H^0({\bf w}) = \frac1{\Rg T_\infty} \bkt{ -\frac{\mathcal R}{R_*(R_*+\mathcal R )}\bke{ -\frac{2\si}{R_*}\, \mathcal R + 4\mu_l \dot{\mathcal R} }  + \rho_l \frac32 \dot{\mathcal R}^2 }.
}
Therefore, the nonlinearity $\mathcal N({\bf w},{\bf p})$ takes the form
\EQ{\label{eq-N-simple-form}
\mathcal N({\bf w},{\bf p}) = \mathcal N^1({\bf w}) {\bf p} + \mathcal N^0({\bf w}),
}
where
\EQ{\label{def-N1-N0-simple-form}
\mathcal N^1({\bf w}) = 
 \begin{pmatrix}
\frac{3\ga\rho_*}{R_*} {\bf G}^1({\bf w})\\
{\bf 0}\\
-\frac{\Rg T_\infty}{\rho_lR_*}\, {\bf H}^1({\bf w})\\
-\Ga_1\,\frac{3\ga\rho_*}{R_*}\, {\bf G}^1({\bf w}) + {\bf F}_1^1({\bf w})\\
-\Ga_2\,\frac{3\ga\rho_*}{R_*}\, {\bf G}^1({\bf w}) + {\bf F}_2^1({\bf w})\\
\vdots
 \end{pmatrix},\quad
\mathcal N^0({\bf w}) =  
 \begin{bmatrix}
\frac{3\ga\rho_*}{R_*} G^0({\bf w})\\
0\\
-\frac{\Rg T_\infty}{\rho_lR_*}\, H^0({\bf w})\\
-\Ga_1\,\frac{3\ga\rho_*}{R_*}\, G^0({\bf w}) + F_1^0({\bf w})\\
-\Ga_2\,\frac{3\ga\rho_*}{R_*}\, G^0({\bf w}) + F_2^0({\bf w})\\
\vdots
 \end{bmatrix}
}
in which $F_j^0({\bf w}) = \int_{B_1} F^0({\bf w}) \phi_j\, dx$ and
\[
{\bf F}_j^1({\bf w}) = \int_{B_1} {\bf F}^1({\bf w}) \phi_j\, dx = ( \frac1\ga \frac1{\rho_*+z} \int_{B_1}  \bke{\frac13 y \sum_{k=1}^\infty c_k\pd_y\phi_k + \sum_{k=1}^\infty c_k\phi_k } \phi_j\, dx,0,0,0,0,\cdots)^\top.
\]
Therefore, \eqref{eq-carr-6.3.1} is of the form \eqref{eq-carr-6.3.1-simple-form}.
This completes the proof of Proposition \ref{prop-equivalent-systems}.
\end{proof}

\begin{proposition}\label{prop-a-priori}
Given $(\rho_*, R_*, \dot R_*=0)\in \mathcal M_*$. 
Consider $(\rho_0, R_0,\dot{R}_0)\in C^{2+2\al}_r(B_{R_0})\times\R_+\times\R$. 

\begin{enumerate}
\item Let ${\bf w}(0)=(z(0),\mathcal R(0), \dot R(0), c_1(0), c_2(0), \ldots)^\top$, where $z(0) = \rho_0(R_0) - \rho_*$, $\mathcal R(0) = R_0 - R_*$, and $c_k(0) = \int_{B_1} (\rho_0(R_0y) - \rho_0(R_0) - \rho_*) \phi_k(y)dy$, $\phi_k$ is defined in \eqref{eq-eigen-formula}, as in \eqref{eq-perturbation}.

Then, 
\[
{\bf w}(0) = (z(0),\mathcal R(0), \dot R(0), c_1(0), c_2(0), \ldots)^\top \in \ell^2,\qquad
\{j^2c_j(0)\}_{j=1}^\infty\in\ell^2.
\]

\item 
Suppose $\oldnorm{ \bke{\rho_0 - \rho_*, R_0 - R_*, \dot R_0 - \dot{R}_*} }$ is sufficiently small.
Let $(\rho, R)\in C^{2+2\al,1+\al}_{r,t}(B_{R(t)}\times[0,\infty))\times C^{3+\al}_t$ be the solution of \eqref{red-eqns}--\eqref{eq-bv-3.14prime} with the initial data $(\rho_0, R_0,\dot{R}_0)$, obtained in Theorem \ref{thm-asystab}, such that $\oldnorm{ \bke{\rho(\cdot,t) - \rho_{**}, R(t) - R_{**}, \dot R(t) } } + |\ddot R(t)| + |\dddot R(t)| \to0$ as $t\to\infty$ for some $(\rho_{**}, R_{**}, 0 )\in\mathcal M_*$. 
Let ${\bf w}$ be the corresponding solution of the infinite-dimensional dynamical system \eqref{eq-carr-6.3.1-simple-form} in Proposition \ref{prop-equivalent-systems}.

Then,  we have the {\it a priori} bounds for ${\bf w}$: $\{j^2c_j\}_{j=1}^\infty\in\ell^2$, and, as $t\to\infty$, $\dot{\bf w}(t)\to{\bf0}$ and ${\bf w}(t)\to{\bf w}_*$ in $\ell^2$, where ${\bf w}_* = (\rho_{**} - \rho_*, R_{**} - R_*, 0, 0, 0,\ldots)^\top$.
\end{enumerate}

\end{proposition}

\begin{proof}
First of all, setting $u_0(y) = \rho_0(R_0y) - \rho_0(R_0) - \rho_*$, we have $u_0(y) \in C^{2+2\al}_y(B_1)\subset\{u\in L^2(B_1): \De u \in L^2(B_1)\}$, $y=R_0^{-1}r$.
Thus, 
\EQ{\label{eq-cj-fast-decay-initial}
\infty > \int_{B_1} (-\De u_0)^2\, dx = \sum_{j,k=1}^\infty c_j(0)\la_j c_k(0)\la_k(0) \int_{B_1} \phi_j\phi_k\, dx
= \sum_{j=1}^\infty \la_j^2 (c_j(0))^2 = \pi^4 \sum_{j=1}^\infty j^4(c_j(0))^2.
}
So $\{j^2c_j(0)\}_{j=1}^\infty\in\ell^2$. This proves Part (1) of the proposition.

For Part (1), by Proposition \ref{prop-equivalent-systems} \eqref{red-eqns}--\eqref{eq-bv-3.14prime} is equivalent to \eqref{red-eqns-linear-new-1}.
Let $(u,z,\mathcal R)$ be the corresponding solution of \eqref{red-eqns-linear-new-1}.
Since $u(\cdot, t) \in C^{2+2\al}_y\subset\{u\in L^2: \De u \in L^2\}$,
from the same argument as in \eqref{eq-cj-fast-decay-initial} one has $\{j^2c_j\}_{j=1}^\infty\in\ell^2$. 
From the convergence of $(\rho, R,\dot R)$ to $(\rho_*,R_*,0)$, it is obvious that ${\bf w}(t) \to{\bf w}_*$ in $\ell^2$ as $t\to\infty$.
It remains to prove the convergence $\dot{\bf w}(t) \to{\bf 0}$, as $t\to\infty$.
It follows from the equations \eqref{eq-bv-3.10prime} and \eqref{eq-bv-3.15prime} in the original system that, as $t\to\infty$, $\pd_t\rho\to0$ or, equivalently, $|\pd_t u| + |\dot z|\to0$.
Thus, $\dot{\bf w}(t) \to{\bf 0}$ in $\ell^2$ as $t\to\infty$. This proves Proposition \ref{prop-a-priori}.
\end{proof}

\subsection{Spectral analysis of the linear operator}

We now study the spectrum of the linear operator $\mathcal L$, defined in \eqref{eq-matrixC}, acting in the space $\ell^2$.
Recall the definition of Laplace transform and its inverse in \eqref{inversion}.
Taking Laplace transform of the linear system $\dot{\bf w} = \mathcal L{\bf w}$, one derive $(\mathcal L - \tau I)\wt{\bf w}(\tau) = -{\bf w}(0)$, where $I$ is the identity operator.
Denote the spectrum of $\mathcal L$ by
\[
\si(\mathcal L) = \{\tau\in \CC: \mathcal L - \tau I\text{ is non-bijective on $\ell^2$}\}.
\]
Then, formally $\si(\mathcal L)$ consists of all the poles of $\wt{\bf w}(\tau)$.
\begin{proposition}\label{prop-spectrum}
Let $\mathcal L$ the the linear operator defined in \eqref{eq-matrixC}.
Then
\begin{enumerate}
\item 
\[
\si(\mathcal L) = \{0\}\cup\{\tau\in\CC : Q(\tau) = 0\}, 
\]
where $\tau=0$ has multiplicity one, and $Q(\tau)$ is a meromorphic function defined by
\EQ{\label{eq-Q-simplified}
Q(\tau) = \frac1{\Rg T_\infty} \bke{\frac{4\pi}{3\ga} + \frac{8(\ga-1)}{\pi\ga} \sum_{j=1}^\infty \frac{\pi^2\bar\ka}{\pi^2\bar\ka j^2 + \tau} } \bke{\rho_lR_*\tau^2 + \frac{4\mu_l}{R_*}\, \tau - \frac{2\si}{R_*^2}} + 4\pi\,\frac{\rho_*}{R_*}.
}
\item There exists $\beta>0$ such that if $\tau\ne0$ is in $\si(\mathcal L)$, the spectrum of $\mathcal{L}$, then $\Re(\tau)\le-\beta<0$.
Moreover, there exists a constant $C=C(\be)$ such that $\norm{(\mathcal L - \tau I)^{-1}}_{\ell^2\to\ell^2} \le C(\be)$ for all $\tau\neq0$ with $\Re(\tau)>-\beta$.
\item  A negative upper bound for  $-\be$ in terms of physical parameters is displayed in  \eqref{eq-be-def} of Appendix \ref{sec-negative-upper-bound}.
\end{enumerate}
\end{proposition}

\begin{proof}
In the same spirit of Proposition \ref{prop-equivalent-systems}, the linear system $\dot{\bf w} = \mathcal L{\bf w}$ is equivalent to the following linear version of \eqref{red-eqns-linear-new-1}:
\begin{subequations}
\label{red-eqns-linear-new}
\begin{align}
\pd_tu &= \bar\ka \De_yu - \bke{ 1 - \frac1{\ga} } \dot z,\quad 0\le y\le1,\qquad u(1,t)=0,\quad \ t>0,\label{eq-u-pde}\\
\dot{\mathcal R} &= -\frac{\bar\ka R_*}{\rho_*}  \pd_yu(1,t) - \frac{R_*}{3\ga\rho_*}\, \dot z,\quad t>0.\label{eq-dotR-u}\\
z(t) &= \frac1{\Rg T_\infty} \bke{ -\frac{2\si}{R_*^2}\mathcal R + \frac{4\mu_l}{R_*}\, \dot{\mathcal R} + \rho_lR_*\ddot{\mathcal R} },\quad t>0.\label{eq-z}
\end{align}
\end{subequations}

Similar to the proof in Proposition \ref{prop-equivalent-systems}, using the eigenfunction decomposition \eqref{eq-u-expansion} in \eqref{eq-u-pde} and testing the equation against $\phi_k$, we have
\EQ{\label{eq-dotc}
\dot c_k(t) = -\bar\ka\la_kc_k(t) - \Ga_k\dot z(t).
}
Taking Laplace transform of \eqref{eq-dotc}, we have
\[
- c_k(0) + \tau\widetilde{c_k}(\tau) 
= - \bar\ka\la_k\widetilde{c_k}(\tau) - \Ga_k(-z(0) + \tau\widetilde z(\tau)),
\]
or
\EQ{\label{eq-tdc}
\widetilde{c_k}(\tau) = \frac{c_k(0) + \Ga_kz(0)}{\bar\ka\la_k+\tau} - \frac{\Ga_k\tau}{\bar\ka\la_k+\tau}\, \widetilde z(\tau).
}
Using the eigenfunction decomposition \eqref{eq-u-expansion} in \eqref{eq-dotR-u}, we get
\EQ{\label{eq-dotR-u-expand}
\dot{\mathcal R} = \sum_{j=1}^\infty c_j(t) \La_j - \frac{R_*}{3\ga\rho_*}\, \dot z,
}
where $\La_j$ is defined in \eqref{eq-def-La}.
Taking Laplace transform of \eqref{eq-dotR-u-expand}, we deduce
\EQ{\label{eq-tdR}
-\mathcal R(0) + \tau\wt{\mathcal R}(\tau) = \sum_{j=1}^\infty \wt c_j(\tau)\La_j - \frac{R_*}{3\ga\rho_*} (-z(0) + \tau\wt z(\tau)).
}

Taking Laplace transform of \eqref{eq-z}, we derive
\[
\Rg T_\infty \widetilde z(\tau) = -\frac{2\si}{R_*^2} \widetilde R(\tau) + \frac{4\mu_l}{R_*} \bke{-\mathcal R(0) + \tau \widetilde R(\tau) } + \rho_lR_*\bke{-\dot{\mathcal R}(0) - \tau\mathcal R(0) + \tau^2\widetilde R(\tau)},
\]
or
\EQ{\label{eq-tdz}
\bke{\rho_lR_*\tau^2 + \frac{4\mu_l}{R_*}\, \tau - \frac{2\si}{R_*^2}} \widetilde{\mathcal R}(\tau) - \Rg T_\infty\widetilde z(\tau) 
= \rho_lR_* \bke{\dot{\mathcal R}(0) + \tau\mathcal{R}(0)}.
}
Replacing the $\wt c_k$'s and $\wt z$ in \eqref{eq-tdR} using \eqref{eq-tdc} and \eqref{eq-tdz}, multiplying the equation by $4\pi\frac{\rho_*}{R_*}$,  using the identity $\Ga_j\La_j = \frac{2(\ga-1)\bar\ka R_*}{\ga\rho_*}$, we obtain
\EQ{\label{eq-QR}
\tau Q(\tau) \wt{\mathcal R}(\tau) = \text{DATA}(\tau),
}
where $Q(\tau)$ is as in \eqref{eq-Q-simplified}, and $\text{DATA}(\tau)$ is analytic for all $\tau\in\CC$ with $\tau\neq-\bar\ka\la_j = -\pi^2\bar\ka j^2$, $j=1,2,\ldots$, and is defined as
\EQN{
\text{DATA}(\tau) &= 4\pi\frac{\rho_*}{R_*} \Bigg[ \frac{\rho_lR_*\tau}{\Rg T_\infty} \bke{\sum_{j=1}^\infty \frac{\Ga_j\La_j}{\bar\ka\la_j + \tau} + \frac{R_*}{3\ga\rho_*} } \bke{\dot{\mathcal R}(0) + \tau\mathcal R(0)}\\
&\qquad\qquad\quad + \sum_{j=1}^\infty \frac{c_j(0) + \Ga_jz(0)}{\bar\ka\la_j + \tau} + \frac{R_*}{3\ga\rho_*} z(0) + \mathcal R(0) \Bigg].
}
It then follows from \eqref{eq-QR}, \eqref{eq-tdz}, and \eqref{eq-tdc}, that
\EQN{
\wt z(\tau) &= \frac1{\Rg T_\infty } \bkt{ \frac{\rho_lR_*\tau^2 + \frac{4\mu_l}{R_*}\, \tau - \frac{2\si}{R_*^2}}{\tau Q(\tau)}\, \text{DATA}(\tau) - \rho_lR_*\bke{ \dot{\mathcal R}(0) + \tau\mathcal R(0) } },\\
\wt{\mathcal R}(\tau) &= \frac1{\tau Q(\tau)}\, \text{DATA}(\tau),\qquad
\wt{\dot{\mathcal R}}(\tau) = -\mathcal R(0) + \frac1{Q(\tau)}\, \text{DATA}(\tau),\\
\wt{c_k}(\tau) &= \frac{c_k(0) + \Ga_kz(0)}{\bar\ka\la_k+\tau} - \frac{\Ga_k\tau}{\bar\ka\la_k+\tau} \frac1{\Rg T_\infty } \bkt{ \frac{\rho_lR_*\tau^2 + \frac{4\mu_l}{R_*}\, \tau - \frac{2\si}{R_*^2}}{\tau Q(\tau)}\, \text{DATA}(\tau) - \rho_lR_*\bke{ \dot{\mathcal R}(0) + \tau\mathcal R(0) } },
}
which amounts to $\wt{\bf w}(\tau) = (\mathcal L - \tau I)^{-1} \wt{\bf w}(0)$ for all $\tau\in\CC$ with $\tau\neq -\pi^2\bar\ka j^2$, $j=1,2,\ldots$, and $\tau Q(\tau)\neq0$.
By Lemma \ref{lem-negative-upper-bound}, there exists $\be>0$ such that $\Re(\tau) < -\be$ for all $\tau$ satisfying $Q(\tau) = 0$.
Thus, $\norm{(\mathcal L - \tau I)^{-1}}_{\ell^2\to\ell^2} \le C(\be)$  for all $\tau\neq0$ with $\Re(\tau) > -\be$.
This completes the proof of the proposition.
\end{proof}

\begin{proposition}
The linear operator $\mathcal L$ defined in \eqref{eq-matrixC} has a one-dimensional kernel $\ker\mathcal L = \textup{span}({\bf b})$.
\EQ{\label{eq-def-bfb}
{\bf b} = (-\frac{2\si}{\Rg T_\infty R_*^2},1,0,0,0,\ldots)^\top,
}
Moreover, ${\bf b}^\dagger\in \textup{coker}\,\mathcal L $ where
\EQ{\label{eq-def-tilde-bfb}
{\bf b}^\dagger = ( \frac{4\pi}3, 4\pi \frac{\rho_*}{R_*}, 0, \frac{\ga}{\ga-1}\Ga_1, \frac{\ga}{\ga-1}\Ga_2,\ldots ).
}
\end{proposition}
\begin{proof}
It is a direct consequence of Proposition \ref{prop-spectrum} that $\dim\ker\mathcal L = 1$ since $\tau=0$ has multiplicity one.
Moreover, it is easy to check that $\mathcal L{\bf b} = {\bf 0}$.
So $\ker\mathcal L = \textup{span}({\bf b})$.
On the other hand, one can check directly ${\bf b}^\dagger \mathcal L = {\bf 0}$ by using the identity $\sum_{j=1}^\infty\Ga_j^2 = \frac{4(\ga-1)^2\pi}{3\ga^2}$.
We note that ${\bf b}^\dagger$ satisfies the linearized constant mass constraint
\EQ{\label{eq-u-mass-conserve} 
\frac{\ga}{\ga-1} \sum_{j=1}^\infty \Ga_j c_j = \int_{B_1} u = -\frac{4\pi}3 z - 4\pi \frac{\rho_*}{R_*}\, \mathcal R.
}
\end{proof}

Using $\Rg T_\infty \rho_* = p_{\infty,*} + 2\si/R_*$,
\[
\bka{{\bf b}^\dagger, {\bf b} } = \frac{4\pi p_{\infty,*}}{3\Rg T_\infty R_*} + \frac{8\pi\rho_*}{3R_*}.
\]
Normalize ${\bf b}^\dagger$ as
\EQ{\label{eq-def-K}
{\bf b}_0^\dagger = K {\bf b}^\dagger,\qquad
K:=\bke{ \frac{4\pi p_{\infty,*}}{3\Rg T_\infty R_*} + \frac{8\pi\rho_*}{3R_*} }^{-1},
}
so that $\bka{{\bf b}_0^\dagger, {\bf b}}=1$.

\subsection{Toward a center manifold formulation}

Denote $X = \ker \mathcal L = \text{span}({\bf b})$, and let $Y$ be the orthogonal complement of $X$ in $\ell^2$.
Decompose
\[
{\bf w} = {\bf x} + {\bf y},\qquad
{\bf x} = Q_1{\bf w} := \bka{{\bf b}_0^\dagger, {\bf w}}\, {\bf b},\quad
{\bf y} = Q_2{\bf w} := {\bf w} - Q_1{\bf w}.
\]
Since $\mathcal L{\bf b} = {\bf b}^\dagger \mathcal L = {\bf 0}$, we have
\[
Q_1\mathcal{L}{\bf w} = \mathcal{L}Q_1{\bf w} = {\bf 0},\qquad Q_2\mathcal{L}{\bf w} = \mathcal{L}Q_2{\bf w} = \mathcal{L}{\bf w}
\]
In particular, $\mathcal{L}Q_2$ is the restriction of $\mathcal{L}$ on $Y$.
Then we derive, from \eqref{eq-carr-6.3.1-simple-form}, a system of $({\bf x},{\bf y})$
\begin{subequations}
\label{eq-carr-6.3.4}
\begin{align}
\dot{\bf x} &= Q_1\mathcal N({\bf x} + {\bf y}, \dot{\bf x} + \dot{\bf y}) = Q_1\bkt{ \mathcal N^1({\bf x} + {\bf y}) [\dot{\bf x} + \dot{\bf y}] } + Q_1\mathcal N^0({\bf x} + {\bf y}),\label{eq-carr-6.3.4-x}\\ 
\dot{\bf y} &= \mathcal L {\bf y} + Q_2 \mathcal N({\bf x} + {\bf y}, \dot{\bf x} + \dot{\bf y}) = \mathcal L {\bf y} + Q_2 \bkt{ \mathcal N^1({\bf x} + {\bf y}) [\dot{\bf x} + \dot{\bf y}] } + Q_2\mathcal N^0({\bf x} + {\bf y}) .\label{eq-carr-6.3.4-y}
\end{align}
\end{subequations}

In order to apply the center manifold analysis developed in Appendix \ref{sec-center-manifold}, we check the setup of the system \eqref{eq-carr-6.3.4} in the following proposition.

\begin{proposition}\label{prop-invariance-decay}
The subspaces $X$ and $Y$ are $\mathcal L$-invariant and $e^{\mathcal L t}$-invariant, respectively.
Moreover, 
\EQ{\label{eq-fast-decay-Q2}
\norm{e^{\mathcal L t} Q_2 }_{\ell^2\to\ell^2} \le ce^{-\be t},\quad t\ge0,
}
for some $c>0$, where $\be$ given in \eqref{eq-be-def}.
\end{proposition}
\begin{proof}
Obviously, $X$ is $\mathcal L$-invariant since $X= \ker \mathcal L$.
To show $Y$ is $e^{\mathcal L t}$-invariant, let ${\bf y}_0\in Y$ and fix ${\bf x}\in X$, ${\bf x}\neq{\bf 0}$, and ${\bf y}\in Y$, ${\bf y}\neq{\bf 0}$.
Decompose $e^{\mathcal L t}{\bf y}_0 = \al_{\bf x}(t){\bf x} + \al_{\bf y}(t){\bf y}$.
Since $\frac{d}{dt} (e^{\mathcal L t}{\bf y}_0) = \mathcal L (e^{\mathcal L t}{\bf y}_0)$ and ${\bf x}\in X=\ker \mathcal L$, $\dot\al_{\bf x}(t){\bf x} + \dot\al_{\bf y}(t){\bf y} = \al_{\bf y}(t)\mathcal L{\bf y}$ which implies $\dot\al_{\bf x}(t) = 0$.
But since $\al_{\bf x}(0)=0$ in view of ${\bf y}_0\in Y$, we have $\al_{\bf x}\equiv0$. 
So we deduce $e^{\mathcal L t}{\bf y}_0 =  \al_{\bf y}(t){\bf y}\in Y$ and conclude that $Y$ is $e^{\mathcal L t}$-invariant. 
Since $\text{Range}(Q_2) = Y = X^\perp$, $X = \ker\mathcal L$, the operator estimate \eqref{eq-fast-decay-Q2} then follows from the spectrum analysis in Proposition \ref{prop-spectrum} and the Gearhart-Pr\"uss theorem \cite{Gearhart-TAMS1978,Pruss-TAMS1984} for $C_0$ semigroups on Hilbert spaces.
This proves the proposition.
\end{proof}

The system \eqref{eq-carr-6.3.4} is equivalent to those systems in Proposition \ref{prop-equivalent-systems}, namely, \eqref{red-eqns-linear-new-1} and \eqref{eq-carr-6.3.1-simple-form} as well as the original free boundary problem \eqref{red-eqns}--\eqref{eq-bv-3.14prime}.

A curve ${\bf y} = h({\bf x})$, defined for $|{\bf x}|$ small, is said to be an \emph{invariant manifold} for \eqref{eq-carr-6.3.4} if the solution $({\bf x}(t),{\bf y}(t))$ of \eqref{eq-carr-6.3.4} starting from $({\bf x}(0), h({\bf x}(0)))$ satisfies ${\bf y}(t) = h({\bf x}(t))$.
A \emph{center manifold} for \eqref{eq-carr-6.3.4} is an invariant manifold that is tangent to $X$ space at the origin.

\subsection{Manifold of equilibria, $\mathcal M_*$, as a local center manifold through $(\rho_*,R_*)$}

In the following lemma, we show that the manifold of equilibria $\mathcal M_*$ (given in \eqref{eq-equilib-manifold}) is a local center manifold for \eqref{eq-carr-6.3.4}.

\begin{lemma}\label{lem-mfld-equil-ceneter-mfld}
Let $\rho_{**}(\al)$ and $R_{**}(\al)$ satisfy 
\EQ{\label{eq-rho-R-double-star}
\Rg T_\infty \rho_{**}(\al) = p_{\infty,*} + \frac{2\si}{R_{**}(\al)},\qquad
R_{**}(\al) = \frac{-\al+\sqrt{\al^2 + 4R_*^2}}2.
} 
Then $(\rho_{**}(0),R_{**}(0)) = (\rho_*, R_*)$ and
\EQ{\label{eq-center-mfd}
{\bf y} = h({\bf x}) = h(\al{\bf b}) 
:= 
\begin{bmatrix}
\rho_{**}(\al) - \rho_*\\
R_{**}(\al) - R_*\\
0\\
0\\
0\\
\vdots
\end{bmatrix}
}
is a local center manifold for the system \eqref{eq-carr-6.3.4}.
Specifically, for $\al$ small enough the dynamics on the center manifold is trivial. That is, for $({\bf x}(t),{\bf y}(t)) = ({\bf x}(t), h({\bf x}(t)))$, \eqref{eq-carr-6.3.4-x} becomes
\EQ{\label{eq-trivial-dynamics}
\dot\al(t) = 0.
}
\end{lemma}

\begin{proof}
We first show that ${\bf y} = h({\bf x})$ is an invariant manifold for \eqref{eq-carr-6.3.4}.
For ${\bf x}(t) = \al(t){\bf b}$,
\[
{\bf x} + h({\bf x}) = 
\begin{bmatrix}
-\frac{2\si}{\Rg T_\infty R_*^2} \al + \rho_{**}(\al) - \rho_*\\
\al + R_{**}(\al) - R_*\\
0\\
0\\
0\\
\vdots
\end{bmatrix},
\qquad
\dot{\bf x} + h'({\bf x})\dot{\bf x} = \dot\al
\begin{bmatrix}
-\frac{2\si}{\Rg T_\infty R_*^2} + \rho_{**}'(\al)\\
1+R_{**}'(\al)\\
0\\
0\\
0\\
\vdots
\end{bmatrix}.
\]
Thus, for ${\bf F}^1, {\bf G}^1,{\bf H}^1$ and $F^0, G^0, H^0$ defined in \eqref{def-FGH-simple-form} we have
${\bf F}^1({\bf x} + h({\bf x})) = {\bf 0}$, $F^0({\bf x} + h({\bf x})) = 0$,
\[
{\bf G}^1({\bf x} + h({\bf x})) 
= ( J(\al), 0, 0, 0, 0,\cdots)^\top,\qquad
G^0({\bf x} + h({\bf x})) = 0,
\]
where
\EQ{\label{eq-def-J}
J(\al) = \frac{-\bke{\rho_* + \frac{2\si}{\Rg T_\infty R_*} }\al - \rho_*R_{**}(\al) + R_*\rho_{**}(\al)}{3\ga\rho_*\bke{-\frac{2\si}{\Rg T_\infty R_*^2}\al + \rho_{**}(\al)}},
}
and
\[
{\bf H}^1({\bf x} + h({\bf x}))  = (0, 0, \frac{\rho_l(\al + R_{**}(\al) - R_*)}{\Rg T_\infty}, 0, 0, \cdots)^\top,
\]
\[
H^0({\bf x} + h({\bf x})) = \frac{2\si}{\Rg T_\infty R_*^2} \frac{(\al + R_{**}(\al) - R_*)^2}{\al + R_{**}(\al)}.
\]
Thus, by \eqref{def-N1-N0-simple-form} we get
\[
\mathcal N^1({\bf x} + h({\bf x})) \bkt{\dot{\bf x} + h'({\bf x})\dot{\bf x}} = 
\bke{-\frac{2\si}{\Rg T_\infty R_*^2} + \rho_{**}'(\al)} \frac{3\ga\rho_*}{R_*} J(\al)\dot\al
 \begin{bmatrix}
1\\
0\\
0\\
-\Ga_1\\
-\Ga_2\\
\vdots
\end{bmatrix},
\]
and
\[
\mathcal N^0({\bf x} + h({\bf x})) 
= \begin{bmatrix}
0\\
0\\
- \frac{2\si}{\rho_lR_*^3} \frac{(\al + R_{**}(\al) - R_*)^2}{\al + R_{**}(\al)}\\
0\\
0\\
\vdots
\end{bmatrix}.
\]
Therefore, $({\bf x}, h({\bf x}))$ solves \eqref{eq-carr-6.3.4-x} if and only if, by using the identity $\sum_{j=1}^\infty\Ga_j^2 = \frac{4(\ga-1)^2\pi}{3\ga^2}$,
\EQN{
\dot\al {\bf b}
&= \dot{\bf x}
= Q_1\bkt{ \mathcal N^1({\bf x} + h({\bf x}) ) [\dot{\bf x} + h'({\bf x})\dot{\bf x}]} + Q_1\mathcal N^0({\bf x} + h({\bf x}) )\\
&= \bka{{\bf b}_0^\dagger, \mathcal N^1({\bf x} + h({\bf x}) ) [\dot{\bf x} + h'({\bf x})\dot{\bf x}] } {\bf b} + \bka{{\bf b}_0^\dagger, Q_1\mathcal N^0({\bf x} + h({\bf x}) )} {\bf b}\\
&= K \bket{ \bke{-\frac{2\si}{\Rg T_\infty R_*^2} + \rho_{**}'(\al)} \frac{3\ga\rho_*}{R_*} J(\al)\dot\al \bke{\frac{4\pi}3 - \frac{\ga}{\ga-1} \sum_{j=1}^\infty \Ga_j^2} + 0 } {\bf b}\\
&= K\frac{4\pi}{3\ga}\bke{-\frac{2\si}{\Rg T_\infty R_*^2} + \rho_{**}'(\al) } \frac{3\ga\rho_*}{R_*} J(\al) \dot\al {\bf b},
}
which is equivalent to 
\[
\dot\al\bket{ 1 + \frac{4\pi K}{3\ga}\bke{-\frac{2\si}{\Rg T_\infty R_*^2} + \rho_{**}'(\al) } \frac{3\ga\rho_*}{R_*} J(\al) } = 0.
\]
Since $J(\al)\to0$ as $\al\to0$, the above equation yields $\dot\al = 0$ for all $\al$ sufficiently small.
This shows that the dynamics on the local center manifold is trivial.

We now verify equation \eqref{eq-carr-6.3.4-y}.
Note that $({\bf x}, h({\bf x}))$ solves \eqref{eq-carr-6.3.4-y} if and only if 
\EQ{\label{eq-carr-6.3.4-y-on-mfld}
h'({\bf x})\dot{\bf x} = \mathcal L h({\bf x}) + Q_2\bkt{\mathcal N^1({\bf x}+h({\bf x}) [\dot{\bf x} + h'({\bf x})\dot{\bf x}]} + Q_2\mathcal N^0({\bf x}+h({\bf x}).
}
Since $Q_1\bkt{ \mathcal N^1({\bf x} + h({\bf x}) ) [\dot{\bf x} + h'({\bf x})\dot{\bf x}]} = {\bf 0}$ for sufficiently small $\al$ and  $Q_1\mathcal N^0({\bf x} + h({\bf x}) ) \equiv 0$, we have 
$Q_2\bkt{\mathcal N^1({\bf x}+h({\bf x}) [\dot{\bf x} + h'({\bf x})\dot{\bf x}] } = \mathcal N^1({\bf x}+h({\bf x}) [\dot{\bf x} + h'({\bf x})\dot{\bf x}]$ for sufficiently small $\al$ and  $Q_2\mathcal N^0({\bf x}+h({\bf x})) \equiv \mathcal N^0({\bf x}+h({\bf x}))$.
Thus, for sufficiently small $\al$, \eqref{eq-carr-6.3.4-y-on-mfld} becomes
\[
h'({\bf x})\dot{\bf x} = \mathcal L h({\bf x}) + \mathcal N^1({\bf x}+h({\bf x}) [\dot{\bf x} + h'({\bf x})\dot{\bf x}] + N^0({\bf x}+h({\bf x})),
\]
or equivalently,
\EQN{
\dot\al
\begin{bmatrix}
\rho_{**}'(\al)\\
R_{**}'(\al)\\
0\\
0\\
0\\
\vdots
\end{bmatrix}
&= \begin{bmatrix}
0\\
0\\
\frac{\Rg T_\infty}{\rho_lR_*}(\rho_{**}(\al)-\rho_*) + \frac{2\si}{\rho_lR_*^3}(R_{**}(\al)-R_*)\\
0\\
0\\
\vdots
\end{bmatrix}
+ \bke{-\frac{2\si}{\Rg T_\infty R_*^2} + \rho_{**}'(\al)} J(\al)\dot\al
 \begin{bmatrix}
1\\
0\\
0\\
-\Ga_1\\
-\Ga_2\\
\vdots
\end{bmatrix}\\
&\quad +  \begin{bmatrix}
0\\
0\\
-\frac{2\si}{\rho_lR_*^3} \frac{(\al+R_{**}(\al)-R_*)^2}{\al+R_{**}(\al)}\\
0\\
0\\
\vdots
\end{bmatrix}.
}
Since $\dot\al=0$ for $\al$ sufficiently small, the above equation is further equivalent to
\[
0 = \frac{\Rg T_\infty}{\rho_lR_*}(\rho_{**}(\al)-\rho_*) + \frac{2\si}{\rho_lR_*^3}(R_{**}(\al)-R_*) -\frac{2\si}{\rho_lR_*^3} \frac{(\al+R_{**}(\al)-R_*)^2}{\al+R_{**}(\al)}.
\]
Using \eqref{eq-rho-R-double-star} and $\Rg T_\infty\rho_* = p_{\infty,*} + 2\si/R_*$, we derive from above a quadratic equation
\[
R_{**}^2 + \al R_{**} - R_*^2 = 0,
\]
for which the solution is $R_{**}(\al) = \bke{-\al\pm\sqrt{\al^2 + 4R_*^2}} / 2$.
In view of the condition $R_{**}(0) = R_*$, we choose the plus sign:
\[
R_{**}(\al) = \frac{-\al+\sqrt{\al^2 + 4R_*^2}}2.
\]

Finally, we verify that ${\bf y} = h({\bf x})$ is tangent to $X$ at the origin.
Differentiating \eqref{eq-rho-R-double-star} with respect to $\al$, evaluating at $\al=0$ and using $R_{**}(0) = R_*$, we get $\Rg T_\infty \rho_{**}'(0) = -(2\si/R_*^2) R_{**}'(0)$.
So,
\[
\frac{d}{d\al} \Big|_{\al=0} h(\al{\bf b}) = 
\begin{bmatrix}
\rho_{**}'(0)\\
R_{**}'(0)\\
0\\
0\\
0\\
\vdots
\end{bmatrix}
=
\begin{bmatrix}
-\frac{2\si}{\Rg T_\infty R_*^2}\, R_{**}'(0)\\
R_{**}'(0)\\
0\\
0\\
0\\
\vdots
\end{bmatrix}
= R_{**}'(0) {\bf b} \in X.
\]
This completes the proof of Lemma \ref{lem-mfld-equil-ceneter-mfld}.
\end{proof}

\subsection{Nonlinear estimates}

In order to apply Proposition \ref{prop-stab-center-mfd}, we derive the following estimates for the nonlinear terms in the dynamical system \eqref{eq-carr-6.3.1-simple-form}.

\begin{proposition}\label{prop-nonlinear-est}
Let $\mathcal N({\bf w}, {\bf p})$ be given in \eqref{eq-def-N}. 
Let ${\bf w}=(z,\mathcal R,\dot{\mathcal R}, c_1, c_2,\ldots)^\top$ be obtained from the solution $(\overline\rho, R)\in C^{2+2\al,1+\al}_{y,t}\times C^{3+\al}_t$ of \eqref{red-eqns-fix-domain} by means of \eqref{eq-perturbation} and \eqref{eq-u-expansion}. 
Then $\mathcal N({\bf w}, {\bf p}) \in \ell^2$.
Moreover,
$\mathcal N({\bf 0}, {\bf p}) = {\bf 0}$ for all ${\bf p}$ and $\pd_{\bf w} \mathcal N({\bf 0}, {\bf 0}) = \pd_{\bf p} \mathcal N({\bf 0}, {\bf 0}) = {\bf O}$.
\end{proposition}

\begin{proof}
We first show that $\mathcal N({\bf w}, {\bf p}) \in \ell^2$. 
Note that $G({\bf w}, {\bf p})$ is well-defined since $\{jc_j\}_{j=1}^\infty\in\ell^1$: $\sum j|c_j| = \sum j^{-1} j^2 |c_j| \le \bke{\sum j^{-2}}^{1/2} \bke{\sum j^4 c_j^2}^{1/2} < \infty$.
Since $\{\Ga_j\sim (-1)^{-j}/j \}_{j=1}^\infty\in \ell^2$, it remains to show $\{F_j({\bf w}, {\bf p})\}_{j=1}^\infty\in\ell^2_j$.
Indeed, using $|\nb u(y,t)|^2 \le E$, where $E>0$ is a constant independent of $y$ and $t$, we have
\EQN{
F_j({\bf w}, {\bf p}) 
&= \int_{B_1} F({\bf w}, {\bf p}) \phi_j\, dx \\
&= O(1) \bkt{ j^2c_j + \Ga_j + \bke{\frac13 \sum_{k=1}^\infty c_k \int_{B_1} y\pd_y\phi_k\phi_j\, dx + \sum_{k=1}^\infty c_k \int_{B_1} \phi_k\phi_j\, dx }}
}
Since $\int_{B_1} \phi_k\phi_j\, dx = \de_{kj}$ and
\EQ{\label{eq-int-y-pdphik-phij}
\int_{B_1} y\pd_y\phi_k\phi_j\, dx 
&= 4\pi \int_0^1 \bke{\frac{k\pi\cos(k\pi y)}{\sqrt{2\pi}y} - \frac{\sin(k\pi y)}{\sqrt{2\pi}y^2} } \frac{\sin(j\pi y)}{\sqrt{2\pi}y}\, y^3\, dy\\
&= 2k\pi\int_0^1 \cos(k\pi y)\sin(j\pi y) y\, dy - 2\int_0^1 \sin(k\pi y)\sin(j\pi y)\, dy\\
&=
\begin{cases}
-\frac12 - 1 = -\frac32,&\ \text{ if }j=k,\\
k\bke{\frac{(-1)^{k-j}}{k-j} - \frac{(-1)^{k+j}}{k+j}} = (-1)^{k+j} \frac{2jk}{k^2-j^2},&\ \text{ if }j\neq k,
\end{cases}
}
we have
\[
F_j({\bf w}, {\bf p})  = O(1) \bkt{ j^2c_j + j^{-1} + \bke{\frac12 c_j + \sum_{k\neq j} c_k (-1)^{k+j} \frac{2jk}{k^2-j^2} } },
\]
where the first three terms are obviously in $\ell^2$. 
For the last term,
by Minkowski inequality,
\EQN{
\sum_{j=1}^\infty& \bke{\sum_{k\neq j}^\infty c_k (-1)^{k+j} \frac{2jk}{k^2-j^2} }^2
\le  \bkt{ \sum_{k=1}^\infty \bke{ \sum_{j\neq k} c_k^2 \bke{\frac{2jk}{k^2-j^2}}^2 }^{\frac12} }^2
\sim \bkt{ \sum_{k=1}^\infty k|c_k| \bke{ \sum_{j\neq k} \frac{j^2}{(k^2-j^2)^2} }^{\frac12} }^2\\
&\lec \bkt{ \sum_{k=1}^\infty k|c_k| }^2 < \infty,\quad \text{ since $\{jc_j\}_{j=1}^\infty\in\ell^1$,}
}
where we've used
\EQN{
\sum_{j\neq k}& \frac{j^2}{(k^2 - j^2)^2} 
\sim \sum_{j\neq k} \bke{ \frac1{k-j} - \frac1{k+j} }^2
= \sum_{j\neq k} \bkt{ \frac1{(k-j)^2} - \frac2{k^2-j^2} + \frac1{(k+j)^2} }\\
&\le \sum_{j\neq k} \frac1{(k-j)^2} - 2 \sum_{j\neq k} \frac1{k^2-j^2} + \sum_{j=1}^\infty \frac1{j^2}
= \bke{\sum_{j=1}^{k-1} \frac1{(k-j)^2} + \sum_{j=k+1}^\infty \frac1{(k-j)^2} } + \frac3{2k^2} + \frac{\pi^2}6\\
&\le \bke{\sum_{j=1}^\infty \frac1{j^2} + \sum_{j=1}^\infty \frac1{j^2} } + \frac3{2} + \frac{\pi^2}6
= \frac3{2} + \frac{3\pi^2}6,\quad \text{ independent of $k$.}
}
Therefore, $\{F_j({\bf w}, {\bf p})\}_{j=1}^\infty\in\ell^2_j$, and thus $\mathcal N({\bf w}, {\bf p}) \in \ell^2$.

It is readily see from \eqref{def-F}--\eqref{def-H} that $\mathcal N({\bf 0}, {\bf p}) = {\bf 0}$ for all ${\bf p}$.
To estimate the derivative, we adopt the form $\mathcal N({\bf w},{\bf p}) = \mathcal N^1({\bf w}){\bf p} + \mathcal N^0({\bf w})$ derived in \eqref{eq-N-simple-form}.
Note that from \eqref{def-FGH-simple-form}, $|F^1({\bf w})| + |G^1({\bf w})| + |H^1({\bf w})| = O(\norm{\bf w})$ and $|F^0({\bf w})| + |G^0({\bf w})| + |H^0({\bf w})| = O(\norm{\bf w}^2)$.
Thus, $\mathcal N^1({\bf 0}) = {\bf O}$, $\mathcal N^0({\bf 0}) = {\bf 0}$, and $\pd_{\bf w} \mathcal N^0({\bf 0}) = {\bf O}$ by \eqref{def-N1-N0-simple-form}. 
Therefore, we have $\mathcal N({\bf 0}, {\bf p}) = \mathcal N^1({\bf 0}){\bf p} + \mathcal N^0({\bf 0}) = {\bf 0}$ for all ${\bf p}$, $\pd_{\bf w} \mathcal N({\bf 0}, {\bf 0}) = \pd_{\bf w}\mathcal N^0({\bf 0}) = {\bf O}$, $\pd_{\bf p} \mathcal N({\bf 0}, {\bf 0}) = \mathcal N^1({\bf 0}) = {\bf O}$,
completing the proof of the proposition.
\end{proof}

\subsection{Proof of Theorem \ref{thm-nonlinear-exp-decay}; nonlinear exponential stability of the manifold of equilibria $\mathcal M_*$}\label{sec-main-proof}

The proof of Theorem \ref{thm-nonlinear-exp-decay} is based on a center manifold analysis of the system \eqref{eq-carr-6.3.4} adapted from a quasi-linear PDE with {\it a priori} estimates; see Proposition \ref{prop-stab-center-mfd}.
Proposition \ref{prop-stab-center-mfd} requires that we verify 
i) a Lipschitz estimate for the nonlinear term, 
ii) a linear exponential decay estimate, 
iii) existence of a local center manifold, and 
iv) an $\ell^2$ {\it a priori} bound on $\dot{\bf w}$.
\EN{
\item [i)] The nonlinear Lipschitz estimate follows from the expression of the nonlinear term in \eqref{def-FGH-simple-form} and is proved in Proposition \ref{prop-nonlinear-est}. 
\item [ii)] Linear exponential decay follows from Proposition \ref{prop-invariance-decay}.
\item [iii)] The existence of a local center manifold follows from the identification of the manifold of equilibria $\mathcal M_*$ as a local center manifold shown in Lemma \ref{lem-mfld-equil-ceneter-mfld}.
\item [iv)] The {\it a priori} estimate follows from Proposition \ref{prop-a-priori}. 
}
With these observations in place, we argue as follows. 

Recall the free boundary problem \eqref{red-eqns} is equivalent to the system \eqref{red-eqns-fix-domain} in the fixed domain $B_1$, and it is also equivalent to the dynamical system \eqref{eq-carr-6.3.1-simple-form} by Proposition \ref{prop-equivalent-systems}.
From the asymptotic stability result in Theorem \ref{thm-asystab}, 
we have, as $t\to\infty$, that ${\bf w}(t) \to {\bf w}_*$, where ${\bf w}_*$ is described in Proposition \ref{prop-a-priori}.
Also, one has $|\ddot R(t)| + |\dddot R(t)|\to 0$ in Theorem \ref{thm-asystab}.
Moreover, it follows from the equation \eqref{eq-bv-3.10prime-fix-domain} that $\pd_t\overline\rho\to0$ uniformly in $B_1$.
So $\dot{\bf w}(t)\to{\bf 0}$ in $\ell^2$ as $t\to\infty$, and thus we may assume that $\sup_{t\ge0} \norm{\dot{\bf w}(t)}$ is arbitrarily small.

We have shown in Lemma \ref{lem-mfld-equil-ceneter-mfld} that the equation on the local center manifold is trivial, \eqref{eq-trivial-dynamics}.
Therefore, by applying Proposition \ref{prop-stab-center-mfd} to the system \eqref{eq-carr-6.3.1} with the nonlinear estimates Proposition \ref{prop-nonlinear-est}, we conclude that the convergence of ${\bf w}(t)\to{\bf w}_*$ is exponentially fast in $\ell^2$.

This amounts to the exponential convergence of $(u,z,\mathcal R)$ to some $(u_{**},z_{**},\mathcal R_{**})$, where $(\rho_*+u_{**}+z_{**}, R_*+\mathcal R_{**}, 0)$ lies on the manifold of equilibria $\mathcal M_*$, in the following $L^2$ sense: for some $\be_0>0$
\EQ{\label{eq-exp-decay-u2-dotR}
\int_{B_1} (u-u_{**})^2\, dx + (z-z_{**})^2 + (\mathcal R-\mathcal R_{**})^2 + \dot{\mathcal R}^2 = O\bke{ e^{-2\be_0 t}  }\ \text{ as $t\to\infty$.}
}

\bigskip\noindent{\bf From time-decay in $L^2$  norms to time-decay of $L^\infty$ norms of gradients.}
We next want to use the $L^2$ bound of \eqref{eq-exp-decay-u2-dotR} to obtain gradient bounds in $L^\infty$.
For this we make use of the following inequality which is based on the equation and an interpolation inequality.

\begin{lemma}\label{lem-H1-energy-ineq}
Assume that $(u,z,\mathcal R)$ is the solution of the system \eqref{red-eqns-linear-new-1} and that $F$ is given in \eqref{red-eqns-linear-nonlinear-term-f}. Then for any constant $u_{**}$
\EQ{\label{eq-H1-energy-ineq}
\frac12\, \frac{d}{dt} \int_{B_1} |\nb_y u|^2\, dx 
\le -\frac{\bar\ka}2\, \frac{\bke{\int_{B_1} |\nb_y u|^2\, dx + 4\pi u_{**}\pd_yu(1,t) }^2}{\int_{B_1} (u-u_{**})^2\, dx} + \frac1{2\bar\ka} \int_{B_1} F^2\, dx + 4\pi\bke{1-\frac1\ga} \dot z\pd_yu(1,t).
}
\end{lemma}
\begin{proof}
Taking gradient of \eqref{eq-u-pde-1}, we have $\nb_y(\pd_t u )= \bar\ka\nb_y\De_yu + \nb_y F$.
Thus, we have
\EQN{
\frac12& \frac{d}{dt} \int_{B_1} |\nb_y u|^2\, dx 
= \int_{B_1} \nb_y u\cdot\nb_y(\pd_tu)\, dx
= \bar\ka \int_{B_1} \nb_y u\cdot\nb_y\De_y u\, dx + \int_{B_1} \nb_yu\cdot\nb_yF\,dx\\
&= -\bar\ka \int_{B_1} (\De_y u)^2\, dx - \int_{B_1} (\De_yu)F\, dx + \bar\ka \int_{\pd B_1} \De_y u(\nb_y u\cdot\hat{\bf n})\, dS + \int_{\pd B_1} F (\nb_y u\cdot\hat{\bf n})\, dS\\
&= -\bar\ka \int_{B_1} (\De_y u)^2\, dx - \int_{B_1} (\De_yu)F\, dx + 4\pi (\bar\ka\De_y u(1,t) + F(1,t) )\pd_yu(1,t)
}
Since $u(1,t)=0$ implies $\pd_tu(1,t)=0$, evaluating \eqref{eq-u-pde-1} at the boundary yields
\[
\bar\ka \De_yu(1,t) + F(1,t) = \bke{1-\frac1\ga}\dot z(t).
\]
Substituting the above equation for $\bar\ka\De_y u(1,t) - F(1,t)$ and using Young's inequality, we obtain
\EQ{\label{eq-H1-energy-law}
\frac12& \frac{d}{dt} \int_{B_1} |\nb_y u|^2\, dx = -\bar\ka \int_{B_1} (\De_y u)^2\, dx - \int_{B_1} (\De_yu)F\, dx + 4\pi\bke{1-\frac1\ga} \dot z\pd_yu(1,t)\\&
\le -\bar\ka \int_{B_1} (\De_y u)^2\, dx + \frac{\bar\ka}2 \int_{B_1} (\De_y u)^2\, dx + \frac1{2\bar\ka} \int_{B_1} F^2\, dx + 4\pi\bke{1-\frac1\ga} \dot z\pd_yu(1,t).
}
Now, since $\int_{B_1} |\nb_y u|^2\, dy = - \int_{B_1} (u-u_{**}) \De_yu\, dy - 4\pi u_{**}\pd_yu(1,t)$, by H\"older's inequality we have the interpolation inequality 
\[
\bke{\int_{B_1} |\nb_y u|^2\, dx + 4\pi u_{**} \pd_yu(1,t) }^2 \le \int_{B_1} (u-u_{**})^2\, dx \cdot \int_{B_1}(\De_y u)^2\, dx.
\]
Using the interpolation inequality for $\int_{B_1}(\De_y u)^2\, dx$ in \eqref{eq-H1-energy-law}, the lemma follows.
\end{proof}

With Lemma \ref{lem-H1-energy-ineq} in hand, we are able to obtain a decay rate of $\nb_y u$ from the decay rate of $u$.

\begin{proposition}\label{prop-decay-nbu}
Let $(u,z,\mathcal R)$ be the solution of the the system \eqref{red-eqns-linear-new-1} with the convergence rate in \eqref{eq-exp-decay-u2-dotR}, then
\[
\norm{\nb_y u(\cdot,t)}_{L^\infty(B_1)} = O\bke{e^{-\frac{\be_0}{10}\, t}}.
\]
\end{proposition}
\begin{proof}
First recall Proposition \ref{prop-equivalent-systems} that the system \eqref{red-eqns-linear-new-1} is equivalent to \eqref{red-eqns}--\eqref{eq-bv-3.14prime}.

\EN{
\item It follows from the Lyapunov stability result for \eqref{red-eqns}--\eqref{eq-bv-3.14prime} in Theorem \ref{thm-bv-4.1-anyperturb} that $u\in C^{2+2\al,1+\al}_{y,t}$ and $\dot z\in C^\al_t$. 
Hence, $\frac{d}{dt} \int_{B_1} |\nb_y u|^2\, dx$ is bounded in $t$. 

\item Moreover, the asymptotic stability result in Theorem \ref{thm-asystab} implies that $\dot z\to0$ and $\norm{\nb_yu(\cdot,t)}_{L^\infty(B_1)}\to0$ as $t\to\infty$.
It also implies that $\int_{B_1} F^2(x,t)\, dx\to0$ as $t\to\infty$, where $F$ given in \eqref{red-eqns-linear-nonlinear-term-f}.

\item Since $z(t) - z_{**} = O(e^{-\be_0 t})$, $\int_0^t (z(\tau) - z_{**}) e^{(\be_0/2)\tau} d\tau$ has a finite limit as $t\to\infty$ and $(z(t) - z_{**}) e^{(\be_0/2) t}$ is uniformly continuous for $t>0$. By the Barbalat's lemma \eqref{eq-barbalat}, $\frac{d}{dt} \bkt{(z(t) - z_{**}) e^{(\be_0/2) t}} \to0$ as $t\to\infty$. This implies $\dot z(t) e^{(\be_0/2) t} + (\be_0/2) \bke{z(t) - z_{**}} e^{(\be_0/2) t} \to0$ as $t\to\infty$.
In particular, we have $\dot z = O(e^{-(\be_0/2)t})$.
It then follows form \eqref{eq-dotR-u-1} and $\dot{\mathcal R} = O(e^{-\be_0 t})$ that $\pd_yu(1,t) = O(e^{-(\be_0/2)t})$.
}
We claim that the quotient 
\EQ{\label{eq-quotient}
\frac{\bke{\int_{B_1} |\nb_y u|^2\, dx + 4\pi u_{**}\pd_yu(1,t) }^2}{\int_{B_1} (u - u_{**})^2\, dx}
}
must be bounded. 
Indeed, if the quotient \eqref{eq-quotient} was unbounded, in view of (1) and (2) the right hand side of \eqref{eq-H1-energy-ineq} would become $-\infty$ as $t\to\infty$ while the left hand side is bounded, violating the inequality \eqref{eq-H1-energy-ineq} in Lemma \ref{lem-H1-energy-ineq}.

Since that $\int_{B_1} (u-u_{**})^2\, dx$ decays exponentially at the rate $2\be_0$ as in \eqref{eq-exp-decay-u2-dotR} and that the quotient \eqref{eq-quotient} is bounded, 
\[
\int_{B_1} |\nb_y u|^2\, dx + 4\pi u_{**}\pd_yu(1,t)  = O(e^{- \be_0 t}).
\] 
Thus, (3) implies $\int_{B_1} |\nb_y u|^2\, dx = O(e^{-(\be_0/2) t})$.
By the interpolation inequality in Lemma \ref{lem-interpolate},
\EQN{
\norm{\nb_y u}_{L^\infty(B_1)} 
\le C_1 \norm{\nb_y u}_{L^2(B_1)}^{\frac25} \norm{\nb_y u}_{C^{1+2\al}_y(B_1)}^{\frac35} + C_2 \norm{\nb_y u}_{L^2(B_1)}
= O\bke{e^{-\frac{\be_0}{10}\,t} }.
}
This proves the proposition.
\end{proof}

Using Proposition \ref{prop-decay-nbu}, we derive the exponential decay of $u-u_{**}$ in the higher norm $\norm{\,\cdot\,}_{W^{1,\infty}}$ from $L^2$ convergence.
By bootstrapping the argument in the proof of Proposition \ref{prop-decay-nbu}, we derive the exponential decay of $u$ in $C^{2+2\al}_y$. 
This amounts to the exponential rate of convergence of $(\rho,R,\dot R)$ toward the center manifold $\mathcal M_*$ in $\oldnorm{\cdot}$ and completes the proof of Theorem \ref{thm-nonlinear-exp-decay}.
\qed

\appendix

\section{Spherically symmetric equilibria of the full liquid / gas model}\label{sec-steadystate}

In this appendix, we prove Proposition \ref{prop-equilib-original}. That is, we show that \eqref{eq-equilibrium} is the unique regular spherically symmetric equilibrium solution to the system \eqref{eq-1.1}--\eqref{eq-1.3} under the radiation condition \eqref{eq-radiation-condition} for $T_l$.

\begin{proof}[Proof of Proposition \ref{prop-equilib-original}]

We consider the full liquid / gas model \eqref{eq-1.1}--\eqref{eq-1.3} and prove Proposition \ref{prop-equilib-original}.
Steady-state solutions of \eqref{eq-1.1}--\eqref{eq-1.3} solve

\begin{subequations}\label{eq-1.1-equilib}
\begin{empheq}[right=\empheqrbrace\text{in $\R^3\setminus \Om_*$,}]{align}
0 =&\, \nu_l\De{\bf v}_{l,*} - {\bf v}_{l,*}\cdot\nb{\bf v}_{l,*} - \frac1{\rho_{l,*}} \nb p_{l,*}, \label{eq-1.1-a-equilib}\\
\div {\bf v}_{l,*} =&\, 0, \label{eq-1.1-b-equilib}\\
\rho_l c_l {\bf v}_{l,*}\cdot\nb T_{l,*} =&\, \ka_l\De T_{l,*} + 2\mu_l\mathbb D({\bf v}_{l,*}) : \nb{\bf v}_{l,*}, \label{eq-1.1-c-equilib}
\end{empheq}
\end{subequations}

 \begin{subequations}\label{eq-1.2-equilib}
\begin{empheq}[right=\empheqrbrace\text{in $\Om_*$,}]{align}
\div(\rho_{g,*}{\bf v}_{g,*}) =&\, 0, \label{eq-1.2-a-equilib}\\
\rho_{g,*}{\bf v}_{g,*}\cdot\nb{\bf v}_{g,*} =&\,
\mu_g\De{\bf v}_{g,*} - \nb p_{g,*}, \label{eq-1.2-b-equilib}\\
\rho_{g,*} T_{g,*} {\bf v}_{g,*}\cdot\nb s_*  =&\, \ka_g\De T_{g,*} + 2 \mu_g\mathbb D({\bf v}_{g,*}): \nb{\bf v}_{g,*}\label{eq-1.2-c-equilib}\\
& - \bke{ \frac23\mu_g - \ze_g}  (\div{\bf v}_{g,*})^2, \notag\\
p_{g,*} =&\, \Rg T_{g,*} \rho_{g,*}, \label{eq-1.2-d-equilib}\\
s_* =&\, c_v \log\bke{\dfrac{p_{g,*}}{\rho_{g,*}^\ga} }, \label{eq-1.2-e-equilib}
\end{empheq}
\end{subequations}

 \begin{subequations}\label{eq-1.3-equilib}
\begin{empheq}[right=\empheqrbrace\text{on $\pd\Om_*$.}]{align}
&\hspace{-0.34cm}{\bf v}_{l,*}(\boldsymbol{\om},t)\cdot\hat{\bf n} = {\bf v}_{g,*}(\boldsymbol{\om},t)\cdot\hat{\bf n} = 0, \label{eq-1.3-a-equilib}\\
\hat {\bf n} \cdot& \bke{-p_{l,*}\mathbb I + 2\mu_l\mathbb D({\bf v}_{l,*})}\label{eq-1.3-b-equilib}\\ 
&- \hat {\bf n} \cdot \bkt{-p_{g,*}\mathbb I + 2\mu_g\bke{\mathbb D({\bf v}_{g,*}) - \frac13(\div {\bf v}_{g,*})\mathbb I} + \ze_g(\div{\bf v}_{g,*})\mathbb I} = \si \hat {\bf n} (\nb_S\cdot {\bf n}), \notag\\
&\hspace{-0.34cm}T_{g,*} = T_{l,*}, \label{eq-1.3-c-equilib}
\end{empheq}
\end{subequations}
For the spherically symmetric case, \eqref{eq-1.1-b-equilib} we have ${\bf v}_{l,*}(x) = v_{l,*}(r) \frac{x}{r}$. Therefore,
 $0=\div{\bf v}_{l,*}= \pd_r v_{l,*}(r) + (2/r)v_{l,*}(r)$ and hence
\EQ{\label{eq-bv-2.1}
\frac1{r^2}\pd_r(r^2 v_{l,*}(r)) = 0,\quad r\ge R_*.
}
Therefore,
\EQ{\label{eq-bv-2.2}
v_{l,*}(r) = \frac{a}{r^2},\quad r\ge R_*,
}
for constant $a$.
But the boundary condition \eqref{eq-1.3-a-equilib} implies $v_{l,*}(R_*) = 0$. So $a=0$, and thus $v_{l,*}\equiv0$.
Therefore, \eqref{eq-1.1-a-equilib} becomes $\nb p_{l,*} = 0$. So the pressure $p_{l,*}$ is a constant and equal to its value at infinity, $p_{\infty,*}$.

For the gas velocity $v_{g,*}$, \eqref{eq-1.2-a-equilib} becomes
\[
\frac1{r^2} \pd_r(\rho_{g,*} v_{g,*})= 0,\quad 0\le r\le R_*,
\]
which implies $\rho_{g,*} v_{g,*}$ is a constant.
Again, the boundary condition \eqref{eq-1.3-a-equilib} implies $v_{g,*}(R_*) = 0$. So $\rho_{g,*} v_{g,*}\equiv0$.
But $\rho_{g,*}\neq0$ since otherwise $s_*$ is singular in \eqref{eq-1.2-e-equilib}.
Therefore, ${\bf v}_{g,*}\equiv{\bf 0}$ and thus \eqref{eq-1.2-b-equilib} becomes $\nb p_{g,*} = 0$.
So $p_{g,*}$ is a constant.
Moreover, by $v_{l,*}=v_{g,*}\equiv0$, \eqref{eq-1.3-b-equilib} yields $-p_{l,*} + p_{g,*} = \frac{2\si}{R_*}$.
So $p_{g,*} = p_{\infty,*} + \frac{2\si}{R_*}$.

For the equations of the temperatures, due to $v_{l,*}=0$ \eqref{eq-1.1-c-equilib} becomes $\De T_{l,*} = 0$ in $\R^3\setminus B_{R_*}$, or, in spherical coordinates,
\[
\frac1{r^2} \pd_r(r^2 \pd_r T_{l,*}) = 0,\quad r\ge R_*,
\]
which implies 
\[
\pd_r T_{l,*}=\frac{a_1}{r^2},\quad r\ge R_*
\] 
for some constant $a_1$.
Integrating over $(r,\infty)$ we get
\[
T_{l,*}(r) = T_\infty - \frac{a_1}r,\quad r\ge R_*.
\]
By the radiation condition \eqref{eq-radiation-condition}, $T_{l,*}(r) = T_\infty + o(r^{-1})$ as $r\to\infty$. 
This gives $a_1=0$ and thus $T_{l,*}\equiv T_\infty$.
On the other hand, \eqref{eq-1.2-c-equilib} becomes $\De T_{g,*} = 0$ in $B_{R_*}$ since $v_{g,*}=0$. 
Since $T_{g,*}$ is regular, $T_{g,*}\equiv T_{g,*}(R_*)=T_\infty$ by the maximum principle.

For the gas density $\rho_{g,*}$, by \eqref{eq-1.2-d-equilib} $\rho_{g,*} = \frac{p_{g,*}}{\Rg T_\infty} = \frac1{\Rg T_\infty} \bke{p_{\infty,*}+\frac{2\si}{R_*}}$.
Due to the conservation of mass \eqref{eq-mass-preserve},
\[
M := \int_{B_{R_0}} \rho_0 
= \lim_{t\to\infty} \int_{B_{R(t)}} \rho_g(\cdot,t)
= \frac{4\pi}3 \rho_{g,*} R_*^3,
\]
where $(\rho_0(x),R_0)$, $\rho_0(x)\ge0, R_0>0$, is the initial data.
Therefore, the steady state $(\rho_{g,*},R_*)$ can be obtained by solving 
 \begin{subequations}\label{eq-steadystate-appdx}
\begin{empheq}[right=\empheqrbrace]
{align}
\dfrac{4\pi}3 \rho_{g,*} R_*^3 &= M, \label{eq-steadystate-a-appdx}\\
\rho_{g,*} &= \frac1{\Rg T_\infty} \bke{p_{\infty,*} + \dfrac{2\si}{R_*}}. \label{eq-steadystate-b-appdx}
\end{empheq}
\end{subequations}
In particular, the stationary radius $R_*$ satisfies the cubic equation
\EQ{\label{eq-cubic-appdx}
p_{\infty,*} R_*^3 + 2\si R_*^2 - \frac{3\Rg T_\infty M}{4\pi} = 0.
}
It is readily seen that for any $M>0$, the cubic equation has a unique positive root $R_*[M]$. 
This proves Proposition \ref{prop-equilib-original}.
\end{proof}

\section{Derivation of the reduced system for $\rho(r,t)$ and $R(t)$: Proof of Proposition \ref{prop:reduction}}\label{sec-reduction}

Considering the uniformity of the pressure $p_g$ in \eqref{eq1.2simplified-b}, we can eliminate $T_g$ by plugging \eqref{eq1.2simplified-d} into \eqref{eq1.2simplified-c} and deduce
\EQ{\label{eq-bv-3.4}
\pd_t s + {\bf v}_g\cdot\nb s = \ka_g \De\bke{\frac1{\rho_g}}.
}
Plugging \eqref{eq1.2simplified-e} into the left hand side of \eqref{eq-bv-3.4} and using \eqref{eq1.2simplified-b}, we have
\EQ{\label{eq-bv-3.5}
c_v \bket{\frac{\pd_t p_g}{p_g} - \frac{\ga}{\rho_g} \bkt{\pd_t\rho_g + {\bf v}_g\cdot\nb\rho_g}} = \ka_g \De\bke{\frac1{\rho_g}}.
}
Using \eqref{eq1.2simplified-a} in \eqref{eq-bv-3.5}, we obtain
\EQ{\label{eq-bv-3.5-nonradial}
c_v \bket{\frac{\pd_t p_g}{p_g} + \ga \div{\bf v}_g} = \ka_g \De\bke{\frac1{\rho_g}},
}
Therefore, the system \eqref{eq1.2simplified} is reduced to 
 \begin{subequations}\label{eq1.2simplified1}
\begin{empheq}[right=\empheqrbrace\text{in $\Om(t)$, $t>0$.}]{align}
\pd_t \rho_g + \div(\rho_g{\bf v}_g) =& 0, \label{eq1.2simplified1-a}\\
\dfrac{\pd_t p_g}{p_g} =& \dfrac{\ka}{c_v} \De\bke{\dfrac1{\rho_g}} - \ga \div {\bf v}_g, \label{eq1.2simplified1-b}
\end{empheq}
\end{subequations}
Expanding the term $\div(\rho_g{\bf v}_g)$ in \eqref{eq1.2simplified1-a} and substituting $\div{\bf v}_g$ using \eqref{eq1.2simplified1-b}, and using the elementary identity 
\EQ{\label{eq-log}
\rho_g \De\bke{\frac1{\rho_g}} = -\De\log\rho_g + \frac{|\nb\rho_g|^2}{\rho_g^2},
} 
we get
\EQ{\label{eq-bv-3.10-nonspherical}
\pd_t\rho_g = \frac\ka{\ga c_v} \De\log\rho_g - \frac{\ka}{\ga c_v} \frac{|\nb\rho_g|^2}{\rho_g^2} - {\bf v}_g\cdot\nb\rho_g + \frac{\pd_tp_g}{\ga p_g} \rho_g.
}

Assuming the bubble is a sphere $B_{R(t)}$ and solutions are spherically symmetric, and
recalling we denoted the radial components of the gas and liquid velocity by $v_g(r,t)$ and $v_l(r,t)$, respectively,
the systems \eqref{eq1.1simplified} and \eqref{eq1.2simplified1} become
 \begin{subequations}\label{eq-1.1-simplified-1}
\begin{empheq}[right=\empheqrbrace\text{for $r\ge R(t)$, $t>0$,}]{align}
\pd_t v_l = \nu_l \bke{ \De_r v_l - \dfrac{2v_l}{r^2}} - v_l \pd_r v_l - \dfrac1{\rho_l} \pd_r p_l, \label{eq-1.1-simplified-1-a}\\
\dfrac1{r^2} \pd_r(r^2 v_l) = 0, \label{eq-1.1-simplified-1-b}
\end{empheq}
\end{subequations}
and
 \begin{subequations}\label{eq-1.2-simplified-1}
\begin{empheq}[right=\empheqrbrace\text{for $0\le r\le R(t)$, $t>0$,}]{align}
\pd_t\rho_g + \dfrac1{r^2} \pd_r (\rho_gr^2 v_g) = 0, \label{eq-1.2-simplified-1-a}\\
\dfrac{\pd_t p_g}{p_g} = \dfrac{\ka}{c_v}
\dfrac1{r^2} \pd_r \bke{r^2 \pd_r\bke{\dfrac1{\rho_g}}} - \ga \dfrac1{r^2} \pd_r(r^2 v_g), \label{eq-1.2-simplified-1-b}
\end{empheq}
\end{subequations}
and the boundary condition \eqref{eq1.3simplified} becomes
 \begin{subequations}\label{eq-1.3-simplified1}
\begin{empheq}[right=\empheqrbrace\text{for $t>0$.}]{align}
v_l(R(t),t) = v_g(R(t),t) = \dot R(t), \label{eq-1.3-simplified1-a}\\
p_g(t) - p_l(R(t),t) + 2\mu_l\pd_rv_l(R(t),t) = \dfrac{2\si}{R(t)}, \label{eq-1.3-simplified1-b}
\\
T(R(t),t) = T_\infty, \label{eq-1.3-simplified1-c}
\end{empheq}
\end{subequations}

The liquid velocity and pressure $(v_l,p_l)$ can be directly solved in terms of $R(t)$, $\dot R(t)$, the liquid pressure $p_l(R(t),t)$ on the bubble wall, and the far-field liquid pressure $p_\infty(t) := p_l(r=\infty,t)$. In fact, the incompressibility condition \eqref{eq-1.1-simplified-1-b} and the kinematic boundary condition \eqref{eq-1.3-simplified1-a} imply
\EQ{\label{eq-vlr}
v_l(r,t) = \frac{(R(t))^2\dot R(t)}{r^2},\quad r\ge R(t),\ t>0.
}
Plugging Equation \eqref{eq-vlr} into Equation \eqref{eq-1.1-simplified-1-a}, we have
\EQ{\label{eq-bv-2.9}
\frac{2R\dot R^2 + R^2 \ddot R}{r^2} 
= \nu_l\bke{\frac{2R^2\dot R}{r^4} - 2\frac{R^2\dot R}{r^4} } + 2 \frac{R^4\dot R^2}{r^5} - \frac1{\rho_l} \pd_rp_l,\quad r\ge R(t),\ t>0.
}
Note that the diffusion term in \eqref{eq-bv-2.9} vanishes. So the reduction using spherical symmetry assumption also works for Euler equation, {\it i.e.}, we can take $\nu_l=0$ in \eqref{eq1.1simplified-a}.
Integrating Equation \eqref{eq-bv-2.9} over $r>R(t)$, we deduce
\EQ{\label{eq-pl}
p_l(r,t) = p_\infty(t) + \rho_l \bke{ \frac{2R(t) (\dot R(t))^2 + (R(t))^2\ddot R(t)}r - \frac{(R(t))^4(\dot R(t))^2}{2r^4} },\quad r\ge R(t),\ t>0.
}
In particular, on the boundary the liquid pressure is
\[
p_l(R(t),t) = p_\infty(t) + \rho_l\bke{\frac32\dot R^2 + R\ddot R},\quad t>0.
\]
Moreover, \eqref{eq-vlr} implies $\pd_rv_l(r,t) = -2(R(t))^2\dot R(t)/r^3$ so that
\[
\pd_rv_l(R(t),t) = -2\, \frac{\dot R(t)}{R(t)}.
\]
This implies
\EQ{\label{eq-bv-2.19prime}
R\ddot R + \frac32 \dot R^2 = \frac1{\rho_l} \bke{p_g(t) - p_\infty(t) - \frac{2\si}R 
-4\mu_l \frac{\dot R}R
},\quad t>0,
}
where the Young--Laplace boundary condition \eqref{eq-1.3-simplified1-b} has been used.

For the gas dynamics in the bubble, by integrating \eqref{eq-1.2-simplified-1-b} in $r$, the radial component of the gas velocity $v_g$ can be expressed in terms of $\rho_g(r,t)$, $\pd_r\rho_g(r,t)$, $p_g(t)$, and $\pd_tp_g(t)$. To be more precise, 
\EQ{\label{eq-bv-3.9}
v_g(r,t) = \frac{\ka}{\ga c_v} \pd_r\bke{\frac1{\rho_g(r,t)}} - \frac{\pd_tp_g(t)}{p_g(t)}\frac{r}{3\ga},\quad 0\le r\le R(t),\ t>0.
}
Using \eqref{eq-bv-3.9} we can eliminate ${\bf v}_g$ in \eqref{eq-bv-3.10-nonspherical} and obtain
\EQ{\label{eq-bv-3.10}
\pd_t\rho_g = \frac\ka{\ga c_v} \De_r\log\rho_g + \frac{\pd_t p_g}{3\ga p_g} r \pd_r\rho_g + \frac{\pd_t p_g}{\ga p_g} \rho_g,\quad 0\le r\le R(t),\ t>0,
}
where $\De_r f = \frac1{r^2}\pd_r(r^2\pd_rf)$ for spherically symmetric functions $f$.
From the boundary condition \eqref{eq-1.3-simplified1-c} for the gas temperature and the equation of state \eqref{eq1.2simplified-d},
\EQ{\label{eq-bv-3.11}
p_g(t) = \Rg \rho_g(R(t),t) T_\infty.
}
Taking time derivative of \eqref{eq-bv-3.11} we obtain
\EQ{\label{eq-bv-3.12}
\frac{\pd_t p_g}{p_g} = \frac{\pd_t\rho_g(R(t),t)}{\rho_g(R(t),t)} + \frac{\pd_r\rho_g(R(t),t)}{\rho_g(R(t),t)} \dot R(t).
}
Evaluating \eqref{eq-bv-3.9} at $r=R(t)$ and using the kinematic boundary condition \eqref{eq-1.3-simplified1-a} it follows that
\EQ{\label{eq-bv-3.13}
\dot R(t) = -\frac\ka{\ga c_v} \frac{\pd_r\rho_g(R(t),t)}{\bke{\rho_g(R(t),t)}^2} - \frac{R(t)}{3\ga} \frac{\pd_t p_g}{p_g}.
}
For the boundary data for the gas density, we use  \eqref{eq-bv-3.11} and \eqref{eq-bv-2.19prime} to deduce
\EQ{\label{eq-bv-3.16}
\rho_g(R(t),t) = \frac1{\Rg T_\infty} \bkt{p_\infty + \frac{2\si}R + \rho_l\bke{R\ddot R + \frac32 (\dot R)^2} }.
}

Collecting the results \eqref{eq-bv-3.10}, \eqref{eq-bv-3.11}, \eqref{eq-bv-3.13},  \eqref{eq-bv-3.16}, we conclude that, under the spherical symmetry assumption, the system \eqref{eq1.1simplified}--\eqref{eq1.3simplified} is reduced to a system of $(\rho(r,t),R(t))$:
\EQ{\label{eq-bv-3.10prime_1}
\pd_t\rho = \frac\ka{\ga c_v} \De_r\log\rho + \frac{\pd_t p}{3\ga p} r \pd_r\rho + \frac{\pd_t p}{\ga p} \rho,\quad 0\le r\le R(t),\ t>0,
}

\EQ{\label{eq-bv-3.14prime_1}
p(t) = \Rg T_\infty \rho(R(t),t),\quad t>0,
}

\EQ{\label{eq-bv-3.15prime_1}
\dot R(t) = -\frac\ka{\ga c_v} \frac{\pd_r\rho(R(t),t)}{\bke{\rho(R(t),t)}^2} - \frac{R(t)}{3\ga} \frac{\pd_t p}p,\quad t>0,
}

\EQ{\label{eq-bv-3.16prime_1}
\rho(R(t),t) = \frac1{\Rg T_\infty} \bkt{p_\infty + \frac{2\si}R 
+ 4\mu_l \frac{\dot R}R
+ 
\rho_l\bke{R\ddot R + \frac32 (\dot R)^2}
}, \quad t>0,
}
where $\rho \equiv \rho_g$, $p\equiv p_g$, $\ka=\ka_g$. This is the  reduced system \eqref{eq-bv-3.10prime}--\eqref{eq-bv-3.16prime}.

\section{A perspective on coercive energy estimate of Biro--Vel\'azquez, and an extension}\label{sec-bv-revisit}

In this appendix, we prove Theorem \ref{thm-XYZ}, which extends the coercivity estimate of  Biro--Vel\'azquez
 to the case where $p_\infty-p_{\infty,*}$ is small in norm. 

\medskip\noindent{\bf Proof of Theorem \ref{thm-XYZ}.}
Let us recall the total energy
\[
\mathcal{E}_{\rm total} = FE + KE_l + U_{g-l} + PV_{p_\infty},
\]
where 
 \begin{subequations}\label{eq-energy-appdx}
\begin{empheq}{align}
FE &= \frac{4\pi c_v}{3\Rg } pR^3 - c_v T_\infty M_0 \log p + c_v\ga T_\infty \int_{B_R} \rho\log\rho,\qquad M_0 = \textrm{Mass}[\rho,R],\label{eq-FE-appdx}\\
KE_l &= 2\pi \rho_l R^3 \dot R^2, \label{eq-KE-appdx}\\
U_{g-l} &= 4\pi \si R^2, \label{eq-Ugl-appdx}\\
PV_{p_\infty} &= \frac{4\pi}3 R^3 p_\infty. \label{eq-PV-appdx}
\end{empheq}
\end{subequations}

The energy is a functional of state variables, which are defined on a deforming regime, $B_{R}$. We fix the region 
to be $B_1$ by setting $x=Ry$, where $y\in B_1$. Defining $\overline\rho(y)=\rho(Rr)$ and using the constitutive relation  $p = \Rg T_\infty \rho(R) = \Rg T_\infty \overline\rho(1)$ we have that 
\[
FE = \frac{4\pi c_v T_\infty}3 \overline\rho(1) R^3 - c_v T_\infty M_0 \log(\Rg T_\infty) - c_v T_\infty M_0 \log\overline\rho(1) + c_v\ga T_\infty R^3 \int_{B_1} \overline\rho \log\overline\rho.
\]
Thus, $\mathcal{E}_{\rm total} $ is a functional of $(\overline\rho, R,\dot R)$:
\EQN{
\mathcal{E}_{\rm total} 
&= \mathcal{E}_{\rm total} [\overline\rho, R,\dot R]\\
&= \frac{4\pi c_v T_\infty}3  \overline\rho(1)R^3 - c_vT_\infty M_0\log(\Rg T_\infty) - c_vT_\infty M_0 \log \overline\rho(1) + c_v\ga T_\infty R^3 \int_{B_1} \overline\rho \log\overline\rho\\
&\quad + 2\pi\rho_lR^3\dot R^2 + 4\pi\si R^2 + \frac{4\pi}3R^3p_\infty.
}

We set $\overline\rho = \rho_* + \varrho$, $R = R_* + \mathcal R$ and expand the total energy $\mathcal{E}_{\rm total}[\overline\rho, R,\dot R]$ at $(\rho_*,R_*,\dot R_*=0)$ along the mass preserving hypersurface $M_0 = \textrm{Mass}[\rho,R]$:
\EQ{\label{eq-energy-expand}
\mathcal{E}_{\rm total}& [\overline\rho,R,\dot R]
= \mathcal{E}_{*} + d\mathcal{E}_{*}[ \varrho, \mathcal R, \dot{\mathcal R}]+ \frac12  d^2\mathcal{E}_{*}[ \varrho, \mathcal R, \dot{\mathcal R}] + O(|( \varrho, \mathcal R, \dot{\mathcal R})|^3),\quad {\rm where}
}
\begin{align*}
\mathcal{E}_{*}&=\mathcal{E}_{\rm total} [\rho_*,R_*,\dot R_*=0]\\
 d\mathcal{E}_{*}[ \varrho, \mathcal R, \dot{\mathcal R}]&= \frac{d}{d\ve}\Big|_{\ve=0} \mathcal{E}_{\rm total} (\rho_*,R_*+\ve\mathcal R,\ve\dot{\mathcal R})\\
 d^2\mathcal{E}_{*}[ \varrho, \mathcal R, \dot{\mathcal R}]&= \frac{d^2}{d\ve^2}\Big|_{\ve=0} \mathcal{E}_{\rm total} (\rho_*,R_*+\ve\mathcal R,\ve\dot{\mathcal R})\
 \end{align*}

To expand along the mass preserving hypersurface, we first use $M_0 = \int_{B_R} \rho\, dx$ to rewrite
\EQN{
R^3 \int_{B_1} \overline\rho \log\overline\rho 
&= \int_{B_R} \rho\log\rho
= \int_{B_R} \rho\log\rho_* + \int_{B_R} \rho\log\frac{\rho}{\rho_*}\\
&= \log\rho_*\int_{B_R} \rho + \int_{B_R} \rho + \int_{B_R} \rho \bke{\log\frac{\rho}{\rho_*} - 1}\\
&= M_0\log\rho_* + M_0 + R^3 \int_{B_1} \overline\rho \bke{\log\frac{\overline\rho}{\rho_*} - 1},
}
giving the following expression for the total energy:
\EQN{
\mathcal{E}_{\rm total} [\overline\rho, R,\dot R]
&= \frac{4\pi c_v T_\infty}3  \overline\rho(1,t)R^3 - c_vT_\infty M_0\log(\Rg T_\infty) - c_vT_\infty M_0 \log \overline\rho(1,t) \\
&\quad+ c_v\ga T_\infty M_0\log\rho_* + c_v\ga T_\infty M_0 + c_v\ga T_\infty R^3 \int_{B_1} \overline\rho \bke{\log\frac{\overline\rho}{\rho_*} - 1}\\
&\quad + 2\pi\rho_lR^3\dot R^2 + 4\pi\si R^2  + \frac{4\pi}3R^3 p_\infty .
}

To expand the logarithmic terms we note that for  $z_*\ne0$ and  $|z-z_*|<\frac12 |z_*|$:
\begin{align}
\Big| \Big( z\left( \log\left(\frac{z}{z_*}\right) - 1 \Big) - \Big(-z_* + \frac1{2z_*} (z-z_*)^2 \right)\Big| &\le \frac{2}{|z_*|}|z-z_*|^3
\label{log1}\\
\Big| \log z - \left( \log z_* + \frac{1}{z_*}(z-z_*)- \frac{1}{2 z_*^2} (z-z_*)^2 \right) \Big| &\le \frac{2}{3|z_*|^3}|z-z_*|^3
\label{log2}\end{align}
Applying \eqref{log1} and \eqref{log2} we have
\EQN{
\mathcal{E}_{\rm total} [\overline\rho, R,\dot R]
&= \frac{4\pi c_v T_\infty}3  \overline\rho(1,t)R^3 - c_vT_\infty M_0\log(\Rg T_\infty)\\
&\quad  - c_vT_\infty M_0\left( \log\rho_* + \frac{1}{\rho_*} \varrho(1)- \frac{1}{2 \rho_*^2} \varrho(1)^2\right)  + 
O\left( \varrho(1)^3\right) \\
&\quad+ c_v\ga T_\infty M_0\log\rho_* + c_v\ga T_\infty M_0\\
&\quad  + c_v\ga T_\infty R^3\left(-\frac{4\pi}{3}\rho_* +\frac{1}{2\rho_*}\int_{B_1} \varrho^2 \right) 
+ O\left( R^3\int_{B_1}| \varrho|^3 \right) \\
&\quad + 2\pi\rho_lR^3\dot R^2 + 4\pi\si R^2 +  \frac{4\pi}3R^3 p_\infty .
}
Rearranging and simplifying gives
\begin{align} 
\mathcal{E}_{\rm total} [\overline\rho, R,\dot R] &=  - c_vT_\infty M_0\log(\Rg T_\infty) + c_v(\ga-1) T_\infty M_0\log\rho_*
 + c_v\ga T_\infty M_0  \label{Estart}\\
&+ \frac{4\pi c_v T_\infty}3  \overline\rho(1)R^3  + c_vT_\infty M_0\left(  -\frac{1}{\rho_*} \varrho(1) + \frac{1}{2 \rho_*^2} \varrho(1)^2\right)  + 
O\left( \varrho(1)^3\right) \nonumber \\
&\quad  + c_v\ga T_\infty R^3\left(-\frac{4\pi}{3}\rho_* +\frac{1}{2\rho_*}\int_{B_1} \varrho^2 \right) 
+ O\left( R^3\int_{B_1} | \varrho|^3 \right)\nonumber \\
&\quad + 2\pi\rho_lR^3\dot R^2 + 4\pi\si R^2 +  \frac{4\pi}3R^3 p_\infty.
\nonumber\end{align}

\medskip\noindent{\bf Verification that $d\mathcal{E}_*[\varrho,\mathcal{R},\dot{\mathcal R}]=0$ when $p_\infty = p_{\infty,*}$.} 
Starting with \eqref{Estart} we calculate:
\begin{align}
d\mathcal{E}_*[ \varrho,\mathcal{R},\dot{\mathcal R}] &= \frac{4\pi c_v T_\infty}{3}\left(3\rho_*R_*^2\mathcal{R} +R_*^3 \varrho(1)\right) -\frac{c_vT_\infty}{\rho_*}\left(\frac{4\pi}{3}\rho_*R_*^3\right) \varrho(1) \label{firstvar}\\
&\quad-4\pi c_v\gamma T_\infty R_*^2\rho_*\mathcal{R}+8\pi\sigma R_*\mathcal{R}+4\pi R_*^2 p_\infty\mathcal{R}\nonumber\\
&= 4\pi R_*^2\left( c_vT_\infty\rho_*(1-\gamma)+ \frac{2\sigma}{R_*}+p_\infty\right) \mathcal R\nonumber\\
&= 4\pi R_*^2\left( -\Rg T_\infty\rho_*+ \frac{2\sigma}{R_*}+p_\infty\right) \mathcal R\qquad 
\left[\gamma-1=\frac{\Rg}{c_v}\ \textrm{by \eqref{eq-def-gamma}}\right]\nonumber\\
&= 4\pi R_*^2 \mathcal P_\infty \mathcal R \qquad \left[ \textrm{by \eqref{eq-steadystate-b}} \right],
\nonumber\end{align}
where $\mathcal P_\infty =  p_\infty - p_{\infty,*}$.
It is readily to see that $d\mathcal{E}_*[ \varrho,\mathcal{R},\dot{\mathcal R}]=0$ when $p_\infty = p_{\infty,*}$. 

\medskip\noindent{\bf Computation of $\frac{1}{2}d^2\mathcal{E}[ \varrho, \mathcal R, \dot{\mathcal R}]$.}

From \eqref{Estart} we compute the quadratic terms:

\begin{align}
\frac{1}{2}d^2\mathcal{E}[ \varrho, \mathcal R, \dot{\mathcal R}] &= 4\pi c_v T_\infty \left( \rho_*R_* \mathcal{R}^2 + R_*^2 \varrho(1) \mathcal{R}\right)  + \frac{c_v T_\infty M_0}{2\rho_*^2} \varrho(1)^2 \label{2var}\\
&\quad + \frac{c_v\gamma T_\infty R_*^3}{2\rho_*} \int_{B_1} \varrho^2 - 4\pi c_v\gamma\rho_* T_\infty R_* \mathcal{R}^2\nonumber\\
&\quad + 2\pi\rho_l R_*^3 \dot{\mathcal{R}}^2 + 4\pi \sigma \mathcal{R}^2 + 4\pi R_*  p_\infty \mathcal{R}^2.\nonumber
\end{align}

Next, using that the perturbed bubble is assumed to have  mass equal to $M_0={\rm Mass}(\rho_*,R_*)$, we express the cross-term just above in terms of a quadratic expression in $\varrho$ as follows: 
\begin{align*}
M_0=R^3 \int_{B_1}\overline\rho &= (R_*+\mathcal{R})^3 \int_{B_1}\left(\rho_*+ \varrho\right)\\
&= M_0 + 4\pi R_*^2 \rho_* \mathcal{R} + R_*^3 \int_{B_1} \varrho + O\Big(\mathcal{R}^2 + \left(\int_{B_1} \varrho\right)^2 \Big)
\end{align*}
and therefore
\begin{equation}
\mathcal{R} = -\frac{R_*}{4\pi\rho_*}\int_{B_1} \varrho + O\Big(\mathcal{R}^2 + \left(\int_{B_1} \varrho\right)^2 \Big).
\label{cRvrho}
\end{equation}
Substitution of \eqref{cRvrho} into \eqref{2var} we obtain a leading expression entirely in terms of the perturbed density
 $ \varrho$. We list the various terms that we rewrite exclusively in terms of  $ \varrho$:
 \begin{align*}
 4\pi c_v T_\infty \rho_*R_* \mathcal{R}^2 &= \frac{c_v T_\infty \rho_* R_*^3}{4\pi\rho_*^2} \left(\int_{B_1} \varrho\right)^2 + O\left( |\mathcal{R}|^3 + \left(\int_{B_1} | \varrho|\right)^3 \right)\\
 4\pi c_v T_\infty  R_*^2 \varrho(1) \mathcal{R} &= -\frac{c_v T_\infty R_*^3}{\rho_*}\ \varrho(1) \int_{B_1} \varrho + O\left( |\mathcal{R}|^3 + | \varrho(1)|^3 + \left(\int_{B_1} | \varrho|\right)^3 \right)
\\
 - 4\pi c_v\gamma\rho_* T_\infty R_* \mathcal{R}^2 &= -\frac{c_v\gamma T_\infty R_*^3}{4\pi\rho_*} \left(\int_{B_1} \varrho\right)^2 + O\left( |\mathcal{R}|^3 + \left(\int_{B_1} | \varrho|\right)^3 \right)\\
 4\pi\left(\si +  R_* p_\infty\right)\mathcal{R}^2 &=  \frac{1}{4\pi \rho_*^2} \left(\frac{\si}{R_*} +  p_\infty\right)R_*^3\left(\int_{B_1} \varrho\right)^2 + O\left( |\mathcal{R}|^3 + \left(\int_{B_1} | \varrho|\right)^3 \right).
\end{align*} 
Inserting these expressions into \eqref{2var}, we obtain
\begin{align}
\frac{1}{2}d^2\mathcal{E}[ \varrho, \mathcal R, \dot{\mathcal R}] &= 
\frac{c_v T_\infty M_0}{2\rho_*^2} \varrho(1)^2   + \frac{c_v\gamma T_\infty R_*^3}{2\rho_*} \int_{B_1} \varrho^2 + 2\pi \rho_l R_*^3 \dot{\mathcal{R}}^2
\label{2var-1}\\
 &\quad + \frac{R_*^3}{4\pi \rho_*^2}\left( c_v(1-\gamma)T_\infty\rho_* + \frac{\sigma}{R_*}+p_\infty \right) \left(\int_{B_1} \varrho\right)^2 - \frac{c_v T_\infty R_*^3}{\rho_*} \varrho(1) \int_{B_1} \varrho
 \nonumber \\
 &\quad + O\left( |\mathcal{R}|^3 +  | \varrho(1)|^3 +\left(\int_{B_1} | \varrho|\right)^3 \right)
 \nonumber\end{align}
The coefficient of the fourth term on the right of \eqref{2var-1} can be simplified using the relation $1-\gamma= -\Rg/c_v$
 and the  relation $\Rg T_\infty\rho_* = p_{\infty,*} + 2\sigma/R_*$ between the equilibrium density and bubble radius:
 \[c_v(1-\gamma)T_\infty\rho_* + \frac{\sigma}{R_*}+p_\infty = -\Rg T_\infty\rho_* + \frac{\sigma}{R_*}+p_\infty=-\frac{\sigma}{R_*} + \mathcal P_\infty 
 ,\]
where $\mathcal P_\infty =  p_\infty - p_{\infty,*}$.
 Thus, 
 \begin{align}
\frac{1}{2}d^2\mathcal{E}[ \varrho, \mathcal R, \dot{\mathcal R}] &= 
\frac{c_v T_\infty M_0}{2\rho_*^2} \varrho(1)^2   + \frac{c_v\gamma T_\infty R_*^3}{2\rho_*} \int_{B_1} \varrho^2 + 2\pi \rho_l R_*^3 \dot{\mathcal{R}}^2
\label{2var-2}\\
 &\quad - \frac{\sigma R_*^2}{4\pi \rho_*^2}\left(\int_{B_1} \varrho\right)^2
  + \frac{R_*^3}{4\pi \rho_*^2} \mathcal P_\infty \left(\int_{B_1} \varrho\right)^2 
  - \frac{c_v T_\infty R_*^3}{\rho_*} \varrho(1) \int_{B_1} \varrho \nonumber\\
  &\quad  + O\left( |\mathcal{R}|^3 +  | \varrho(1)|^3 +\left(\int_{B_1} | \varrho|\right)^3 \right)
 \nonumber\end{align}
 Using that $M_0=\rho_* R_*^3 |B_1|$, we may rewrite \eqref{2var-2} as 
 \begin{align}
\frac{1}{2}d^2\mathcal{E}[ \varrho, \mathcal R, \dot{\mathcal R}] &= 
\frac{c_v T_\infty R_*^3 |B_1|}{2} \left( 
\frac{ \varrho(1)^2}{\rho_*} 
- 2\frac{ \varrho(1)}{\rho_*} \frac{1}{|B_1|}\int_{B_1} \varrho  \right)  + \frac{c_v\gamma T_\infty R_*^3 }{2\rho_*} \int_{B_1} \varrho^2 + 2\pi \rho_l R_*^3 \dot{\mathcal{R}}^2
\label{2var-3}\\
 &\quad - \frac{\sigma R_*^2}{4\pi \rho_*^2}\left(\int_{B_1} \varrho\right)^2   
 + \frac{R_*^3}{4\pi \rho_*^2} \mathcal P_\infty \left(\int_{B_1} \varrho\right)^2 
 + O\left( |\mathcal{R}|^3 +  | \varrho(1)|^3 +\left(\int_{B_1} | \varrho|\right)^3 \right) \nonumber
 \end{align}
 or 
 \begin{align}
\frac{1}{2}d^2\mathcal{E}[ \varrho, \mathcal R, \dot{\mathcal R}]  &=  
\frac{c_v T_\infty R_*^3 |B_1|}{2\rho_*} \left( \varrho(1)- \frac{1}{|B_1|}\int_{B_1} \varrho \right)^2
  + \frac{c_v\gamma T_\infty R_*^3}{2\rho_*} \int_{B_1} \varrho^2 + 2\pi \rho_l R_*^3 \dot{\mathcal{R}}^2 \label{2var-4}\\
  &\quad - \left[ \frac{\sigma R_*^2}{4\pi \rho_*^2} + \frac{c_vT_\infty R_*^3}{2 \rho_*|B_1|} \right]\left(\int_{B_1} \varrho\right)^2
  + \frac{R_*^3}{4\pi \rho_*^2} \mathcal P_\infty \left(\int_{B_1} \varrho\right)^2  \nonumber\\
 &\quad + O\left( |\mathcal{R}|^3 +  | \varrho(1)|^3 +\left(\int_{B_1} | \varrho|\right)^3 \right)
 \nonumber\end{align}
 By the Cauchy-Schwarz inequality $\left(\int_{B_1} \varrho\right)^2\le |B_1| \int_{B_1} \varrho^2 $ and therefore
  \begin{align}
\frac{1}{2}d^2\mathcal{E}[ \varrho, \mathcal R, \dot{\mathcal R}]  &\ge  
\frac{c_v T_\infty R_*^3 |B_1|}{2\rho_*} \left( \varrho(1)- \frac{1}{|B_1|}\int_{B_1} \varrho \right)^2
   + 2\pi \rho_l R_*^3 \dot{\mathcal{R}}^2 \label{2var-5}\\
  &\quad + \bke{\frac{c_v\gamma T_\infty R_*^3}{2\rho_*}- \left[ \frac{\sigma R_*^2|B_1|}{4\pi \rho_*^2} + \frac{c_vT_\infty R_*^3}{2 \rho_*} \right] } \int_{B_1} \varrho^2 
  + \frac{R_*^3}{4\pi \rho_*^2}\mathcal P_\infty \left(\int_{B_1} \varrho\right)^2 \nonumber\\
  &\quad + O\left( |\mathcal{R}|^3 +  | \varrho(1)|^3 +\left(\int_{B_1} | \varrho|\right)^3 \right)
 \nonumber\end{align}

Finally, we find 
for the constant in \eqref{2var-5} that
 \begin{equation} \frac{c_v\gamma T_\infty R_*^3}{ 2\rho_*}- \left[ \frac{\sigma R_*^2|B_1|}{4\pi \rho_*^2} + \frac{c_vT_\infty R_*^3}{2 \rho_*} \right]
  = \frac{R_*^3}{ \rho_*^2} \left( \frac{p_{\infty,*}}2 + \frac{2\sigma}{3R_*}\right).\label{pos-coeff}\end{equation}
  This follows, yet again, from the relations $\gamma - 1= \Rg/c_v$
 and  $\Rg T_\infty\rho_* = p_{\infty,*} + 2\sigma/R_*$.
For the term involving $\mathcal P_\infty$, using the Cauchy-Schwarz inequality $\left(\int_{B_1} \varrho\right)^2\le |B_1| \int_{B_1} \varrho^2 $,  
\[
\mathcal P_\infty \bke{\int_{B_1} \varrho}^2
\ge -|\mathcal P_\infty| |B_1| \int_{B_1} \varrho^2.
\]
 
Summarizing
 \begin{align}
\mathcal{E}_{\rm total} - \mathcal{E}_*  &\ge
 \frac{c_v T_\infty R_*^3 |B_1|}{2\rho_*} \left( \varrho(1)- \frac{1}{|B_1|}\int_{B_1} \varrho \right)^2
   + 2\pi \rho_l R_*^3 \dot{\mathcal{R}}^2 + \frac{R_*^3}{\rho_*^2} \left( \frac{p_{\infty,*}}2 + \frac{2\sigma}{3R_*}\right)\int_{B_1} \varrho^2 \label{2var-6}\\
  &\quad - 4\pi R_*^2 |\mathcal P_\infty| |\mathcal R| 
  - \frac{R_*^3}{4\pi \rho_*^2} |\mathcal P_\infty| |B_1| \int_{B_1} \varrho^2 + O\left( |\mathcal{R}|^3 +  | \varrho(1)|^3 +\left(\int_{B_1} | \varrho|\right)^3 \right),
 \nonumber\end{align}
 where all explicit terms are non-negative except for the terms involving $\mathcal P_\infty$ which can be made small since $|\mathcal P_\infty| = |p_\infty - p_{\infty,*}|\le \de_0$.

We now conclude the proof by bounding the error term in \eqref{2var-6}  from above by a sufficiently small constant 
 times $\int_{B_1} \varrho^2$. 

For the fourth term on the right hand side of \eqref{2var-6}, using \eqref{cRvrho}, in terms of the perturbed density $ \varrho$,
\[
- 4\pi R_*^2 |\mathcal P_\infty| |\mathcal R|
\ge - \frac{R_*^3}{\rho_*} |\mathcal P_\infty| \int_{B_1} | \varrho| - C_0|\mathcal P_\infty| \left(\int_{B_1} \varrho\right)^2
\]
for some constant $C_0>0$.
Since $| \varrho|\le\de_0\le1$ and $|\mathcal P_\infty| = |p_\infty - p_{\infty,*}|\le\de_0$,
by the Cauchy-Schwarz inequality $\left(\int_{B_1} \varrho\right)^2\le |B_1| \int_{B_1} \varrho^2 $ 
\[
- 4\pi R_*^2 |\mathcal P_\infty| |\mathcal R|
\ge - \frac{R_*^3}{\rho_*} |\mathcal P_\infty| \int_{B_1} \varrho^2 - C_0|\mathcal P_\infty| |B_1| \int_{B_1} \varrho^2
\ge - C_1 \de_0 \int_{B_1} \varrho^2
\]
for some constant $C_1>0$.

Now we estimate the cubic term in the third line on the right hand side of \eqref{2var-6}.
Since $M_0 = \textrm{Mass}[\rho,R]$,
\[
\int_{B_R} (\rho - \rho_*)\, dx
= \int_{B_R} \varrho\, dx 
= M_0 - \frac{4\pi R^3}3 \rho_* 
= \frac{4\pi R_*^3}3 \rho_* - \frac{4\pi R^3}3 \rho_*,
\]
or
\EQ{\label{eq-bv-4.37}
\frac{4\pi\rho_*}3 (R^3 - R_*^3) = -\int_{B_R} (\rho-\rho_*)\, dx,
}
which implies
\EQ{\label{eq-bv-4.38}
|\mathcal R| = |R - R_*| \le \frac3{4\pi\rho_*(R^2+RR_*+R_*^2)} |B_R|^{\frac12} \bke{\int_{B_R} |\rho-\rho_*|^2\, dx}^{\frac12} 
\le C_2 \bke{\int_{B_R} |\rho-\rho_*|^2\, dx}^{\frac12},
}
where $C_2>0$ depends only on $\nu, M_0, T_\infty$.
We now control $| \varrho(1,t)|^3$ by the first and the third terms on the right hand side of \eqref{2var-6}. 
Indeed,
\EQN{
| \varrho(1)| = \abs{ \bke{ \varrho(1) - \frac1{|B_1|} \int_{B_1} \varrho}  + \frac1{|B_1|} \int_{B_1} \varrho\, }
&\le \abs{ \varrho(1) - \frac1{|B_1|} \int_{B_1} \varrho\, } + \frac1{|B_1|^{\frac12}} \bke{\int_{B_1} \varrho^2 }^{\frac12}\\
&\le C_3\bket{\abs{ \varrho(1) - \frac1{|B_1|} \int_{B_1} \varrho \,} + \bke{\int_{B_1} \varrho^2 }^{\frac12}}
}
for some $C_3>0$ depending only on $\nu, M_0, T_\infty$.
Since $| \varrho(1)| = |\rho(R) - \rho_*| \le \de_0$,
\EQ{\label{eq-bv-4.39}
| \varrho(1)|^3 = | \varrho(1)| |\varrho(1)|^2
&\le \de_0 C_3 \bket{\abs{ \varrho(1) - \frac1{|B_1|} \int_{B_1} \varrho\, } + \bke{\int_{B_1} \varrho^2 }^{\frac12}}^2\\
&\le 2 \de_0 C_3 \abs{ \varrho(1) - \frac1{|B_1|} \int_{B_1} \varrho\, }^2 + 2\de_0 C_3 \int_{B_1} \varrho^2.
}
Using \eqref{eq-bv-4.38} and \eqref{eq-bv-4.39}, one has
\EQN{
& O\left( |\mathcal{R}|^3 +  | \varrho(1)|^3 +\left(\int_{B_1} | \varrho|\right)^3 \right)\\
&\ge 
 - C \bke{\int_{B_R} (\rho-\rho_*)^2 }^{\frac32} - C\de_0 \abs{ \varrho(1) - \frac1{|B_1|} \int_{B_1} \varrho\, }^2 - C\de_0 \int_{B_1} \varrho^2
 - C \int_{B_1}| \varrho|^3
}
for some $C>0$ depending only on $\nu, M_0, T_\infty$.

Consequently, using $|\mathcal P_\infty| = |p_\infty - p_{\infty,*}|\le \de_0$, \eqref{2var-6} can be further computed as 
\EQN{
\mathcal{E}_{\rm total} - \mathcal{E}_*  
&\ge - C_1 \de_0 \int_{B_1} \varrho^2
+ \bke{\frac{c_v T_\infty R_*^3 |B_1|}{2\rho_*} - C\de_0} \bke{ \varrho(1) - \frac1{|B_1|} \int_{B_1} \varrho }^2 + \frac{R_*^3}{\rho_*}\bke{ \frac{p_{\infty,*}}2 + \frac{2\si}{3R_*} } \int_{B_1} \varrho^2\\
&\quad  - C\de_0 |B_R|^{\frac12} \int_{B_R} (\rho-\rho_*)^2 \,dx - 2C\de_0 \int_{B_1} \varrho^2 - \de_0 \frac{R_*^3}{4\pi\rho_*} |B_1| \int_{B_1} \varrho^2\\
&\ge \Theta \bke{\int_{B_R} (\rho-\rho_*)^2 }
}
for some constant $\Theta>0$, provided $\de_0>0$ is sufficiently small. 
Note that we've used $\int_{B_1} \varrho^2\, dy = R^{-3}\int_{B_R}(\rho - \rho_*)^2\, dx$ in which $R^{-3}\ge \nu^3$ above. 
This completes the proof of Theorem \ref{thm-XYZ}.

\section{An interpolation lemma}

\begin{lemma}\label{lem-interpolate}
Let $\Om$ be a bounded Lipschitz domain in $\R^n$, $k< m$, and $0<\ga\le1$. For $u\in C^\infty(\Om)$,
\[
\norm{\nb^k u}_{L^\infty(\Om)} \le  C_1 \norm{u}_{L^p(\Om)}^\la\norm{u}_{C^{m,\ga}(\Om)}^{1-\la} + C_2 \norm{u}_{L^s(\Om)}
\]
for arbitrary $s\ge1$,
where $-\frac{k}n = \frac{\la}p - (1-\la) \frac{m}n$, and the constants $C_1$, $C_2$ depend on the domain $\Om$ and on $s$ in addition to the other parameters.
\end{lemma}

\begin{proof}
By Gagliardo--Nirenberg interpolation inequality, 
\[
\norm{\nb^k u}_{L^\infty(\Om)} \le C_1 \norm{u}_{L^p(\Om)}^{\la} \norm{\nb^m u}_{L^\infty(\Om)}^{1-\la} + C_2 \norm{u}_{L^s(\Om)}
\]
for arbitrary $s\ge1$, where 
\[
0 = \frac{k}{n} - \frac{m}{n}(1-\la) + \frac{\la}p,
\]
and the constants $C_1$, $C_2$ depend on the domain $\Om$ and on $s$ in addition to the other parameters.
The lemma then follows since $\norm{\nb^m u}_{L^\infty(\Om)} \le \norm{u}_{C^{m,\ga}(\Om)}$.
\end{proof}

\section{Estimate of the exponential decay rate $\be$ in the linearized system}\label{sec-negative-upper-bound}

In this appendix, we prove parts (2) and (3) of Theorem \ref{prop-spectrum}. In particular, we investigate the location of the roots of the meromorphic function $Q(\tau)$ defined in \eqref{eq-Q-simplified} which corresponds to the spectrum of the linear operator $\mathcal L$.

\begin{lemma}\label{lem-negative-upper-bound}
There exists a negative upper bound for the real parts of the roots of the meromorphic function $Q(\tau)$ in \eqref{eq-Q-simplified}.
More precisely, there exists $\be>0$ such that $\xi<-\be$ for all roots $\tau=\xi+i\eta$ of $Q(\tau)$.
The constant $\be$ can be chosen as
\EQ{\label{eq-be-def}
&\be = \min\Bigg\{ \bke{1 - \sqrt{\dfrac{1-\frac1\ga}{\frac{3p_{\infty,*}R_*+6\si}{2p_{\infty,*}R_*+6\si} - \frac1\ga}}}\pi^2\bar\ka,\,  \sqrt{ \frac{\Rg T_\infty \rho_*}{\rho_lR_*^2} },\\
&\qquad\qquad\quad \frac{2\mu_l}{\rho_lR_*^2} + \mathbbm{1}_{\De\le0}\, \frac{\Rg T_\infty \rho_*}{\pi^4\bar\ka \rho_lR_*^2} \bke{1-\frac1\ga} \bkt{\frac{\pi^4}{90} + O\bke{\bke{\frac1{\pi^2\bar\ka}\sqrt{ \frac{\Rg T_\infty \rho_*}{\rho_lR_*^2} }}^{3/2}}} - \mathbbm{1}_{\De>0}\, \frac{\sqrt{\De}}{2\rho_lR_*} \Bigg\},
}
in which
\[
\De:= \bke{ \frac{4\mu_l}{R_*} }^2 - 8\rho_l\Rg T_\infty \rho_*.
\]
\end{lemma}

\begin{proof}[Proof of Lemma \ref{lem-negative-upper-bound}]
Let $\tau = \xi + i\eta$ be a root of $Q(\tau)$, {\it i.e.}, $Q(\tau)=0$. Plugging $\tau = \xi + i\eta$, $\xi\in\R,\eta\in\R$, into \eqref{eq-Q-simplified}, we have
\[
Q(\xi + i\eta) = \frac1{\Rg T_\infty} \bke{\Xi_1 + iH_1} \bke{\Xi_2 + iH_2} + 4\pi\,\frac{\rho_*}{R_*},
\]
where
\EQ{\label{eq-Xi-Eta}
\Xi_1 &= \frac{4\pi}{3\ga} + \frac{8(\ga-1)}{\pi\ga}\sum_{j=1}^\infty \frac{\pi^2\bar\ka\bke{\pi^2\bar\ka j^2 + \xi}}{\bke{\pi^2\bar\ka j^2 + \xi}^2 + \eta^2},\\
H_1 &=  - \frac{8(\ga-1)}{\pi\ga} \sum_{j=1}^\infty \frac{\pi^2\bar\ka\eta}{\bke{\pi^2\bar\ka j^2 + \xi}^2 + \eta^2},\\
\Xi_2 &= \rho_lR_*\bke{\xi^2 - \eta^2} + \frac{4\mu_l}{R_*}\, \xi - \frac{2\si}{R_*^2},\\
H_2 &= \rho_lR_*(2\xi\eta) + \frac{4\mu_l}{R_*}\, \eta. 
}
Setting real and imaginary parts of $Q$ equal to zero, we obtain
\EQ{\label{eq-real-imaginary-Q}
\text{real part: }& \frac1{\Rg T_\infty} \bke{\Xi_1\Xi_2 - H_1H_2} + 4\pi\, \frac{\rho_*}{R_*} = 0,\\
\text{imaginary part: }& \frac1{\Rg T_\infty} \bke{\Xi_1H_2 + H_1\Xi_2} = 0.
}
The real part in \eqref{eq-real-imaginary-Q} reads
\EQ{\label{eq-real-part-1}
0 &= \frac1{\Rg T_\infty} \bigg[ \bke{ \frac{4\pi}{3\ga} + \frac{8(\ga-1)}{\pi\ga}\sum_{j=1}^\infty \frac{\pi^2\bar\ka\bke{\pi^2\bar\ka j^2 + \xi}}{\bke{\pi^2\bar\ka j^2 + \xi}^2 + \eta^2} } \bke{\rho_lR_*\bke{\xi^2-\eta^2} + \frac{4\mu_l}{R_*}\, \xi - \frac{2\si}{R_*^2} }\\
&\qquad\qquad\qquad\qquad\qquad + \frac{8(\ga-1)}{\pi\ga} \sum_{j=1}^\infty \frac{\pi^2\bar\ka\eta^2}{\bke{\pi^2\bar\ka j^2 + \xi}^2 + \eta^2 } \bke{\rho_lR_*(2\xi) + \frac{4\mu_l}{R_*}}\bigg] + 4\pi\, \frac{\rho_*}{R_*}\\
&= \frac1{\Rg T_\infty} \bigg[ \bke{ \frac{4\pi}{3\ga} + \frac{8(\ga-1)}{\pi\ga}\sum_{j=1}^\infty \frac{\pi^4\bar\ka^2j^2}{\bke{\pi^2\bar\ka j^2 + \xi}^2 + \eta^2} } \bke{\rho_lR_*\bke{\xi^2-\eta^2} + \frac{4\mu_l}{R_*}\, \xi - \frac{2\si}{R_*^2} }\\
&\qquad\qquad\quad + \frac{8(\ga-1)}{\pi\ga} \sum_{j=1}^\infty \frac{\pi^2\bar\ka}{\bke{\pi^2\bar\ka j^2 + \xi}^2 + \eta^2 } \bke{\rho_lR_*\xi(\xi^2+\eta^2) + \frac{4\mu_l}{R_*}(\xi^2+\eta^2) - \frac{2\si}{R_*^2}\, \xi }\bigg] + 4\pi\, \frac{\rho_*}{R_*}.
}
When $\eta\neq0$, the imaginary part in \eqref{eq-real-imaginary-Q} reads
\EQ{\label{eq-imaginary-part-1}
0 &= \bke{\frac{4\pi}{3\ga} + \frac{8(\ga-1)}{\pi\ga} \sum_{j=1}^\infty \frac{\pi^2\bar\ka\bke{\pi^2\bar\ka j^2 + \xi}}{\bke{\pi^2\bar\ka j^2 + \xi}^2 + \eta^2}} \bke{\rho_lR_*(2\xi) + \frac{4\mu_l}{R_*} }\\
&\qquad\qquad\qquad\qquad\qquad - \frac{8(\ga-1)}{\pi\ga} \sum_{j=1}^\infty \frac{\pi^2\bar\ka}{\bke{\pi^2\bar\ka j^2 + \xi}^2 + \eta^2 } \bke{\rho_lR_*(\xi^2 - \eta^2) + \frac{4\mu_l}{R_*}\, \xi - \frac{2\si}{R_*^2} }\\
&= \bke{\frac{4\pi}{3\ga} + \frac{8(\ga-1)}{\pi\ga} \sum_{j=1}^\infty \frac{\pi^4\bar\ka^2 j^2}{\bke{\pi^2\bar\ka j^2 + \xi}^2 + \eta^2}} \bke{\rho_lR_*(2\xi) + \frac{4\mu_l}{R_*} }\\
&\qquad\qquad\qquad\qquad\qquad + \frac{8(\ga-1)}{\pi\ga} \sum_{j=1}^\infty \frac{\pi^2\bar\ka}{\bke{\pi^2\bar\ka j^2 + \xi}^2 + \eta^2 } \bke{\rho_lR_*(\xi^2 + \eta^2) + \frac{2\si}{R_*^2} }.
}
For $\eta\neq0$, the equation \eqref{eq-imaginary-part-1} implies 
\EQ{\label{eq-imaginary-part-2}
\rho_lR_*(2\xi) + \frac{4\mu_l}{R_*} < 0,
}
which gives
\EQ{\label{eq-xi-first-upper-bound}
\xi < - \frac{2\mu_l}{\rho_lR_*^2} \le 0,\qquad \eta\neq0.
}
The equation \eqref{eq-imaginary-part-1} also implies 
\EQ{\label{eq-imaginary-part-3}
\frac{8(\ga-1)}{\pi\ga} \sum_{j=1}^\infty \frac{\pi^2\bar\ka}{\bke{\pi^2\bar\ka j^2 + \xi}^2 + \eta^2 }
= - \bke{\frac{4\pi}{3\ga} + \frac{8(\ga-1)}{\pi\ga} \sum_{j=1}^\infty \frac{\pi^4\bar\ka^2 j^2}{\bke{\pi^2\bar\ka j^2 + \xi}^2 + \eta^2}} \dfrac{\rho_lR_*(2\xi) + \frac{4\mu_l}{R_*}}{\rho_lR_*(\xi^2+\eta^2) + \frac{2\si}{R_*^2}}.
}
Plugging \eqref{eq-imaginary-part-3} into the real part \eqref{eq-real-part-1}, we derive
\EQ{\label{eq-real-part-2}
0 &= \frac1{\Rg T_\infty} \bke{ \frac{4\pi}{3\ga} + \frac{8(\ga-1)}{\pi\ga}\sum_{j=1}^\infty \frac{\pi^4\bar\ka^2j^2}{\bke{\pi^2\bar\ka j^2 + \xi}^2 + \eta^2} }
\Bigg[ \rho_lR_*\bke{\xi^2-\eta^2} + \frac{4\mu_l}{R_*}\, \xi - \frac{2\si}{R_*^2} \\
&\quad - \dfrac{\bke{ \rho_lR_*(2\xi) + \frac{4\mu_l}{R_*} } \bke{ \rho_lR_*\xi(\xi^2+\eta^2) + \frac{4\mu_l}{R_*}(\xi^2+\eta^2) - \frac{2\si}{R_*^2}\, \xi } }{\rho_lR_*(\xi^2+\eta^2) + \frac{2\si}{R_*^2}} \Bigg] + 4\pi\, \frac{\rho_*}{R_*}\\
&=: \frac1{\Rg T_\infty} \bke{ \frac{4\pi}{3\ga} + \frac{8(\ga-1)}{\pi\ga}\sum_{j=1}^\infty \frac{\pi^4\bar\ka^2j^2}{\bke{\pi^2\bar\ka j^2 + \xi}^2 + \eta^2} } \La + 4\pi\, \frac{\rho_*}{R_*},
}
where $\La$ is the square bracket on the right hand side of the first equation.
A straightforward calculation shows that
\EQN{
&\bke{ \rho_lR_*(\xi^2+\eta^2) + \frac{2\si}{R_*^2} }\La \\
&= -\rho_l^2R_*^2 (\xi^2 + \eta^2)^2 - \frac{4\mu_l}{R_*} \bke{\rho_lR_*(2\xi) + \frac{4\mu_l}{R_*} } (\xi^2 + \eta^2) + \frac{2\si}{R_*^2} \bke{2\rho_lR_*\xi^2 + 2\,\frac{4\mu_l}{R_*}\, \xi - 2\rho_lR_*\eta^2} - \bke{\frac{2\si}{R_*^2}}^2\\
&> -\rho_l^2R_*^2 (\xi^2 + \eta^2)^2 - \frac{2\si}{R_*^2}\, 2\rho_lR_*(\xi^2+\eta^2) - \bke{\frac{2\si}{R_*^2}}^2 = -\bke{\rho_lR_*(\xi^2+\eta^2) + \frac{2\si}{R_*^2}}^2,
}
where we've used \eqref{eq-imaginary-part-2} and so $2\rho_lR_*\xi^2 + 2\,\frac{4\mu_l}{R_*}\, \xi = \xi \bke{\rho_lR_*(2\xi) + 2\,\frac{4\mu_l}{R_*} } > -2\rho_lR_*\xi^2$ in the last inequality.
This implies
\EQ{\label{eq-Lambda-ineq}
\La > -\bke{\rho_lR_*(\xi^2+\eta^2) + \frac{2\si}{R_*^2}}.
}
Using \eqref{eq-Lambda-ineq} in \eqref{eq-real-part-2}, we get
\EQ{\label{eq-real-part-3}
0 &> - \frac1{\Rg T_\infty} \bke{ \frac{4\pi}{3\ga} + \frac{8(\ga-1)}{\pi\ga}\sum_{j=1}^\infty \frac{\pi^4\bar\ka^2j^2}{\bke{\pi^2\bar\ka j^2 + \xi}^2 + \eta^2} } \bke{ \rho_lR_*(\xi^2+\eta^2) + \frac{2\si}{R_*^2} } + 4\pi\, \frac{\rho_*}{R_*}.
}
Suppose 
\EQ{\label{eq-xi-assumption-1}
\xi\ge-\th\pi^2\bar\ka,
} 
where $\th\in(0,1)$ to be chosen. Then $\xi\ge-\th\pi^2\bar\ka j^2$ for all $j=1,2,\ldots$.
We further assume that 
\EQ{\label{eq-xi-assumption-2}
\xi \ge - \sqrt{\frac{2p_{\infty,*}}{\rho_lR_*^2} + \frac{4\si}{\rho_lR_*^3} - \eta^2},
}
provided $\eta^2 \le \frac{2p_{\infty,*}}{\rho_lR_*^2} + \frac{4\si}{\rho_lR_*^3}$, so that $\rho_lR_*(\xi^2+\eta^2) + \frac{2\si}{R_*^2} \le \frac{2p_{\infty,*}}{R_*} + \frac{6\si}{R_*^2}$.
Then \eqref{eq-real-part-3} gives
\EQN{
0 &> - \frac1{\Rg T_\infty} \bke{ \frac{4\pi}{3\ga} + \frac{8(\ga-1)}{\pi\ga}\, \frac1{(1-\th)^2} \sum_{j=1}^\infty \frac1{j^2} } \bke{ \frac{2p_{\infty,*}}{R_*} + \frac{6\si}{R_*^2} } + 4\pi\, \frac{\rho_*}{R_*}.
}
Using $\sum_{j=1}^\infty j^{-2} = \pi^2/6$, one has
\[
3\Rg T_\infty\, \frac{\rho_*}{R_*} < \bkt{\frac1\ga + \bke{1-\frac1\ga} \frac1{(1-\th)^2}} \bke{\frac{2p_{\infty,*}}{R_*} + \frac{6\si}{R_*^2} },
\]
or, equivalently, using $\Rg T_\infty \rho_* = p_{\infty,*} + 2\si/R_*$,
\[
\th > 1 - \sqrt{\dfrac{1-\frac1\ga}{\frac{3p_{\infty,*}R_*+6\si}{2p_{\infty,*}R_*+6\si} - \frac1\ga}}.
\]
We simply choose 
\[
\th = 1 - \sqrt{\dfrac{1-\frac1\ga}{\frac{3p_{\infty,*}R_*+6\si}{2p_{\infty,*}R_*+6\si} - \frac1\ga}} \in (0,1)
\]
to reach a contradiction to \eqref{eq-xi-assumption-1} and \eqref{eq-xi-assumption-2}.
Therefore, we have for $\eta\in\R$ with $\eta\neq0$ and $\eta^2\le \frac{2p_{\infty,*}}{\rho_lR_*^2} + \frac{4\si}{\rho_lR_*^3}$ that
\[
\xi < - \min\bket{ \bke{1 - \sqrt{\dfrac{1-\frac1\ga}{\frac{3p_{\infty,*}R_*+6\si}{2p_{\infty,*}R_*+6\si} - \frac1\ga}}}\pi^2\bar\ka,\, \sqrt{\frac{2p_{\infty,*}}{\rho_lR_*^2} + \frac{4\si}{\rho_lR_*^3} - \eta^2}}.
\]
Combining \eqref{eq-xi-first-upper-bound}, for $\eta\in\R$ with $\eta\neq0$ and $\eta^2\le \frac{p_{\infty,*}}{\rho_lR_*^2} + \frac{2\si}{\rho_lR_*^3} $ we have
\EQ{\label{eq-xi-upper-bound-eta-small}
\xi < - \max\bket{\frac{2\mu_l}{\rho_lR_*^2},\, \min\bket{ \bke{1 - \sqrt{\dfrac{1-\frac1\ga}{\frac{3p_{\infty,*}R_*+6\si}{2p_{\infty,*}R_*+6\si} - \frac1\ga}}}\pi^2\bar\ka,\, \sqrt{  \frac{p_{\infty,*}}{\rho_lR_*^2} + \frac{2\si}{\rho_lR_*^3} }} }.
}

Now, we consider the case $\eta^2 > \frac{p_{\infty,*}}{\rho_lR_*^2} + \frac{2\si}{\rho_lR_*^3} $. Since $\eta\neq0$, the imaginary part in \eqref{eq-real-imaginary-Q} gives the identity $\Xi_1 = - \frac{H_1}{H_2}\, \Xi_2$. Using this identity in the real part in \eqref{eq-real-imaginary-Q}, we derive
\[
4\pi\, \frac{\rho_*}{R_*}\, \Rg T_\infty H_2 = H_1\bke{\Xi_2^2 + H_2^2},
\]
which implies
\EQ{\label{eq-Xi-Eta-square}
4\pi\, \frac{\rho_*}{R_*}\, \Rg T_\infty \bke{\rho_lR_*(2\xi) + \frac{4\mu_l}{R_*} } = -\frac{8(\ga-1)}{\pi\ga} \sum_{j=1}^\infty \frac{\pi^2\bar\ka}{\bke{\pi^2\bar\ka j^2 + \xi}^2 + \eta^2}\, \bke{\Xi_2^2 + H_2^2}.
}
To find a positive lower bound for $\Xi_2^2 + H_2^2$,
note that
\[
\Xi_2^2 + H_2^2 = \abs{\rho_lR_*\tau^2 + \frac{4\mu_l}{R_*}\, \tau - \frac{2\si}{R_*^2}}^2 = \rho_l^2R_*^2 \abs{\tau - \tau_+}^2 \abs{\tau - \tau_-}^2,
\]
where 
\[
\tau_{\pm} = \dfrac{-\dfrac{4\mu_l}{R_*} \pm \sqrt{\bke{\dfrac{4\mu_l}{R_*}}^2 + 4\rho_lR_*\, \dfrac{2\si}{R_*^2}}}{2\rho_lR_*}
\]
are on the real axis.
By the triangular inequality $|\tau - \tau_{\pm}| > |\eta|$, and so
\[
\Xi_2^2 + H_2^2 > \rho_l^2 R_*^2 \eta^4.
\]
Therefore, \eqref{eq-Xi-Eta-square} yields
\[
4\pi\, \frac{\rho_*}{R_*}\, \Rg T_\infty \bke{\rho_lR_*(2\xi) + \frac{4\mu_l}{R_*} } < -\frac{8(\ga-1)}{\pi\ga} \sum_{j=1}^\infty \frac{\pi^2\bar\ka}{\pi^4\bar\ka^2 j^4 + \eta^2}\, \rho_l^2R_*^2\eta^4.
\]
Since $\eta^2 > \frac{p_{\infty,*}}{\rho_lR_*^2} + \frac{2\si}{\rho_lR_*^3}$ and $\frac{a^2}{\pi^4\bar\ka^2 j^4 + a}$ is increasing in $a$ for $a = \eta^2 \ge 0$,
we further derive 
\EQ{\label{eq-Xi-Eta-square-1}
4\pi\, \frac{\rho_*}{R_*}\, \Rg T_\infty \bke{\rho_lR_*(2\xi) + \frac{4\mu_l}{R_*} } 
&< -\frac{8(\ga-1)}{\pi\ga} \sum_{j=1}^\infty \frac{\pi^2\bar\ka}{\pi^4\bar\ka^2 j^4 + \bke{\frac{p_{\infty,*}}{\rho_lR_*^2} + \frac{2\si}{\rho_lR_*^3}}}\, \rho_l^2R_*^2 \bke{\frac{p_{\infty,*}}{\rho_lR_*^2} + \frac{2\si}{\rho_lR_*^3}}^2\\
&= -\frac{8(\ga-1)}{\pi\ga}\, \frac1{\pi^2\bar\ka} \sum_{j=1}^\infty \frac1{j^4 + B^2}\, \rho_l^2R_*^2 \bke{\frac{p_{\infty,*}}{\rho_lR_*^2} + \frac{2\si}{\rho_lR_*^3}}^2,
}
where
\[
B:= \frac1{\pi^2\bar\ka}\sqrt{ \frac{p_{\infty,*}}{\rho_lR_*^2} + \frac{2\si}{\rho_lR_*^3}}.
\]
Using 
\EQN{
\sum_{j=1}^\infty \frac1{j^4 + B^2} 
&= \frac{e^{\frac{i\pi}4}\pi\cot\bke{e^{\frac{i\pi}4}\pi\sqrt{B}} + e^{\frac{3i\pi}4}\pi\cot\bke{e^{\frac{3i\pi}4}\pi\sqrt{B}} }{4B^{3/2}} - \frac1{2B^2}\\
&= \dfrac{\frac{2\pi}{\sqrt2} \bke{\frac{\cot\bke{\pi\sqrt{B/2}}\csch^2\bke{\pi\sqrt{B/2}} + \csc^2\bke{\pi\sqrt{B/2}} \coth\bke{\pi\sqrt{B/2}}}{\cot^2\bke{\pi\sqrt{B/2}} + \coth^2\bke{\pi\sqrt{B/2}}} }}{4B^{3/2}} - \frac1{2B^2}\\
& = \dfrac{\frac{2\pi}{\sqrt2} \bke{ \frac1{\pi\sqrt{B/2}} + \frac4{45} \bke{\pi\sqrt{B/2}}^3 + O\bke{\bke{\pi\sqrt{B/2}}^6}}}{4B^{\frac32}} - \frac1{2B^2}\\
& = \frac{\pi^4}{90} + O\bke{B^{3/2}}\ \text{ for }B\ll1,
}
we have from \eqref{eq-Xi-Eta-square-1} that
\EQN{
4\pi\, &\frac{\rho_*}{R_*}\, \Rg T_\infty \bke{\rho_lR_*(2\xi) + \frac{4\mu_l}{R_*} } \\
&< -\frac{8(\ga-1)}{\pi\ga}\, \frac1{\pi^2\bar\ka} \bke{\frac{\pi^4}{90} + O\bke{\bke{\frac1{\pi^2\bar\ka}\sqrt{ \frac{p_{\infty,*}}{\rho_lR_*^2} + \frac{2\si}{\rho_lR_*^3} }}^{3/2}}}  \rho_l^2R_*^2 \bke{\frac{p_{\infty,*}}{\rho_lR_*^2} + \frac{2\si}{\rho_lR_*^3}}^2.
}
Consequently, we have for $\eta^2 > \frac{p_{\infty,*}}{\rho_lR_*^2} + \frac{2\si}{\rho_lR_*^3} $ that
\EQ{\label{eq-xi-upper-bound-eta-large}
\xi< -\frac{2\mu_l}{\rho_lR_*^2} - \frac{\Rg T_\infty \rho_*}{\pi^4\bar\ka \rho_lR_*^2} \bke{1-\frac1\ga} \bke{\frac{\pi^4}{90} + O\bke{\bke{\frac1{\pi^2\bar\ka}\sqrt{\frac{p_{\infty,*}}{\rho_lR_*^2} + \frac{2\si}{\rho_lR_*^3}} }^{3/2}}},
}
where $\Rg T_\infty \rho_* = p_{\infty,*} + 2\si/R_*$ is used.

It remains to consider the case $\eta=0$. 
We first show that $\xi<0$.
Suppose for the sake of contradiction that $\xi\ge0$ then
\EQN{
Q(\xi) 
&= \frac1{\Rg T_\infty} \bke{\frac{4\pi}{3\ga} + \frac{8(\ga-1)}{\pi\ga} \sum_{j=1}^\infty \frac{\pi^2\bar\ka}{\pi^2\bar\ka j^2 + \xi} } \bke{\rho_lR_*\xi^2 + \frac{4\mu_l}{R_*}\, \xi }\\
&\quad\quad\quad - \frac{2\si}{\Rg T_\infty R_*^2} \bke{\frac{4\pi}{3\ga} + \frac{8(\ga-1)}{\pi\ga} \sum_{j=1}^\infty \frac{\pi^2\bar\ka}{\pi^2\bar\ka j^2 + \xi} } + 4\pi\,\frac{\rho_*}{R_*},
}
where the second line is greater than
\[
 - \frac{2\si}{\Rg T_\infty R_*^2} \bke{\frac{4\pi}{3\ga} + \frac{8(\ga-1)}{\pi\ga} \sum_{j=1}^\infty \frac{\pi^2\bar\ka}{\pi^2\bar\ka j^2} } + 4\pi\,\frac{\rho_*}{R_*}
=  \frac{8\pi\rho_*}{3R_*} + \frac{4\pi p_{\infty,*}}{3\Rg T_\infty R_*} > 0.
\]
Here the identities $\sum_{j=1}^\infty j^{-2} = \frac{\pi^2}6$ and $\Rg T_\infty\rho_* = p_{\infty,*} + \frac{2\si}{R_*}$ are used. 
This yields $Q(\xi)>0$, which contradicts to the fact that $Q(\tau)=0$. Thus, we have $\xi<0$ for $\eta=0$.

Now we search for a negative upper bound for $\xi$ when $\eta=0$.
Suppose 
\EQ{\label{eq-xi-assumption-3}
\xi \ge -\th_0\pi^2\bar\ka,
}
where $0<\th_0<1$ to be chosen. Suppose further that 
\EQ{\label{eq-xi-assumption-4}
\xi > \frac{-\frac{4\mu_l}{R_*} + \sqrt{\De}}{2\rho_lR_*}\quad \text{ if }\De:= \bke{ \frac{4\mu_l}{R_*} }^2 - 4\rho_lR_* \bke{ \frac{2p_{\infty,*}}{R_*} + \frac{4\si}{R_*^2}} >0
}
such that 
\EQ{\label{eq-xi-assumption-4-consequence}
\rho_lR_*\xi^2 + \frac{4\mu_l}{R_*}\, \xi - \frac{2\si}{R_*^2} \ge -\frac{2p_{\infty,*}}{R_*} - \frac{6\si}{R_*^2}.
} 
Note that  the inequality \eqref{eq-xi-assumption-4-consequence} always holds when $\De\le0$.
Then
\EQN{
0 = Q(\xi) &= \frac1{\Rg T_\infty} \bke{\frac{4\pi}{3\ga} + \frac{8(\ga-1)}{\pi\ga} \sum_{j=1}^\infty \frac{\pi^2\bar\ka}{\pi^2\bar\ka j^2 + \xi} } \bke{\rho_lR_*\xi^2 + \frac{4\mu_l}{R_*}\, \xi - \frac{2\si}{R_*^2}} + 4\pi\,\frac{\rho_*}{R_*}\\
&> - \frac1{\Rg T_\infty} \bke{\frac{4\pi}{3\ga} + \frac{8(\ga-1)}{\pi\ga(1-\th_0)} \sum_{j=1}^\infty \frac{\pi^2\bar\ka}{\pi^2\bar\ka j^2} } \bke{ \frac{2p_{\infty,*}}{R_*} + \frac{6\si}{R_*^2}} + 4\pi\,\frac{\rho_*}{R_*}\\
&= - \frac1{\Rg T_\infty} \bke{\frac{4\pi}{3\ga} + \frac{4\pi}3 \bke{1-\frac1\ga} \frac1{1-\th_0} } \bke{ \frac{2p_{\infty,*}}{R_*} +  \frac{6\si}{R_*^2}} + 4\pi\,\frac{\rho_*}{R_*}
}
So
\[
4\pi\Rg T_\infty\, \frac{\rho_*}{R_*} < \bke{\frac{4\pi}{3\ga} + \frac{4\pi}3 \bke{1-\frac1\ga} \frac1{1-\th_0} } \bke{\frac{2p_{\infty,*}}{R_*} + \frac{6\si}{R_*^2} },
\]
or, equivalently,
\[
\th_0 > 1 - \dfrac{1-\dfrac1\ga }{\dfrac{3\Rg T_\infty\, \frac{\rho_*}{R_*}}{\frac{2p_{\infty,*}}{R_*} + \frac{6\si}{R_*^2}} - \dfrac1{\ga} } 
= 1 - \dfrac{1-\dfrac1\ga }{ \frac{3p_{\infty,*}R_*+6\si}{2p_{\infty,*}R_*+6\si} - \dfrac1{\ga} },
\]
where $\Rg T_\infty \rho_* = p_{\infty,*} + \frac{2\si}{R_*}$ has been used in the last equation.
We then choose 
\[
\th_0 = 1 - \dfrac{1-\dfrac1\ga }{ \dfrac{3p_{\infty,*}R_*+6\si}{2p_{\infty,*}R_*+6\si} - \dfrac1{\ga} } \in (0,1)
\]
to reach a contradiction to \eqref{eq-xi-assumption-3} and \eqref{eq-xi-assumption-4}. 
Hence we derive for $\eta=0$
\EQ{\label{eq-xi-upper-bound-eta-zero}
\xi < 
\begin{cases}
- \min\bket{ \bke{ 1 - \dfrac{1-\dfrac1\ga }{ \dfrac{3p_{\infty,*}R_*+6\si}{2p_{\infty,*}R_*+6\si} - \dfrac1{\ga} } }\pi^2\bar\ka,\, \dfrac{ \dfrac{4\mu_l}{R_*} - \sqrt{\De}}{2\rho_lR_*} }&\ \text{ if }\De>0,\\
- \bke{ 1 - \dfrac{1-\dfrac1\ga }{ \dfrac{3p_{\infty,*}R_*+6\si}{2p_{\infty,*}R_*+6\si} - \dfrac1{\ga} } }\pi^2\bar\ka&\ \text{ if }\De\le0.
\end{cases}
}
Combining the upper bounds \eqref{eq-xi-upper-bound-eta-small}, \eqref{eq-xi-upper-bound-eta-large}, and \eqref{eq-xi-upper-bound-eta-zero} for different cases of $\eta$ and using the identity $\Rg T_\infty \rho_* = p_{\infty,*} + \frac{2\si}{R_*}$, the lemma follows.

\end{proof}

\section{Rate of convergence of slow solutions approaching to center manifold for a class of fully nonlinear autonomous systems}\label{sec-center-manifold}

As mentioned in the paragraph below Theorem \ref{thm-nonlinear-exp-decay}, there are several obstacles preventing us from direct applying center manifold theorem to prove the exponential decay in nonlinear bubble oscillations.
One of which is the quasilinear character of the problem \eqref{eq-carr-6.3.1-simple-form}.
For this purpose, we develop in this appendix a geometric theory for a class of fully nonlinear autonomous systems which covers the quasilinear system \eqref{eq-carr-6.3.1-simple-form}.

We study a larger class of fully nonlinear autonomous systems of the form $\dot{\bf w} = \mathcal L{\bf w} + \mathcal N({\bf w}, \dot{\bf w})$ that includes the quasilinear autonomous system \eqref{eq-carr-6.3.1-simple-form} of our concern.
Assuming that the solution of such equation converges toward a given center manifold and that the time derivative of the solution is sufficiently small for all time, we prove that the convergence rate is exponential.
The proof is an adaptation of \cite[Sections 2.4 and 6.3]{Carr-book1981}, where the existence and stability of center manifold for semilinear equations are established.

\bigskip\noindent{\bf Setup of the fully nonlinear autonomous system, assumptions on the solution, and the center manifold.}
Let $Z$ be a Banach space with norm $\norm{\,\cdot\,}$.
We consider the evolution equation
\EQ{\label{eq-carr-6.3.1-general}
\dot{\bf w} = \mathcal L{\bf w} + \mathcal N({\bf w}, \dot{\bf w}),\quad
{\bf w}(0) \in Z,
}
where $\mathcal N({\bf w},{\bf p}):Z\times Z\to Z$ has a uniformly continuous second derivative with $\mathcal N({\bf 0},{\bf p}) = {\bf 0}$ and $\pd_{({\bf w}, {\bf p})}\mathcal N({\bf 0}, {\bf 0}) = {\bf O}$.

Assume

(i) $Z = X\,\oplus\, Y$ where $X$ is finite dimensional and $Y$ is closed.

(ii) $X$ is $\mathcal L$-invariant and that if $\mathcal A:= \mathcal L|_X$, then the real parts of the eigenvalues of $\mathcal A$ are all zeros.

(iii) $Y$ is $e^{\mathcal L t}$-invariant. Let $Q_1$ be a projection on $X$ (not necessarily along $Y$) and $Q_2 := I - Q_1$. For some positive constants $b, c$,
\EQ{\label{eq-carr-6.3.2}
\norm{e^{\mathcal L t} Q_2 } \le ce^{-bt},\quad t\ge0.
}

Let ${\bf w}$ be a solution of \eqref{eq-carr-6.3.1-general}.
Decompose ${\bf w}$ into ${\bf w} = {\bf x} + {\bf y}$ where ${\bf x} = Q_1{\bf w}$ and ${\bf y} = Q_2{\bf w}$.
Let $\mathcal B = Q_2\mathcal L$.
Then equation \eqref{eq-carr-6.3.1} can be written as 
\EQ{\label{eq-carr-6.3.4-original}
\dot{\bf x} &= \mathcal A{\bf x} + f({\bf x}, {\bf y}, \dot{\bf x}, \dot{\bf y}),\\
\dot{\bf y} &= \mathcal B{\bf y} + g({\bf x}, {\bf y}, \dot{\bf x}, \dot{\bf y}),
}
where
\[
f({\bf x}, {\bf y}, \dot{\bf x}, \dot{\bf y}) = Q_1\mathcal N({\bf x}+{\bf y}, \dot{\bf x}+\dot{\bf y}),\quad
g({\bf x}, {\bf y}, \dot{\bf x}, \dot{\bf y}) = Q_2\mathcal N({\bf x}+{\bf y}, \dot{\bf x}+\dot{\bf y}).
\]

A curve ${\bf y} = h({\bf x})$, defined for $|{\bf x}|$ small, is said to be an \emph{invariant manifold} for \eqref{eq-carr-6.3.4-original} if the solution $({\bf x}(t),{\bf y}(t))$ of \eqref{eq-carr-6.3.4-original} through $({\bf x}(0),h({\bf x}(0)))$ satisfies ${\bf y}(t) = h({\bf x}(t))$.
A \emph{center manifold} for \eqref{eq-carr-6.3.4-original} is an invariant manifold that is tangent to $X$ space at the origin.

By the assumption on the nonlinearity $\mathcal N({\bf w},\dot{\bf w})$, there exists a continuous function $k(\ve)$ with $k(0)=0$ such that 
\EQ{\label{eq-carr-2.3.5}
\norm{ f({\bf x}, {\bf y}, \dot{\bf x}, \dot{\bf y}) } + \norm{ g({\bf x}, {\bf y}, \dot{\bf x}, \dot{\bf y}) } &\le \ve k(\ve),\\
\norm{ f({\bf x}, {\bf y}, \dot{\bf x}, \dot{\bf y}) - f({\bf x}', {\bf y}', \dot{\bf x}', \dot{\bf y}') } &\le k(\ve) \bkt{ \norm{{\bf x} - {\bf x}'} + \norm{{\bf y} - {\bf y}'} + \norm{\dot{\bf x} - \dot{\bf x}'} + \norm{\dot{\bf y} - \dot{\bf y}'} },\\
\norm{ g({\bf x}, {\bf y}, \dot{\bf x}, \dot{\bf y}) - g({\bf x}', {\bf y}', \dot{\bf x}', \dot{\bf y}') } &\le k(\ve) \bkt{ \norm{{\bf x} - {\bf x}'} + \norm{{\bf y} - {\bf y}'} + \norm{\dot{\bf x} - \dot{\bf x}'} + \norm{\dot{\bf y} - \dot{\bf y}'} },
}
for all ${\bf x}, {\bf x}'\in X$, ${\bf y}, {\bf y}'\in Y$ and all $\dot{\bf x}, \dot{\bf x}', \dot{\bf y}, \dot{\bf y}' \in Z$ with $\norm{({\bf x}, {\bf y})}, \norm{(\dot{\bf x}, \dot{\bf y})}, \norm{({\bf x}',{\bf y}')}, \norm{(\dot{\bf x}', \dot{\bf y}')} < \ve$.

Let $\mathcal M$ be a center manifold for \eqref{eq-carr-6.3.4-original} given by ${\bf y} = h({\bf x})$.
If we substitute ${\bf y}(t)=h({\bf x}(t))$ into \eqref{eq-carr-6.3.4-original} we obtain
\EQ{
h'({\bf x}) \bkt{\mathcal A{\bf x} + f({\bf x}, h({\bf x}), \dot{\bf x}, h'({\bf x})\dot{\bf x})}
= \mathcal Bh({\bf x}) + g({\bf x}, h({\bf x)}, \dot{\bf x}, h'({\bf x})\dot{\bf x}).
}
The equation on the center manifold is given by
\EQ{\label{eq-carr-6.3.5}
\dot{\bf u} = \mathcal A{\bf u} + f({\bf u}, h({\bf u}), \dot{\bf u}, h'({\bf u})\dot{\bf u}).
}

We assume that ${\bf w}(t)$ converges to some element in $\mathcal M$, as $t\to\infty$, and that $\sup_{t\ge0} \norm{ \dot{\bf w}(t) }$ is sufficiently small.

\bigskip\noindent{\bf Rate of convergence to the center manifold.}
The following lemma describes that the trajectory shadows the center manifold
and corresponds to \cite[Lemma 2.4.1]{Carr-book1981}.

\begin{lemma}\label{lem-carr-2.4.1}
Let $({\bf x}(t), {\bf y}(t))$ be a solution of \eqref{eq-carr-6.3.4-original} with $\norm{({\bf x}(0), {\bf y}(0))}$ and $\norm{(\dot{\bf x}(t), \dot{\bf y}(t))}$, for all $t\ge0$, sufficiently small.
Then there exist positive $C_1$ and $\be_1$ such that 
\[
\norm{{\bf y}(t) - h({\bf x}(t))} \le C_1 e^{-\be_1 t} \norm{{\bf y}(0) - h({\bf x}(0))}
\]
for all $t\ge0$.
\end{lemma}
\begin{proof}
Let $({\bf x}(t),{\bf y}(t))$ be a solution of \eqref{eq-carr-6.3.4-original} with $({\bf x}(0),{\bf y}(0))$ sufficiently small.
Let ${\bf z}(t) = {\bf y}(t) - h({\bf x}(t))$, then
\EQ{\label{eq-carr-2.4.2}
\dot{\bf z} = \mathcal B{\bf z} + \mathcal R({\bf x},{\bf z},\dot{\bf x},\dot{\bf z})
}
where
\EQ{\label{eq-def-R}
\mathcal R({\bf x},{\bf z},\dot{\bf x},\dot{\bf z}) &= h'({\bf x}) \bkt{f({\bf x}, h({\bf x}), \dot{\bf x}, h'({\bf x})\dot{\bf x}) - f({\bf x},{\bf z}+h({\bf x}), \dot{\bf x}, \dot{\bf z} + h'({\bf x})\dot{\bf x})}\\
&\quad + g({\bf x},{\bf z}+h({\bf x}), \dot{\bf x}, \dot{\bf z} + h'({\bf x})\dot{\bf x}) - g({\bf x}, h({\bf x}), \dot{\bf x}, h'({\bf x})\dot{\bf x}).
}
Using the hypotheses of $f$ and $g$ and the bounds on $h$, 
\EQ{\label{eq-R-est}
\norm{ \mathcal R({\bf x},{\bf z},\dot{\bf x},\dot{\bf z}) }
&\le \norm{ h'({\bf x}) } \Big( \norm{ f({\bf x}, h({\bf x}), \dot{\bf x}, h'({\bf x})\dot{\bf x}) - f({\bf x},{\bf z}+h({\bf x}), \dot{\bf x}, h'({\bf x})\dot{\bf x}) } \\
&\qquad\qquad\quad + \norm{ f({\bf x}, {\bf z} + h({\bf x}), \dot{\bf x}, h'({\bf x})\dot{\bf x}) - f({\bf x},{\bf z}+h({\bf x}), \dot{\bf x}, \dot{\bf z} + h'({\bf x})\dot{\bf x}) } \Big)\\
&\quad + \norm{ g({\bf x},{\bf z}+h({\bf x}), \dot{\bf x}, \dot{\bf z} + h'({\bf x})\dot{\bf x}) - g({\bf x}, h({\bf x}), \dot{\bf x}, \dot{\bf z} + h'({\bf x})\dot{\bf x}) }\\
&\quad + \norm{ g({\bf x}, h({\bf x}), \dot{\bf x}, \dot{\bf z} + h'({\bf x})\dot{\bf x}) - g({\bf x}, h({\bf x}), \dot{\bf x}, h'({\bf x})\dot{\bf x}) }\\
&\le \de(\ve) \bkt{ \norm{\bf z} + \norm{\dot{\bf z}} },
}
if $\norm{\bf z}, \norm{\dot{\bf z}} < \ve$, 
for some continuous function $\de(\ve)$ with $\de(0)=0$.
Using \eqref{eq-carr-6.3.2} we obtain, from \eqref{eq-carr-2.4.2},
\EQ{\label{eq-z-norm-ineq}
\norm{ {\bf z}(t) } \le ce^{-bt} \norm{{\bf z}(0)} + c\de(\ve) \int_0^t e^{-b(t-s)} \bkt{ \norm{\bf z(s)} + \norm{\dot{\bf z}(s)}  } ds.
}
Using \eqref{eq-R-est} in \eqref{eq-carr-2.4.2} one has
\[
\norm{\dot{\bf z}} \le \norm{\mathcal B}\norm{\bf z} + \norm{\mathcal R({\bf x},{\bf z},\dot{\bf x},\dot{\bf z})}
\le \norm{\mathcal B}\norm{\bf z} + \de(\ve)\bkt{\norm{\bf z} + \norm{\dot{\bf z}} },
\]
and so
\EQ{\label{eq-bound-dotz-by-z}
\norm{\dot{\bf z}} \le C_0 \norm{\bf z},\qquad
C_0 = \bke{1-\de(\ve)}^{-1} \bke{\norm{\mathcal B} + \de(\ve) }.
}
Therefore, \eqref{eq-z-norm-ineq} yields 
\EQ{
e^{bt} \norm{ {\bf z}(t) } 
&\le c \norm{{\bf z}(0)} + c\de(\ve) \bke{1 + C_0 } \int_0^t e^{bs} \norm{\bf z(s)} ds
}
By Gronwall's lemma,
\[
e^{bt} \norm{ {\bf z}(t) }  \le c \norm{{\bf z}(0)} e^{c\de(\ve)  \bke{1 + C_0 } t }.
\]
The lemma follows.
\end{proof}

\begin{proposition}\label{prop-stab-center-mfd}
Suppose that the zero solution of \eqref{eq-carr-6.3.5} is Lyapunov stable.
Let $({\bf x}(t), {\bf y}(t))$ be a solution of \eqref{eq-carr-6.3.4-original}.
There exists $\ve>0$ such that if $\norm{({\bf x}(0), {\bf y}(0))} < \ve$ and
if $\norm{(\dot{\bf x}(t), \dot{\bf y}(t))} < \ve$ for all $t\ge0$,
then there exists a solution ${\bf u}(t)$ of \eqref{eq-carr-6.3.5} such that as $t\to\infty$,
\EQ{\label{eq-carr-2.4.5}
{\bf x}(t) &= {\bf u}(t) + O(e^{-b_1 t}),\\
{\bf y}(t) &= h({\bf u}(t)) + O(e^{-b_1 t}),
}
where $b_1 = \min(b, \be_1)$, $b$ and $\be_1$ are given in the assumption \eqref{eq-carr-6.3.2} and in Lemma \ref{lem-carr-2.4.1}, respectively.
\end{proposition}

\begin{proof}
The proof is based on that of \cite[Theorem 2.4.2]{Carr-book1981}.
Let $({\bf x}(t), {\bf y}(t))$ be a solution of \eqref{eq-carr-6.3.4-original}. 
Since the zero solution of \eqref{eq-carr-6.3.5} is Lyapunov stable, solutions ${\bf u}(t)$ of \eqref{eq-carr-6.3.5} are Lyapunov stable if ${\bf u}(0)$ is sufficiently small.
Let ${\bf u}(t)$ be a solution of \eqref{eq-carr-6.3.5} with ${\bf u}(0)$ sufficiently small.
Let ${\bf z}(t) = {\bf y}(t) - h({\bf x}(t))$, $\boldsymbol{\phi}(t) = {\bf x}(t) - {\bf u}(t)$.
Then
\begin{subequations}
\begin{align}
\dot{\bf z} &= \mathcal B{\bf z} + \mathcal R(\boldsymbol{\phi} + {\bf u},{\bf z},\dot{\boldsymbol{\phi}} + \dot{\bf u},\dot{\bf z}),\label{eq-carr-2.4.7}\\
\dot{\boldsymbol{\phi}} &= \mathcal A\boldsymbol{\phi} + \mathcal V(\boldsymbol{\phi},{\bf z},\dot{\boldsymbol{\phi}},\dot{\bf z}),\label{eq-carr-2.4.8}
\end{align}
\end{subequations}
where $\mathcal R$ is defined in \eqref{eq-def-R} and
\EQN{
\mathcal V(\boldsymbol{\phi},{\bf z},\dot{\boldsymbol{\phi}},\dot{\bf z})
&= f({\bf u}+\boldsymbol{\phi}, {\bf z} + h({\bf u}+\boldsymbol{\phi}), \dot{\bf u} + \dot{\boldsymbol{\phi}}, \dot{\bf z} + h'({\bf u}+\boldsymbol{\phi}) (\dot{\bf u} + \dot{\boldsymbol{\phi}})) - f({\bf u}, h({\bf u}), \dot{\bf u}, h'({\bf u})\dot{\bf u}).
}

We now formulate \eqref{eq-carr-2.4.7}--\eqref{eq-carr-2.4.8} as a fixed point problem.
Let $\mathscr X$ be the set of continuous differentiable functions $\boldsymbol{\phi}:[0,\infty)\to X$ of \eqref{eq-carr-2.4.8} with $\norm{\boldsymbol{\phi}(t) e^{at}}\le 1$ and $\norm{\dot{\boldsymbol{\phi}}(t) e^{at}} \le a$ for all $t\ge0$, where $a=b/2$ in which $b$ is defined in \eqref{eq-carr-6.3.2}.
We define the norm $\norm{\boldsymbol{\phi}}_{\mathscr X} = \sup\{\norm{\boldsymbol{\phi}(t) e^{at}} + \norm{\dot{\boldsymbol{\phi}}(t) e^{at}} : t\ge0\}$.
By the assumption on the operator $\mathcal A$, we can decompose $\mathcal A$ into $\mathcal A = \mathcal A_1 + \mathcal A_2$ where
\EQ{\label{eq-carr-2.4.3-4}
\norm{ e^{\mathcal A_1 t}{\bf x} } = \norm{ {\bf x} },\qquad
\norm{ \mathcal A_2{\bf x}} \le (b/4) \norm{ {\bf x} },
}
where $b$ is defined by \eqref{eq-carr-6.3.2}.
Then \eqref{eq-carr-2.4.8} can be written as
\[
\dot{\boldsymbol{\phi}} = \mathcal A_1\boldsymbol{\phi} + \bkt{ \mathcal A_2\boldsymbol{\phi}  + \mathcal V(\boldsymbol{\phi},{\bf z},\dot{\boldsymbol{\phi}},\dot{\bf z}) },
\]
and $\boldsymbol{\phi}(\infty) = 0$ for $\boldsymbol{\phi}\in\mathscr X$.
Let ${\bf z}(t)$ be a given solution of \eqref{eq-carr-2.4.7}.
By Duhamel's formula, a solution $\boldsymbol{\phi}\in\mathscr X$ of \eqref{eq-carr-2.4.8} must satisfy
\[
\boldsymbol{\phi}(t) = -\int_t^\infty e^{\mathcal A_1(t-s)} \bkt{ \mathcal A_2\boldsymbol{\phi}(s)  + \mathcal V(\boldsymbol{\phi}(s),{\bf z}(s),\dot{\boldsymbol{\phi}}(s),\dot{\bf z}(s)) } ds.
\]
Thus, a solution $\boldsymbol{\phi}\in\mathscr X$ of \eqref{eq-carr-2.4.8} is a fixed point of the mapping $T$ that defined by
\EQ{\label{eq-carr-2.4.9}
(T\boldsymbol{\phi})(t) = -\int_t^\infty e^{\mathcal A_1(t-s)} \bkt{ \mathcal A_2\boldsymbol{\phi}(s)  + \mathcal V(\boldsymbol{\phi}(s),{\bf z}(s),\dot{\boldsymbol{\phi}}(s),\dot{\bf z}(s)) } ds.
}

Using the bounds on $f$, $g$, $h$, and the fact that $\mathcal R({\bf x},{\bf 0},\dot{\bf x},{\bf 0}) = {\bf 0}$ and $\mathcal V({\bf 0}, {\bf 0}, {\bf 0}, {\bf 0}) = {\bf 0}$, there is a continuous function $k(\ve)$ with $k(0)=0$ such that if $\boldsymbol{\phi}_1,\boldsymbol{\phi}_2\in X$, ${\bf z}_1, {\bf z}_2\in Y$, and $\dot{\boldsymbol{\phi}}_1,\dot{\boldsymbol{\phi}}_2, \dot{\bf z}_1, \dot{\bf z}_2\in Z$ with $\norm{(\boldsymbol{\phi}_i, {\bf z}_i)}$, $\norm{(\dot{\boldsymbol{\phi}}_i, \dot{\bf z}_i)} < \ve$, $i=1,2$,
then
\EQ{\label{eq-carr-2.4.10}
&\norm{ \mathcal R(\boldsymbol{\phi}_1,{\bf z}_1,\dot{\boldsymbol{\phi}}_1,\dot{\bf z}_1) - \mathcal R(\boldsymbol{\phi}_2,{\bf z}_2,\dot{\boldsymbol{\phi}}_2,\dot{\bf z}_2) }\\
&\qquad\qquad\qquad\qquad\qquad\quad \le k(\ve) \bkt{ \norm{({\bf z}_1, \dot{\bf z}_1)} \bke{ \norm{ \boldsymbol{\phi}_1 - \boldsymbol{\phi}_2 } + \norm{ \dot{\boldsymbol{\phi}}_1 - \dot{\boldsymbol{\phi}}_2 } } + \norm{ {\bf z}_1 - {\bf z}_2 } + \norm{ \dot{\bf z}_1 - \dot{\bf z}_2 } },\\
&\norm{ \mathcal V(\boldsymbol{\phi}_1,{\bf z}_1,\dot{\boldsymbol{\phi}}_1,\dot{\bf z}_1) - \mathcal V(\boldsymbol{\phi}_2,{\bf z}_2,\dot{\boldsymbol{\phi}}_2,\dot{\bf z}_2) } \le k(\ve) \bkt{ \norm{ {\bf z}_1 - {\bf z}_2 } + \norm{ \dot{\bf z}_1 - \dot{\bf z}_2 } + \norm{ \boldsymbol{\phi}_1 - \boldsymbol{\phi}_2 } + \norm{ \dot{\boldsymbol{\phi}}_1 - \dot{\boldsymbol{\phi}}_2 } }.
}
By the same argument as in the proof of Lemma \ref{lem-carr-2.4.1}, one has
\EQ{\label{eq-carr-2.4.11}
\norm{{\bf z}(t)} \le C_1 \norm{{\bf z}(0)} e^{-\be_1 t},
}
where
\[
\be_1 = b - ck(\ve)  \bke{1 + C_0 },\qquad
C_0 = \bke{1-k(\ve)}^{-1} \bke{\norm{\mathcal B} + k(\ve)}.
\]
Using \eqref{eq-carr-2.4.3-4}, we obtain, from \eqref{eq-carr-2.4.9}, that
\EQ{\label{eq-bound-Tphi}
\norm{ T\boldsymbol{\phi}(t) }
&\le \int_t^\infty \frac{b}4 |\boldsymbol{\phi}(s)|\, ds + \int_t^\infty \norm{ \mathcal V(\boldsymbol{\phi}(s),{\bf z}(s),\dot{\boldsymbol{\phi}}(s),\dot{\bf z}(s)) } ds\\
&= \frac{a}4 \int_t^\infty \norm{ \boldsymbol{\phi}(s) } ds + \int_t^\infty \norm{ \mathcal V(\boldsymbol{\phi}(s),{\bf z}(s),\dot{\boldsymbol{\phi}}(s),\dot{\bf z}(s)) - \mathcal V({\bf 0}, {\bf 0}, {\bf 0}, {\bf 0}) } ds\\
&\le \frac{e^{-at}}2 + k(\ve) \int_t^\infty \bkt{ \norm{ \boldsymbol{\phi}(s) } + \norm{ {\bf z}(s) } + \norm{ \dot{\boldsymbol{\phi}}(s) } + \norm{ \dot{\bf z}(s)} } ds,
}
where we've used $\mathcal V({\bf 0}, {\bf 0}, {\bf 0}, {\bf 0}) = {\bf 0}$ in the second equation and \eqref{eq-carr-2.4.10} in the last inequality. 
In view of \eqref{eq-bound-dotz-by-z} in the proof of Lemma \ref{lem-carr-2.4.1} and \eqref{eq-carr-2.4.11}, 
\EQ{\label{eq-bound-z-dotz}
\norm{ {\bf z}(s) } + \norm{ \dot{\bf z}(s)} \le \bke{1 + C_0 } \norm{ {\bf z}(s) }
\le \bke{1 + C_0 } C_1 \norm{{\bf z}(0)} e^{-\be_1 s}
=: C_2\, e^{-\be_1 s}.
}
Together with the hypothesis $\boldsymbol{\phi}\in \mathscr X$, we derive
\[
\norm{ T\boldsymbol{\phi}(t) } \le \frac{e^{-at}}2 + k(\ve) \int_t^\infty \bkt{ (1+a) e^{-as} + C_2 e^{-\be_1 s} } ds
\le e^{-at}
\]
for $\ve$ sufficiently small such that $\be_1 = b - ck(\ve)  \bke{1 + C_0 }\ge b/2 = a$ and $k(\ve)\le\min\{(1+a)^{-1}, C_2^{-1}\}/ 2$.
To estimate $\frac{d}{dt} T\boldsymbol{\phi}$, we compute
\EQN{
\norm{ \frac{d}{dt} T\boldsymbol{\phi}(t) } 
&= \norm{ A_2\boldsymbol{\phi}(t)  + V(\boldsymbol{\phi}(t),{\bf z}(t),\dot{\boldsymbol{\phi}}(t),\dot{\bf z}(t)) }\\
&\le \frac{a}2 \norm{ \boldsymbol{\phi}(t) } + k(\ve) \bke{ \norm{ \boldsymbol{\phi}(t) } + \norm{ {\bf z}(t) } + \norm{ \dot{\boldsymbol{\phi}}(t) } + \norm{ \dot{\bf z}(t)} }\\
&\le \bke{\frac{a}2 + k(\ve) (1+a) } e^{-at} + k(\ve) C_2 e^{-\be_1t}
\le ae^{-at}
}
for $\ve$ sufficiently small.
This proves $T$ maps $\mathscr X$ into $\mathscr X$.

We now show that $T$ is a contraction on $\mathscr X$.
Let $\boldsymbol{\phi}_1, \boldsymbol{\phi}_2\in \mathscr X$ and let ${\bf z}_1, {\bf z}_2$ be the corresponding solutions of \eqref{eq-carr-2.4.7} with ${\bf z}_i(0) = {\bf z}_0$, $i=1,2$.
We first estimate ${\bf v}(t) = {\bf z}_1(t) - {\bf z}_2(t)$.
From \eqref{eq-carr-2.4.7} and \eqref{eq-carr-2.4.10}, 
\EQ{\label{eq-v-bound}
\norm{{\bf v}(t)} \le ck(\ve) \int_0^t e^{-b(t-s)} \bkt{ \norm{({\bf z}_1(s), \dot{\bf z}_1(s))} \bke{ \norm{ \boldsymbol{\phi}_1(s) - \boldsymbol{\phi}_2(s) } + \norm{ \dot{\boldsymbol{\phi}}_1(s) - \dot{\boldsymbol{\phi}}_2(s) } } + \norm{ {\bf v}(s) } + \norm{ \dot{\bf v}(s) } } ds.
}
Since 
\[
\dot{\bf v} = \mathcal B{\bf v} + \mathcal R(\boldsymbol{\phi}_1 + {\bf u},{\bf z}_1,\dot{\boldsymbol{\phi}}_1 + \dot{\bf u},\dot{\bf z}_1) - \mathcal R(\boldsymbol{\phi}_2 + {\bf u},{\bf z}_2,\dot{\boldsymbol{\phi}}_2 + \dot{\bf u},\dot{\bf z}_2),
\]
using \eqref{eq-carr-2.4.10}, we get
\[
\norm{ \dot{\bf v} } \le \norm{\mathcal B} \norm{ v } + k(\ve) \bkt{ \norm{({\bf z}_1, \dot{\bf z}_1)} \bke{ \norm{ \boldsymbol{\phi}_1 - \boldsymbol{\phi}_2 } + \norm{ \dot{\boldsymbol{\phi}}_1 - \dot{\boldsymbol{\phi}}_2 } } + \norm{ {\bf v} } + \norm{ \dot{\bf v} } }.
\]
So
\[
\norm{ \dot{\bf v} } \le (1 - k(\ve))^{-1} \bkt{(\norm{\mathcal B} + k(\ve)) \norm{\bf v} + k(\ve) \norm{({\bf z}_1, \dot{\bf z}_1)} \bke{ \norm{ \boldsymbol{\phi}_1 - \boldsymbol{\phi}_2 } + \norm{ \dot{\boldsymbol{\phi}}_1 - \dot{\boldsymbol{\phi}}_2 } } },
\]
implying
\[
\norm{ \dot{\bf v}(s) } \le C_3 \norm{{\bf v}(s) }+ k_1(\ve) \norm{({\bf z}_1, \dot{\bf z}_1)} \norm{\boldsymbol{\phi}_ 1 - \boldsymbol{\phi}_2}_{\mathscr X} e^{-as},\qquad k_1(0) = 0,
\]
for some $C_3>0$. 
Together with $\norm{({\bf z}_1(s), \dot{\bf z}_1(s))} \le C_2 e^{-\be_1s}$, which is followed by \eqref{eq-bound-z-dotz}, \eqref{eq-v-bound} implies
\EQN{
\norm{{\bf v}(t)}
&\le ck(\ve) \int_0^t e^{-b(t-s)} \bkt{ \bke{1 + k_1(\ve) } C_2 e^{-\be_1s} \norm{\boldsymbol{\phi}_ 1 - \boldsymbol{\phi}_2}_{\mathscr X} e^{-as} + (1+C_3) \norm{{\bf v}(s)}} ds\\
&\le C_4 k(\ve) \norm{\boldsymbol{\phi}_ 1 - \boldsymbol{\phi}_2}_{\mathscr X} e^{-bs} + c k(\ve)(1+C_3) \int_0^t e^{-b(t-s)}  \norm{{\bf v}(s)} ds
}
for $\ve$ sufficiently small, where $C_4>0$ is a constant.
By Gronwall's lemma,
\EQ{\label{eq-carr-2.4.12}
\norm{{\bf v}(t)} \le C_4 k(\ve) \norm{\boldsymbol{\phi}_ 1 - \boldsymbol{\phi}_2}_{\mathscr X} e^{-\be_2 t},\qquad
\be_2 = b - c k(\ve)(1+C_3).
}
Using \eqref{eq-carr-2.4.3-4} and \eqref{eq-carr-2.4.12}, 
\EQN{
&\norm{T\boldsymbol{\phi}_1(t) - T\boldsymbol{\phi}_2(t) } \\
&\ \ \le \int_t^\infty \bkt{\frac{a}2 \norm{ \boldsymbol{\phi}_1(s) - \boldsymbol{\phi}_2(s) } + \norm{\mathcal R(\boldsymbol{\phi}_1(s) ,{\bf z}_1(s), \dot{\boldsymbol{\phi}}_1(s), \dot{\bf z}_1(s)) - \mathcal R(\boldsymbol{\phi}_2(s) ,{\bf z}_2(s), \dot{\boldsymbol{\phi}}_2(s), \dot{\bf z}_2(s) )}}\\
&\ \ \le \frac{a}2 \int_t^\infty \norm{\boldsymbol{\phi}_ 1 - \boldsymbol{\phi}_2}_{\mathscr X} e^{-as}\, ds\\
&\ \ \quad + k(\ve) \int_t^\infty \bkt{ \norm{({\bf z}_1(s), \dot{\bf z}_1(s))} \bke{ \norm{ \boldsymbol{\phi}_1(s) - \boldsymbol{\phi}_2(s) } + \norm{ \dot{\boldsymbol{\phi}}_1(s) - \dot{\boldsymbol{\phi}}_2(s) } } + \norm{ {\bf v}(s) } + \norm{ \dot{\bf v}(s) } } ds\\
&\ \ \le \frac12 \norm{\boldsymbol{\phi}_ 1 - \boldsymbol{\phi}_2}_{\mathscr X}\\
&\ \ \quad + k(\ve) \int_t^\infty \bkt{ (1+k_1(\ve))C_2 e^{-\be_1s} \norm{\boldsymbol{\phi}_ 1 - \boldsymbol{\phi}_2}_{\mathscr X} e^{-as} + (1+C_3)C_4 k(\ve) \norm{\boldsymbol{\phi}_ 1 - \boldsymbol{\phi}_2}_{\mathscr X} e^{-\be_2 s} } ds\\
&\ \ \le \al\norm{\boldsymbol{\phi}_ 1 - \boldsymbol{\phi}_2}_{\mathscr X},\qquad \al<1,
}
for $\al$ sufficiently small.
This shows that $T$ is a contraction and, hence, has a unique fixed point.

Note that $T = T_{\bf z}$ and the above analysis proves that $T_{\bf z}$ has a unique fixed point in $\mathscr X$ provided ${\bf z}$ and $\dot{\bf z}$ are sufficiently small. 
Denote $\boldsymbol{\phi}\in\mathscr X$ the fixed point of the contraction $T_{\bf z}$, where ${\bf z} = {\bf y} - h({\bf x})$.
Let ${\bf u}(0) = {\bf x}(0) - \boldsymbol{\phi}(0)$. 
Then let ${\bf u}(t)$ be the solution of \eqref{eq-carr-6.3.5} evolving from ${\bf u}(0)$.
Then $({\bf x}(t), {\bf y}(t))$ can be decomposed as in ${\bf x}(t) = {\bf u}(t) + \boldsymbol{\phi}(t)$ and ${\bf y} = h({\bf x}(t)) + {\bf z}(t)$.
In view of the fact that $\boldsymbol{\phi}\in\mathscr X$ and Lemma \ref{lem-carr-2.4.1}, the asymptotic limits in \eqref{eq-carr-2.4.5} follow, which completing the proof of Proposition \ref{prop-stab-center-mfd}.
\end{proof}

\section{Asymptotic expansion of the local center manifold}

In this appendix,
we check the expression of local center manifold in Lemma \ref{lem-mfld-equil-ceneter-mfld} by asymptotic expansion.

If we substitute ${\bf y} = h({\bf x})$ into \eqref{eq-carr-6.3.4-y}, we see that the center manifold can be obtained by solving 
\[
h'({\bf x}) \bkt{ Q_1 \mathcal N({\bf x} + h({\bf x}), \dot{\bf x} + h'({\bf x})\dot{\bf x})}
= \mathcal L h({\bf x}) + Q_2 \mathcal N({\bf x} + h({\bf x)}, \dot{\bf x} + h'({\bf x})\dot{\bf x}).
\]
We want to show the manifold of equilibria $\mathcal M_*$ is a center manifold by showing it satisfies the above equation.
On the manifold of equilibria $\mathcal M_*$, $\dot{\bf x} = 0$.
So we need to check
\EQ{\label{eq-center-mfd-equilib}
h'({\bf x}) \bkt{ Q_1 \mathcal N({\bf x} + h({\bf x}), {\bf 0})}
= \mathcal L h({\bf x}) + Q_2 \mathcal N({\bf x} + h({\bf x)}, {\bf 0}).
}
Note that
\[
\mathcal N({\bf x} + h({\bf x}), {\bf 0}) = \mathcal N^0({\bf x} + h({\bf x})) = 
 \begin{bmatrix}
0\\
0\\
-\frac{\Rg T_\infty}{\rho_lR_*}\, H^0({\bf x} + h({\bf x}))\\
0\\
0\\
\vdots
 \end{bmatrix}
 =
 \begin{bmatrix}
0\\
0\\
-\frac{2\si}{\rho_lR_*^3} \frac{(\al+R_{**}(\al)-R_*)^2}{\al+R_{**}(\al)}\\
0\\
0\\
\vdots
 \end{bmatrix}.
\]
So $Q_1\mathcal N({\bf x} + h({\bf x}), {\bf 0}) = {\bf 0}$ and $\mathcal N({\bf x} + h({\bf x}), {\bf 0}) = \mathcal N({\bf x} + h({\bf x}), {\bf 0})$.
It suffices to check $\mathcal L h({\bf x}) + Q_2 \mathcal N({\bf x} + h({\bf x)}, {\bf 0}) = {\bf 0}$ in  \eqref{eq-center-mfd-equilib}.
Taylor expand $h({\bf x}) = h(\al{\bf b})$ at $\al=0$, one gets
\[
h({\bf x}) 
= 
\begin{bmatrix}
\rho_{**}(\al) - \rho_*\\
R_{**}(\al) - R_*\\
0\\
0\\
0\\
\vdots
\end{bmatrix}
=
\begin{bmatrix}
\rho_{**,1} \al + \rho_{**,2} \al^2 + \rho_{**,3} \al^3 + \cdots\\
R_{**,1} \al + R_{**,2} \al^2 + R_{**,3} \al^3 + \cdots\\
0\\
0\\
0\\
\vdots
\end{bmatrix},
\]
where $\rho_{**,m} = (m!)^{-1} \rho_{**}^{(m)}(0)$ and $R_{**,m} = (m!)^{-1} R_{**}^{(m)}(0)$ in which $\rho_{**}^{(m)}(0)$ and $R_{**}^{(m)}(0)$ are the $m$-th derivatives of $\rho_{**}$ and $R_{**}$ at $\al=0$, respectively.
Then $\mathcal L h({\bf x}) + \mathcal N({\bf x} + h({\bf x)}, {\bf 0}) = {\bf 0}$ is equivalent to
\EQ{\label{eq-infinite-sum-taylor}
\sum_{m=1}^\infty \bke{ \frac{\Rg T_\infty}{\rho_lR_*} \rho_{**,m} + \frac{2\si}{\rho_lR_*^3} R_{**,m} } \al^m - \frac{2\si}{\rho_lR_*^3} \frac{\bke{\al+R_{**}(\al)-R_*}^2}{\al+R_{**}(\al)} = 0.
}

For $m=1$, we differentiate \eqref{eq-rho-R-double-star} with respect to $\al$, evaluate at $\al=0$ and use $R_{**}(0) = R_*$, we get $\Rg T_\infty \rho_{**}'(0) = -(2\si/R_*^2) R_{**}'(0)$.
So we have
\[
\frac{\Rg T_\infty}{\rho_lR_*} \rho_{**,1} + \frac{2\si}{\rho_lR_*^3} R_{**,1} = 0.
\]

For $m=2$, 
we first expand
\[
\frac{\bke{\al+R_{**}(\al)-R_*}^2}{\al+R_{**}(\al)} = ( 1+R_{**}'(0) )^2 \al^2 + \cdots.
\]
Then the coefficient of $\al^2$ in \eqref{eq-infinite-sum-taylor} is
\EQ{\label{eq-m=2}
\frac{\Rg T_\infty}{\rho_lR_*} \rho_{**,2} + \frac{2\si}{\rho_lR_*^3} R_{**,2} - \frac{2\si}{\rho_lR_*^4} (1+R_{**,1})^2 .
}
By differentiating \eqref{eq-rho-R-double-star} with respect to $\al$ twice, evaluating at $\al=0$ and using $R_{**}(0) = R_*$, 
we have $\Rg T_\infty \rho_{**}''(0) = (4\si/R_*^3) (R_{**}'(0))^2 - (2\si/R_*^2) R_{**}''(0)$.
So
\[
\frac{\Rg T_\infty}{\rho_lR_*} \rho_{**,2} + \frac{2\si}{\rho_lR_*^3} R_{**,2} = (2!)^{-1} \frac{4\si}{\rho_l R_*^4} (R_{**,1})^2.
\]
In order to have the term \eqref{eq-m=2} for $m=2$ vanishing, we require
\[
0 = \frac{2\si}{\rho_lR_*^4} R_{**,1}^2  - \frac{2\si}{\rho_lR_*^4} (1 + R_{**,1})^2,
\] 
whose solution is $R_{**,1} = -1/2$.

For $m=3$, we further expand, using $R_{**,1} = -1/2$, that
\[
\frac{\bke{\al+R_{**}(\al)-R_*}^2}{\al+R_{**}(\al)} = ( 1+R_{**,1} )^2 \al^2 + \bke{\frac3{R_*} R_{**}''(0) - \frac3{4R_*^2} } \al^3 + \cdots
\]
Then the coefficient of $\al^3$ in \eqref{eq-infinite-sum-taylor} is
\EQ{\label{eq-m=3}
\frac{\Rg T_\infty}{\rho_lR_*} \rho_{**,3} + \frac{2\si}{\rho_lR_*^3} R_{**,3} - \frac{2\si}{\rho_lR_*^3} \frac1{3!} \bke{ \frac3{R_*} R_{**}''(0) - \frac3{4R_*^2} }.
}
By differentiating \eqref{eq-rho-R-double-star} with respect to $\al$ three times, evaluating at $\al=0$ and using $R_{**}(0) = R_*$, 
we get $\Rg T_\infty \rho_{**}'''(0) = -(12\si/R_*^4)(R_{**}'(0))^3 + (12\si/R_*^3)R_{**}'(0)R_{**}''(0)  -(2\si/R_*^2) R_{**}'''(0)$.
So we have
\[
\frac{\Rg T_\infty}{\rho_lR_*} \rho_{**,3} = -\frac{2\si}{\rho_lR_*^5} R_{**,1}^3 + \frac{4\si}{\rho_lR_*^4} R_{**,1} R_{**,2} - \frac{2\si}{\rho_lR_*^3} R_{**,3}.
\]
In order to have the term \eqref{eq-m=3} for $m=3$ vanishing, we require
\[
0 = -\frac{2\si}{\rho_lR_*^5} R_{**,1}^3 + \frac{4\si}{\rho_lR_*^4} R_{**,1} R_{**,2} - \frac{2\si}{\rho_lR_*^3} R_{**,3} + \frac{2\si}{\rho_lR_*^3} R_{**,3} - \frac{\si}{\rho_lR_*^3} \bke{ \frac1{R_*} R_{**}''(0) - \frac1{4R_*^2} },
\]
where the third term cancels the fourth.
Using $R_{**,1}=-1/2$ and $R_{**}''(0) = 2R_{**,2}$,
\[
0 = \frac{\si}{4\rho_lR_*^5} - \frac{2\si}{\rho_lR_*^4} R_{**,2} - \frac{2\si}{\rho_lR_*^4} R_{**,2} + \frac{\si}{4\rho_lR_*^5}
\]
for which the solution is $R_{**,2} = 1/(8R_*)$.

Continuing this process, we obtain $R_{**,3} = 0$, $R_{**,4} = -1/(128R_*^3),\ldots$.
Thus,
\[
R_{**}(\al) = R_* - \frac12 \al + \frac1{8R_*} \al^2 -\frac1{128R_*^3} \al^4 + \cdots =  \frac{-\al+\sqrt{\al^2 + 4R_*^2}}2,
\]
which coincides with the center manifold expression in \eqref{eq-rho-R-double-star}.

\end{document}